\documentclass{article}
\usepackage[numbers,sort&compress]{natbib}
\usepackage{amsmath}
\usepackage{amssymb}
\usepackage{multirow}
\usepackage{amsfonts}
\usepackage{mathrsfs}
\usepackage{booktabs}
\usepackage{authblk}
\usepackage{algorithmicx}
\usepackage{algorithm}
\usepackage{booktabs}
\usepackage{algpseudocode}
\usepackage{comment}
\usepackage{amsthm}
\usepackage{latexsym}
\usepackage{graphicx}
\usepackage{rotating}
\usepackage{subfig}
\usepackage{color}
\usepackage{varioref}
\usepackage{longtable}
\allowdisplaybreaks[4]
\usepackage[colorlinks,linkcolor=blue]{hyperref}

\newtheorem{theorem}{Theorem}
\newtheorem{assumption}{Assumption}
\newtheorem{proposition}{Proposition}
\newtheorem{definition}{Definition}
\newtheorem{lemma}{Lemma}
\newtheorem{remark}{Remark}
\newtheorem{corollary}{Corollary}
\usepackage{geometry}
\begin{document}

\title{Inertial Accelerated  Stochastic Mirror Descent for  Large-Scale Generalized Tensor CP Decomposition}


\author{Zehui Liu\thanks{LMIB, School of Mathematical Sciences, Beihang University, Beijing 100191, China. Email: liuzehui@buaa.edu.cn}, \hskip 0.2cm 
Qingsong Wang\thanks{School of Mathematics and Computational Science, Xiangtan University, Xiangtan 411105, China. Email: nothing2wang@hotmail.com}, \hskip 0.2cm 
Chunfeng Cui\thanks{LMIB, School of Mathematical Sciences, Beihang University, Beijing 100191, China. Email: chunfengcui@buaa.edu.cn}, \hskip 0.2cm 
Yong Xia\thanks{LMIB, School of Mathematical Sciences, Beihang University, Beijing 100191, China. Email: yxia@buaa.edu.cn}
}

\maketitle

\begin{abstract} 
The majority of classic tensor CP decomposition models are designed for squared loss, employing Euclidean distance as a local proximal term. However, the Euclidean distance is unsuitable for the generalized loss function applicable to various types of real-world data, such as integer and binary data. Consequently, algorithms developed under the squared loss are not easily adaptable to handle these generalized losses, partially due to the lack of the gradient Lipschitz continuity. This paper considers the generalized tensor CP decomposition. We use the Bregman distance as the proximal term and propose an inertial accelerated  block randomized stochastic mirror descent algorithm (iTableSMD). 
Within a broader multi-block variance reduction and inertial acceleration framework, we demonstrate the sublinear convergence rate for the subsequential sequence produced by the iTableSMD algorithm. We further show that   iTableSMD requires at most $\mathcal{O}(\varepsilon^{-2})$ iterations in expectation to attain an $\varepsilon$-stationary point and establish the global convergence of the sequence.
Numerical experiments on real datasets demonstrate that our proposed algorithm is efficient and achieve better performance than the existing state-of-the-art methods. 
\end{abstract}

\begin{keywords}
Generalized tensor CP decomposition, Inertial acceleration, Stochastic mirror descent, Bregman divergence, Non-Lipschitz gradient continuity,  Variance reduction.

\end{keywords}

\maketitle
\section{Introduction}



A fundamental generic optimization model that encompasses a wide range of   multi-block models arising in various applications is the well-known composite minimization problem. It can be formally defined as:
\begin{equation}\label{Pure_problem}
\min _{\{x_t\}_{t=1}^{s}} \Phi \left({x}_1, \ldots, {x}_s\right) \equiv f\left({x}_1, \ldots, {x}_s\right)+\sum_{t=1}^s h_t\left({x}_t\right),
\end{equation}
where variable $x$ can be decomposed into $s$ blocks ${x}_1, \ldots, {x}_s$, $f$ is assumed to be a continuously differentiable nonconvex function over $x=\left({x}_1, \ldots, {x}_s\right)$, and can be  convex in the manner of a block ${x}_t$, while all the other blocks are fixed. It can also admit a finite-sum structure form $f(x)=\frac{1}{n}\sum_{i=1}^{n}f_{i}(x)$ with different indexes. 
The usual restrictive condition of the gradient Lipschitz continuity of $f_{i}$ is not required, and $h_t, t=1, \ldots, s$, are extended-value weakly convex functions, which is a structure-promoting regularizer that captures the prior information about ${x}_i$, such as column-wise orthogonality \cite{Krijnen2008}, Tikhonov regularization \cite{Paatero2000}, iterated Tikhonov regularization \cite{Navasca2008}, non-negativity \cite{Lim2009}.

The majority of classic models and algorithms developed for multi-block problem \eqref{Pure_problem} are least squares (LS) problems using the Euclidean distance-based fitting criterion \cite{Harshman1970, 1970Analysis}. Efficient methods in this line include the block coordinate descent (BCD) \cite{120887795}, alternating direction method of multipliers (ADMM) \cite{Han22}, and proximal alternating linearized minimization (PALM) \cite{BolteST14} algorithm. Then Pock and Sabach \cite{PockS16} 
introduced a 
inertial variant of PALM (iPALM). Huang et al. \cite{7484753} introduced a primal-dual algorithm AO-ADMM, which is a hybrid between alternating optimization   and  ADMM. There are also some other first-order type algorithms \cite{6544287, Comon2009TensorDA} and (quasi-)second-order methods \cite{Second2013, Order2013}. 
However, the Euclidean distance is unsuitable for measuring the proximity between various real-world data types, including nonnegative, integer, and binary data. Essentially, utilizing data geometry-aware divergences as fitting criteria has the potential to substantially improve both performance and robustness in real-world applications \cite{GCP2020, Kolda2019, WMLL2020, VMVNLL2020}. For instance, the proximity between two probability distributions is measured using an appropriate divergence, such as the generalized Kullback-Leibler (KL) divergence \cite{KLDiver2017, Kargas2019, FSCH2021, CTWWP2020} or the Itakura-Saito  divergence \cite{LP2015}. From a statistical perspective, numerous non-Euclidean divergences share a close connection with maximum likelihood estimators (MLEs) under reasonable data distribution assumptions. For example, the generalized KL divergence \cite{Kolda2012} and logistic loss can be obtained as MLEs for integer data with Poisson distributions and binary data with Bernoulli distributions, respectively. Nevertheless, methods under non-Euclidean divergences are much more challenging compared with the case under Euclidean loss, especially when the data size becomes huge. Algorithms designed for the LS loss are not easily adaptable to handle complex loss functions, primarily due to the absence of gradient Lipschitz continuity, even under relatively mild conditions.

The form of problem \eqref{Pure_problem} can be applied to tensor CP decomposition with regularization, which can be viewed as an extension of matrix factorization \cite{TeboulleV20}. The first idea of canonical polyadic (CP) decomposition is from Hitchcock \cite{Hitchcock1927, Hitchcock1928} in 1927, which expresses a tensor as a sum of a finite number of rank-1 tensors. Subsequently, Cattell \cite{Cattell1944PP, Cattell1952TheTB} proposed ideas for parallel proportional analysis and multiple axes for analysis. Furthermore, the research
developed by Hong et al. in 2020 is a generalized canonical polyadic (GCP) low-rank tensor decomposition that allows other loss functions besides squared error. Below, we briefly review existing developments for multi-block models with non-Euclidean loss functions.

Many existing non-Euclidean approaches employ the block coordinate descent (BCD) \cite{120887795} for updating block variables. 
Convergence of the BCD method typically requires the uniqueness of the minimizer at each step or the quasi-convexity of the objective function \cite{Tseng2001}. Unfortunately, these requirements can be restrictive in some important practical problems such as the tensor decomposition problem \cite{KoldaB09}.  Cichocki and Phan \cite{Cichocki2009} proposed a hierarchical alternating optimization algorithm for CP decomposition  with $\alpha$- and $\beta$-divergence. In \cite{Kolda2012}, the generalized KL-divergence loss was explored, leading to the development of a block majorization-minimization (MM) algorithm. The work in \cite{Kargas2019} presented the exponential gradient algorithm for handling the KL-divergence. Additionally, alternative optimization frameworks like Gauss-Newton based methods \cite{VMVNLL2020} and quasi-Newton methods \cite{GCP2020} have been devised for non-Euclidean models.

It is worth noting that most of the mentioned algorithms use the entire dataset for each update, which will be time-consuming. In contrast, stochastic algorithms reduce computational and memory requirements per iteration. A recent stochastic gradient-based algorithm \cite{Kolda2019} was introduced for tensor CP decomposition. However, it randomly samples tensor entries for updates, neglecting the potential computational efficiency enhancements through multilinear algebraic properties of low-rank tensors. More importantly, this update strategy loses the opportunity to incorporate regularization terms on the entire latent factors because the sampled entries only provide partial information about them. To address this, Battaglino et al. proposed an algorithm \cite{BattaglinoBK18} that samples tensor fibers containing information about complete latent factors. However, these stochastic algorithms lack convergence guarantees. Pu et al. \cite{Pu2022} developed a block-randomized stochastic mirror descent (SMD) \cite{LanSMD2015, BECK2003167} algorithmic framework for large-scale CP decomposition  under various non-Euclidean losses, also referred to as generalized CP decomposition. Specifically, at each iteration, one block factor is randomly chosen for an update while keeping all other factors fixed. Then, instead of solving the subproblem directly, it updates the unknown factor by one SMD step. This work also incorporated a fiber sampling strategy to assist in designing SMD updates. In this way, the computational cost is much smaller. However, the pure SMD is still slow in convergence.  
Wang et al. proposed mBrasCPD \cite{WangCH21} and iBrasCPD\cite{Wang2023}, which speed up the SGD scheme by the heavy ball method \cite{Polyak64} and inertial acceleration. Both of these algorithms are designed for scenarios involving Euclidean loss functions and are not suitable for the generalized non-Euclidean loss functions directly. Additionally, these algorithms only consider that stochastic gradient is unbiased, which can only induce weak convergence properties. Recently, \cite{WangH23} and \cite{WangLCH24} introduced the Bregman proximal stochastic gradient (BPSG) method and BPSG with extrapolation (BPSGE), respectively and established the convergence properties of the generated sequence in terms of subsequential and global convergence under a general framework of variance reduction. However, both BPSG and BPSGE are designed for single-block problems, and the potential for computing generalized tensor CP decomposition remains untapped.

An overview of several state-of-the-art algorithms for multi-block problem \eqref{Pure_problem} is presented in Table \ref{comparision}.

\begin{table}[!ht]	
\centering
\fontsize{7}{15}\selectfont
\caption{\footnotesize Summary of the properties of iTableSMD (Algorithm \ref{iTableSMD}) and several state-of-the-art stochastic methods.  {``subseq.'' and ``seq.'' denote the subsequential and sequential convergence, respectively.} ``Complexity'' means the complexity (in expectation) to obtain an $\varepsilon$-stationary point (Definition \ref{stationary-point}) of $\Phi$ and 
``-'' means not given.}
\label{comparision}
\begin{tabular}{|c| c| c |c |c |c |c|c|}
\hline
Algorithm&Loss&$h(x)$&Acceleration&Convergence & Complexity& $\nabla f_i$-Lip&$\tilde\nabla f_i$\cr\hline


 BrasCPD \cite{FuIWGH20} &  LS &convex &  no & subseq. &- &yes&unbiased \\\hline
 mBrasCPD \cite{WangCH21} & LS & convex  & heavy ball & subseq. &- &yes&unbiased  \\ \hline
 iBrasCPD \cite{Wang2023} & LS & convex & inertial &subseq. & - &yes&unbiased  \\\hline
 SPRING \cite{DriggsTLDS2020} & general & nonconvex & no &subseq./seq.  & - &yes &biased \\\hline
 iSPALM \cite{HertrichS20} & general & nonconvex & inertial &subseq./seq.  & - &yes &biased \\\hline
 SmartCPD \cite{Pu2022} & general  & convex &no & subseq. &- &no &unbiased \\\hline
 iTableSMD (Alg. \ref{iTableSMD}) & general  & weakly-convex & inertial &subseq./seq.  & $\mathcal{O}(\varepsilon^{-2})$ &no&biased  \\
 \hline
\end{tabular}
\end{table}

In this paper, inspired by inertial acceleration skill and variance reduction framework for stochastic algorithms, we propose an inertial accelerated stochastic mirror descent to solve the nonconvex and nonsmooth optimization problem \eqref{Pure_problem}, which can be applied to the block-wise subproblem of tensor generalized CP decomposition under non-Euclidean loss functions. Our main contributions addressed in this article are as follows: 


\begin{itemize}
\item[(1)] We introduce an inertial accelerated block-randomized SMD algorithm, denoted as iTableSMD, designed to address the GCP decomposition problem. Within a broader multi-block variance reduction framework, we demonstrate the sublinear convergence rate for the subsequential sequence produced by the iTableSMD algorithm.

\item[(2)]  Within a broader multi-block variance reduction and inertial acceleration
framework, we establish the sublinear convergence rate for the subsequential sequence produced
by the iTableSMD algorithm. Furthermore, we introduce a novel Lyapunov function and prove that it requires at most $\mathcal{O}(\varepsilon^{-2})$ iterations in expectation to attain an $\varepsilon$-stationary point. Additionally, we establish the global convergence of the sequence generated by iTableSMD.  

\item[(3)] We conduct extensive experiments, including three synthetic datasets and two real-world datasets in several distributions, to demonstrate the effectiveness of our proposed algorithms iTableSMD. 
Our numerical experiments exhibit that our proposed methods can achieve better convergence.
\end{itemize}

The rest of this paper is organized as follows. Section \ref{preliminary} outlines necessary definitions and preliminary results of existing models and algorithms. Section \ref{algorithm} details the formulation of the iTableSMD algorithm, while Section \ref{convergence_analysis} is dedicated to proving its convergence and analyzing its rate of convergence. In Section \ref{numercial_experiments}, we compare the performance of the iTableSMD algorithm with several baselines using synthetic and real-world datasets. We conclude the paper in Section \ref{conclusion}.

\section{Preliminaries}\label{preliminary}


\begin{definition}(Bregman Divergence): Given a strongly convex function $\psi(\cdot): \operatorname{dom} \psi\left(\subseteq \mathbb{R}^n\right) \rightarrow \mathbb{R}$, the Bregman distance between $x \in \operatorname{dom} \psi$ and $y \in$ int dom $\psi$ is
\begin{equation}
    D_{\psi}(x, y) := \psi (x) - \psi (y) - \langle \nabla\psi(y), x - y\rangle. 
\end{equation}
\end{definition} 
It measures the proximity of $x$ and $y$. Indeed, $\psi$ is convex if and only if $D_{\psi}(x,y) \ge 0$ for any $x\in\text{dom }\psi,y\in\text{int dom }\psi$ due to the gradient inequality.

\begin{remark}
  The Bregman divergence was originally defined in \cite{8821380} using a Legendre function $\psi$. Here, we consider the case where $\psi$ is a strongly convex function; see more details in \cite{BauschkeBT17}. It should be noted that 
  \begin{equation}
      D_{\psi}(x,y) \ge \frac{\sigma}{2}\|x-y\|^2, \quad \forall\, x\in \text{dom }\psi,y\in\text{int dom }\psi, 
  \end{equation}
  where $\sigma$ is the strongly convex parameter of $\psi$.
   $D_\psi\left(x, y\right)=0$ if and only if $x=y$. In addition, $D_{\psi}(x,y)$ can be defined in a coordinate-wise form as $D_\psi\left(x, y\right)=\sum_{i=1}^{I_n} \sum_{j=1}^R\psi\left(x_{ij}\right)-\psi\left(y_{ij}\right)-$ $\left\langle\nabla \psi\left(y_{ij}\right), x_{ij}-y_{ij}\right\rangle$.
\end{remark}
\begin{definition}\label{Ll-smooth}
(\cite{MukkamalaOPS20} $(\bar{L},\underline{L})$-smooth adaptable) Given $\psi$, let $f:\mathcal{X}\rightarrow(-\infty,+\infty]$ be a proper and lower semi-continuous function with $\mathrm{dom}\,\psi\subset\mathrm{dom}\,f$, which is continuously differentiable. We say $(f, \psi)$ is $(\bar{L},\underline{L})$- smooth adaptable  on $C$ if there exist $\bar{L}>0$ and $\underline{L}\ge0$ such that for any $x,y$,
\begin{eqnarray}
f(x)-f(y)-\langle\nabla f(y),x-y\rangle\le \bar{L} D_{\psi}(x,y),\label{L_upper}
\end{eqnarray}
and
\begin{eqnarray}
-\underline{L}D_{\psi}(x,y)\le f(x)-f(y)-\langle \nabla f(y),x-y\rangle.\label{L_lower}
\end{eqnarray}
\end{definition}
If $\underline{L}=\bar{L}$, it recovers \cite[Definition 2.2]{BolteSTV18First}. Suppose  $f$ is convex. If $\underline{L}=0$, this definition recovers \cite[Lemma 1]{BauschkeBT17} and \cite[Definition 1.1]{LuFN18}.

\begin{definition} \label{stationary-point} (\cite{Lan2020First} $\epsilon$-stationary point) 
Given $\epsilon>0$, a solution $\{x_{1}^{*},\dots,x_{s}^{*}\}$ is said to be an $\epsilon$-stationary point of function $\Phi(x_{1},\dots,x_{s})$ if
\[
\mbox{dist}(0,\partial \Phi(x_{1}^{*},\dots,x_{s}^{*}))\le\epsilon. 
\]
\end{definition}

\subsection{Generalized CP decomposition}

Consider an $N$-th order tensor $\mathcal{X}\in\mathbb{R}^{I_{1}\times I_{2}\times\dots\times I_{N}}$, where $I_n$ represents the size of the $n$-th mode of $\mathcal{X}$. Such multi-array arises in many applications. $\mathcal{X}(i_{1},\dots,i_{N})$ is an entry of $\mathcal{X}$. The entries of the data tensor  $\mathcal{X}$ could be various types of real datasets, such as continuous numbers, non-negative integers, and binaries. A general problem of tensor CP decomposition is to approximate $\mathcal{X}$ using a low rank tensor $\mathcal{M}$, defined by
\begin{equation}\label{CP_DE}
    \mathcal{M}=\sum_{r=1}^{R} A_{1}(:, r) \circ \dots \circ A_{N}(:, r),
\end{equation}
where $\circ$ denotes the outer product of vectors, $A_n\in\mathbb R^{I_n\times R}$, are the unknown mode-$n$ latent factor matrix. $R$ is the smallest positive
integer for which equation \eqref{CP_DE} is satisfied, and it is also known as the rank of $\mathcal{M}$.

Let $I=(I_1I_2\dots I_N)^{\frac1N}$ be  the geometric mean of the dimensions. An $N$-dimensional integer vector $i$ is used to represent the entry coordinate, i.e.,
$$
i \in \mathcal{I} \triangleq\left\{\left(i_1, i_2, \ldots, i_N\right) \mid i_n=1,2, \ldots, I_n, \forall n\right\}
$$
The generalized CP (GCP) decomposition problem can be formulated as the following optimization task, where the primary objective is to minimize a data-adaptive loss function denoted by $f(\cdot,\cdot)$: $\mathbb{R} \times \mathbb{R} \mapsto \mathbb{R}$,
\begin{eqnarray}\label{GCP_element}
 \begin{aligned}
\min _{A_1, A_2, \ldots, A_n} &\quad \frac{1}{I^N} \sum_{i \in \mathcal{I}} f\left(\underline{\mathcal{M}}_{i} ; \underline{\mathcal{X}}_{i}\right)+\sum_{n=1}^N h_n\left(A_n\right) \\
 \text { s.t. } &\quad \underline{\mathcal{M}}_{i}=\sum_{r=1}^R \prod_{n=1}^N {A}_n\left(i_n, r\right), \forall\, i \in \mathcal{I}, \\
\end{aligned}   
\end{eqnarray}
where the entries $\underline{\mathcal{X}}_i$ and $\underline{\mathcal{M}}_i$ correspond to the elements of $\mathcal{X}$ and $\mathcal{M}$ indexed by $i$, respectively. $h_{n}(A_{n})$ is a structure-promoting regularizer that captures the prior information about the latent factors ${A}_n$, such as column-wise orthogonality \cite{Krijnen2008}, Tikhonov regularization \cite{Paatero2000}, iterated Tikhonov regularization \cite{Navasca2008},  nonnegativity \cite{Lim2009}. Those regularizations can result in well-posed problems. For instance, if $A_{n}\in \mathbb{R}^{I_n\times R}_{+}:=\{A_{n}\vert A_{n}\ge 0\}$ is applied, we can write $h_{n}(\cdot)$ as the indicator function: 
\begin{equation*}\label{nonnegative_eq}
	h_{n}(A)=\mathcal{I}_{\mathbb R_+^{I_n\times R}}\left(A\right)=\left\{\begin{array}{ll}
		0, & A \ge0, \\
		\infty, & \text { otherwise. }
	\end{array}\right.
\end{equation*}  Lim and Comon \cite{Lim2009} showed that the nonnegative CP decomposition always has optimal solutions.  
By selecting appropriate loss functions $f$, problem \eqref{GCP_element} becomes adaptable to handle diverse data types, including continuous, count, and binary data. To illustrate, we present several representative motivating examples.


The difference between GCP and the conventional CP formulation lies in the flexibility in the selection of loss functions. In this section, we provide alternative loss functions by examining the statistical likelihood of a model for a specific data tensor. We assume a parameterized probability density function (PDF) or probability mass function (PMF) is available, offering the likelihood estimation for each entry, i.e.,
$$
x_i \sim p\left(x_i \mid \theta_i\right), \quad \text { where } \quad \ell\left(\theta_i\right)=m_i. 
$$
Here $x_i$ represents an observation of a random variable, while $\ell(\cdot)$ denotes an invertible link function connecting the model parameter $m_i$ with the natural parameter of the distribution, $\theta_i$. The link function is commonly assumed to be the identity function or related to the expectation of distribution. 

We aim to find the model $\mathcal{M}$ that is the maximum likelihood estimate (MLE) across all entries. By assuming conditional independence of observations, the overall likelihood is simplified to the product of individual likelihoods. Hence, the MLE is the solution of the following optimization problem:
$$
\max_{\mathcal{M}} \ L(\mathcal{M} ; \mathcal{X}) \equiv \prod_{i \in \Omega} p\left(x_i \mid \ell^{-1}(m_i)\right), \quad \forall \, i \in \Omega.
$$
Here, $\mathcal M=\{m_i\}_{i\in\Omega}$. Then, we employ the negative logarithm to transform the product into a summation and convert it to a minimization problem. As the logarithm is a monotonic function, it preserves the maximizer.
\begin{equation}\label{MLE}
 \min_{\mathcal{M}}\ F(\mathcal{M} ; \mathcal{X}) \equiv \sum_{i \in \Omega} f_{0}\left(m_i ; x_i \right), \quad \text { where} \quad f_{0}(m_i ; x_i) \equiv-\log p\left(x \mid \ell^{-1}(m)\right). 
\end{equation}

\begin{table}[!htb]
\renewcommand\arraystretch{1.2}
\caption{Generalized loss function for different data types.}\label{dif_loss}
\begin{tabular}{ccccc}
\hline Data Type & Distribution & Link Function & Loss function & Constraints \\
\hline \multirow{2}{*}{ Continuous: $x \in \mathbb{R}$} & Gaussian & $\ell^{-1}(m)=m$ & $\frac{1}{2}(x-m)^2$ & $A_{n}\in\mathbb{R}^{I_n\times R}$ \\
& Gamma & $\ell^{-1}(m)=m$ & $\frac{x}{m}+\log (m)$ & $A_{n}\in \mathbb{R}^{I_n\times R}_{+}$ \\
\hline \multirow{2}{*}{Count: $x \in \mathbb{N}$ }& \multirow{2}{*}{Poisson} & $\ell^{-1}(m)=m$ & $m-x \log (m)$ & $A_{n}\in\mathbb{R}^{I_n\times R}_{+}$ \\
& & $\ell^{-1}(m)=e^m$ & $e^m-x m$ & $A_{n}\in\mathbb{R}^{I_n\times R}$ \\
\hline \multirow{2}{*}{Binary: $x \in\{0,1\}$ }& \multirow{2}{*}{Bernoulli} & $\ell^{-1}(m)=\frac{m}{1+m}$ & $\log (m+1)-x \log (m)$ & $A_{n}\in\mathbb{R}^{I_n\times R}_{+}$ \\
& & $\ell^{-1}(m)=\frac{e^m}{1+e^m}$ & $\log \left(1+e^m\right)-x m$ & $A_{n}\in\mathbb{R}^{I_n\times R}$ \\
\hline
\end{tabular} 
\end{table}
  
Table \ref{dif_loss} presents commonly utilized generalized loss functions $f(x,m)$, associated link functions, and various distributions. 


\subsection{Stochastic methods for GCP decomposition}
We can modify the element-wise regularized problem \eqref{GCP_element} to a block-wise regularized problem 
\begin{align}\label{regular_GCPD}
	\min_{\{A_{n}\}_{n=1}^{N}}\,\, \Phi(A_{1},\dots,A_{N}) := f(A_{1},\dots,A_{N})+\sum_{n=1}^{N}h_{n}(A_{n}),
\end{align}
{where $f=f_{0}([\![A_1, A_2,\cdots,A_N ]\!]_i ; x_i)$, $m_i=\sum_{r=1}^R \prod_{n=1}^N {A}_n\left(i_n, r\right), \forall\, i \in \mathcal{I}$, and $f : \Pi_{n=1}^{N}\mathbb{R}^{I_{n}\times R}\rightarrow \mathbb{R}$ is finite-valued and differential. Shortly, let $\mathcal{V}(\Phi)= \min_{A_n} 	\Phi(A_{1},\dots,A_{N})$.} 

We present the stochastic gradient with respect to the factor $A_{n}  (n=1,\dots,N)$, denoted by   $\tilde{\nabla}_{A_n}f(A_{1},\dots,A_{N})$. Suppose a mode index $n$ is sampled from $1$ to $N$. For instance, we take the squared error loss function as an example. 
Namely, consider 
\begin{eqnarray*}
		  f(A_1,\dots,A_N)=\frac{1}{2 I^N} \left\| [\![A_1, A_2,\cdots,A_N ]\!] - \mathcal X\right\|_F^2 = \frac{1}{2 I^N} \left\|H_nA_n^\top - X_{(n)}\right\|_F^2,
  \end{eqnarray*}
where $H_n=A_N\odot \dots A_{n+1}\odot A_{n-1}\dots A_1\in\mathbb R^{J_n\times R}$, $\odot$ denotes the Khatri-Rao product or a column-wise Kronecker product, $X_{(n)}\in\mathbb R^{J_n\times I_n}$ is the mode-$n$ matrization and $J_n=I^N/I_n$. 

Then,  we rewrite the  full gradient as follows
\begin{eqnarray}\label{full_gra}
	\begin{aligned}
		\nabla_{A_n} f(A_1,\dots,A_N)=&\frac{1}{I^N}\left(A_{n} \sum_{i=1}^{J_n} H_{n}(i,:)^{\top}H_{n}(i:)-X_{(n)}(i,:)^{\top}H_{n}(i,:)\right)\\
		=&\frac{1}{I_n}\left(A_{n} \mathbb E_i[ H_{n}(i,:)^{\top}H_{n}(i:)]- \mathbb{E}_i[X_{(n)}(i,:)^{\top}H_{n}(i,:)]\right).
	\end{aligned}
\end{eqnarray}
Here, $H(i,:)$ denotes the $i$-th column of $H$, and $\mathbb E_i$ is the expectation over the index $i$. To alleviate the burden of computing the full gradient \eqref{full_gra}, we randomly sample a set of mode-$n$ fibers that is indexed by $\mathcal{F}_{n}\subset\{1,\dots,J_{n}\}$ with $|\mathcal{F}_{n}|=B$. 
Note that a mode-$n$ fiber of $\mathcal{X}$  is a row of the mode-$n$ unfolding $X_{(n)}$. Compared with the ﬁber sampling-based method in \cite{BattaglinoBK18}, our requirement on the batchsize $B$ is much lower. Hence, it admits lower per-iteration memory and computational complexities, especially when the rank is high. 


Let $\tilde{\nabla}_{A_n}f(A_{1},\dots,A_{N})\in\mathbb{R}^{I_{n}\times R}$ be the stochastic gradient of $f(A_{1},\dots, A_{N})$ for $A_{n}$,   we have
\begin{eqnarray}
    \begin{aligned}
	 \tilde{\nabla}_{A_n}f(A_{1},\dots,A_{N})&=\frac{1}{{I_n}\left|\mathcal{F}_{n}\right|}\left(A_{n} H_{n}^{\top}\left(\mathcal{F}_{n}\right) H_{n}\left(\mathcal{F}_{n}\right)-X_{n}^{\top}\left(\mathcal{F}_{n}\right) H_{n}\left(\mathcal{F}_{n}\right)\right), 
     \end{aligned}\label{sto_gradient}
\end{eqnarray}
where 
\[
X_{n}\left(\mathcal{F}_{n}\right)=X_{n}\left(\mathcal{F}_{n},:\right), \quad H_{n}\left(\mathcal{F}_{n}\right)=H_{n}\left(\mathcal{F}_{n},:\right).
\]

\subsection{Stochastic mirror descent}

In this subsection, we introduce some basics for optimization methods used in this paper, such as stochastic mirror descent (SMD) and inertial framework. 

Consider the special case in optimization problem \eqref{Pure_problem} with the block $s=1$
\[
\min_{x\in\mathbb{R}^{d}} \Phi\left(x\right)=f (x)+ h\left(x\right),
\]
where the component function $f$ is a continuously differentiable nonconvex function, and $h$ is an extended valued function that are bounded from below. The update of SMD \cite{lu2018relativecontinuity, doi:10.1137, zhang2018convergence} is given as follows,
\[
x^{k+1}\in \underset{x}{\arg\min}\,\, h(x) + \langle \tilde{\nabla} f(x^{k}), x-x^{k}\rangle + \frac{1}{\eta^{k}} D_{\psi}(x, x^{k}),
\]
where the stochastic gradient can be chosen as mini-batch version  $\tilde{\nabla} f(x^{k})=\frac{1}{|B^{k}|}\sum_{i\in B^{k}}\nabla f_{i}(x^{k})$. Here, the mini-batch $B^{k}$ is chosen uniformly at random from all subsets of $\{1,\dots,T\}$, with the batchsize $|B^{k}|=B$ being considerably smaller than $T$.
Compared with SGD, SMD replaces the quadratic term $\left\|{A}_n-{A}_n^k\right\|_F^2$.
When $\psi$ is properly designed, SMD can exploit the geometry of the problem and achieve significant efficiency enhancements compared to SGD, particularly when utilizing generalized loss functions. The extensive literature on MD and SMD in optimization is available \cite{BolteSTV18First, Latafat2022, davis2018, zhang2018convergence, li2019provable}.

As usual with the analysis of Bregman based
schemes, the following simple but remarkable three points identity for $D_{\psi}$ is very useful, which follows from elementary algebra. Given any $x \in \operatorname{dom} \psi$ and $y,z \in$ int dom $\psi$, the three point equality is 
\begin{equation}\label{equ:3point}
    D_{\psi}(x,z) = D_{\psi}(x,y) + D_{\psi}(y,z)+\langle \nabla \psi(y) - \nabla \psi(z), x-y\rangle.
\end{equation}

{For the multi-block problem, Pu et al. \cite{Pu2022} develop a unified stochastic mirror descent algorithmic framework (SmartCPD) for large-scale CPD under various non-Euclidean losses, which is a special case of multi-block problem and updates the factor variables by }
\begin{eqnarray}
	\begin{aligned}
		A_{n}^{k+1} &= \arg \min _{A_{n}}\,\,h_{n}\left(A_{n}\right)+\langle \tilde{\nabla}_{A_n}f(A_{1}^k,\dots,A_{N}^k),A_{n}-A_{n}^{k}\rangle+\frac{1}{\eta^{k}} D_{\psi}(A_{n}, A_{n}^{k}), \\
		A_{n^{\prime}}^{k+1} &= A_{n^{\prime}}^{k}, \quad n^{\prime} \neq n.
	\end{aligned}\label{update_A_O}
\end{eqnarray}
However, directly employing stochastic mirror descent for the GCP problem may not yield the most effective results. In this paper, we study stochastic gradients under the variance-reduced stochastic gradient estimators, such as SAGA \cite{DefazioBL14} and SARAH \cite{NguyenLST17}. Furthermore, the inertial acceleration framework is applied, which can be given by
\[
\begin{aligned}
	\bar{x}^{k}=&x^{k}+\alpha^{k}(x^{k}-x^{k-1}), \\
	\hat{x}^{k}=&x^{k}+\beta^{k}(x^{k}-x^{k-1}),\\
	x^{k+1}=& \hat{x}^{k}-\eta^{k} \nabla f(\bar{x}^{k}),
\end{aligned}
\]
where $\alpha^{k},\beta^{k}\in[0,1]$ are two inertial parameters. For example, if $\alpha^{k}=\beta^{k}=0$, it will be degenerated into the gradient descent method; If $\alpha^{k}=0$, then it will be reduced to the heavy-ball method \cite{Polyak64}; If $\alpha^{k}=\beta^{k}$, then it will be reduced to the Nesterov accelerated gradient method \cite{Nesterov1983}.

\section{Inertial accelerated block-randomized SMD}\label{algorithm} 

In this section, we propose an inertial accelerated block-randomized stochastic mirror descent algorithm (iTableSMD) for GCP decomposition \eqref{regular_GCPD}.
Before presenting the algorithm framework of iTableSMD, we make the following assumptions throughout the paper.

\begin{assumption} \label{assume_01}
	We assume that the following three conditions hold:
	\begin{itemize}
            \item[(i)]  \label{a_1} $h_{n}: \mathbb{R}^{I_{n}\times R}\rightarrow\mathbb{R}\cup \{+\infty\}  (n=1,2,\dots,N)$ are proper lower semi-continuous (l.s.c.) functions that are bounded from below. There exists $\alpha\in\mathbb{R}_{+}$ such that $h(\cdot)+\frac{\alpha}{2}\|\cdot\|^{2}$ is convex.
            \item[(ii)] $\psi:\mathbb{R}^{I_{n}\times R}\rightarrow\mathbb{R}$ is continuously differentiable and $\sigma$-strongly convex. Let $\sigma=1$ for simplicity.
            \item[(iii)]$\nabla \psi $ is Lipschitz continuous with modulus $M_{2} > 0$. For any two points $A_{i}$, $\hat{A}_{i}$ $\in \mathbb{R}^{I_{n}\times R}$, it presents that
            \[
		\|\nabla \psi(A_{i})-\nabla \psi(\hat{A}_{i})\|\le M_2\|A_{i}-\hat{A}_{i}\|, \quad i=1,2,\dots,N.
		\]
            \item[(iv)]$f : \Pi_{n=1}^{N}\mathbb{R}^{I_{n}\times R}\rightarrow \mathbb{R}$ is a proper and lower semi-continuous function with $\text{dom }\psi\subset\text{dom }f$.
		\item[(v)] The couple of functions $(f,\psi)$ is $(\bar{L},\underline{L})$-smooth adaptable.
            \item[(vi)] The function $\Phi$ is bounded from below, i.e., there exists a finite optimal objective value $\mathcal{V}(\Phi)$.
	\end{itemize}
\end{assumption}

\begin{algorithm}[H]
	\caption{iTableSMD: inertial accelerated block-randomized stochastic mirror descent for the optimization  problem \eqref{regular_GCPD}}
	\label{iTableSMD}
	{\bfseries Input:} an $N$-way tensor $\mathcal{X}\in\mathbb{R}^{I_{1}\times \dots\times I_{N}}$; the rank $R$; the sample size $B$; initialization $\{A_{n}^{-1}\}_{n=1}^{N}, \{A_{n}^{0}\}_{n=1}^{N}$; stepsize $\{\eta_{k}\}_{k\ge0}$; inertial parameters  $\{\alpha^{k}\}_{k\ge0},\{\beta^{k}\}_{k\ge0}\in[0,1]$; two constants $\delta,\epsilon$ with $1>\delta>\epsilon>0$.
	\begin{algorithmic}[1]
		\State $k\leftarrow 0$;
		\Repeat 
		\State sample $n$ uniformly  from $\{1, \dots, N\}$.
		\State sample $\mathcal{F}_{n}$ uniformly from $\{1,\dots,J_{n}\}$ with $|\mathcal{F}_{n}|=B$.
            \State compute $\tilde{A}_{n}^{k}={A}_{n}^{k}+\alpha^{k}(A_{n}^{k}-A_{n}^{k-1})\in\mathrm{int}\,\mathrm{dom}\, \psi$, where $\alpha_{k}\in[0,1)$. 
            \State compute an extrapolation parameter $\beta^{k}$ such that
            \begin{eqnarray}
            D_{\psi}(A_{n}^{k},\underline{A}_{n}^{k})\le \frac{\delta-\epsilon} {1+ \underline L\eta_{k-1}}D_{\psi}(A_{n}^{k-1},A_{n}^{k}), \label{extra_ineq}
            \end{eqnarray}
 		\,\,\,\,\,\,\,\,\,where $\underline{A}_{n}^{k}=A_{n}^{k}+\beta^{k}(A_{n}^{k}-A_{n}^{k-1})\in\mathrm{int}\,\mathrm{dom}\, \psi$.
		\State compute the stochastic gradient $\tilde{\nabla}_{A_n}f(A_1^k \cdots \underline{A}_n^k \cdots A_N^k)$ with the batchsize of fibers $B$.
		\State set $\eta_{k}\le\min\{\eta_{k-1},\bar{L}^{-1}\}$ and update $A_{n}^{k+1}$ and $A_{n^{\prime}}^{k+1}$:
		\begin{equation}
		\begin{aligned}
			A_{n}^{k+1} &= \arg \min _{A_{n}}\,\,h_{n}\left(A_{n}\right)+\langle\tilde{\nabla}_{A_n}f(A_1^k \cdots \underline{A}_n^k \cdots A_N^k),A_{n}-\tilde{A}_{n}^{k}\rangle+\frac{1}{\eta_{k}} D_{\psi}(A_{n},\tilde{A}_{n}^{k}), \label{xk_update}\\
			A_{n^{\prime}}^{k+1} &= A_{n^{\prime}}^{k}, \quad \forall n^{\prime} \neq n.
		\end{aligned}
		\end{equation}
		\State $k\leftarrow k+1$;
		\Until{some stopping criterion is reached;}
	\end{algorithmic}
	{\bfseries Output:} $\{A_{n}^{k}\}_{n=1}^{N}$.
\end{algorithm}

Let $\xi^k$ and $\zeta^k$  be the stochastic parameters for the block index and the stochastic gradient, respectively. Denote $\mathbb E_k[\cdot]=\mathbb E[\cdot |  \xi^k,\zeta^k]$ and $\mathbb E[\cdot]=\mathbb E[\cdot |\xi^0,\zeta^0,\dots]$.
\begin{definition}(Variance reduced stochastic gradient)\label{vr_definition}
	We say a gradient estimator $\tilde{\nabla}_{A_{n}}f$ with $n=1,2\dots,N$,  is variance-reduced with constants $V_{1}, V_{2},V_{\Gamma}\ge 0$, and $\tau\in(0,1]$ if it satisfies the following conditions:
	\begin{itemize}
		\item[(i)] (MSE Bound): there exists a sequence of random variables $\{\Gamma_{k}\}_{k\ge1}$ such that
		\begin{small}
			\begin{eqnarray}
				\begin{aligned}
	&\mathbb{E}_{k}[\|\tilde{\nabla}_{A_{\xi^k}}f(A_{1}^{k},\cdots,A_{\xi^{k}-1}^{k},\underline{A}_{\xi^{k}}^{k},A_{\xi^{k}+1}^{k},\cdots,A_{N}^{k})-\nabla_{A_{\xi^k}}f(A_{1}^{k},\cdots,A_{\xi^k-1}^{k},\underline{A}_{\xi^k}^{k},A_{\xi^{k}+1}^{k},\cdots,A_{N}^{k})\|_{*}^{2}]\\
	\le&\Gamma_{k}+V_{1}(\|A^{k}-A^{k-1}\|^{2}+\|A^{k-1}-A^{k-2}\|^{2}),
\end{aligned}\label{MSE_l22}
\end{eqnarray}
\end{small}
and  random variables $\{\Upsilon_{k}\}_{k\ge1}$ such that
\begin{small}
\begin{eqnarray}
\begin{aligned}
&\mathbb{E}_{k}[\|\tilde{\nabla}_{A_{\xi^{k}}}f(A_{1}^{k},\cdots,A_{\xi^{k}-1}^{k},\underline{A}_{\xi^{k}}^{k},A_{\xi^{k}+1}^{k},\cdots,A_{N}^{k})-\nabla_{A_{\xi^{k}}}f(A_{1}^{k},\cdots,A_{\xi^{k}-1}^{k},\underline{A}_{\xi^{k}}^{k},A_{\xi^{k}+1}^{k},\cdots,A_{N}^{k})\|_{*}]\\
\le&\Upsilon_{k}+V_{2}(\|A^{k}-A^{k-1}\|+\|A^{k-1}-A^{k-2}\|).
\end{aligned}\label{MSE_l2}
\end{eqnarray}
\end{small}
\item[(ii)] (Geometric Decay): The sequence $\{\Gamma_{k}\}_{k\ge1}$ satisfy the following inequality in expectation:
\begin{eqnarray}
\begin{aligned}
\mathbb{E}_{k}[\Gamma_{k+1}]\le& (1-\tau)\Gamma_{k}+V_{\Gamma}(\|A^{k}-A^{k-1}\|^{2}+\|A^{k-1}-A^{k-2}\|^{2}).\label{Gamma_k1_k}
\end{aligned}
\end{eqnarray}
\item[(iii)] (Convergence of Estimator): For all sequences $\{A^{k}\}_{k=0}^{\infty}$, if they satisfy \\
$\lim_{k\rightarrow\infty}\mathbb{E}\|A^{k}-A^{k-1}\|^{2}\rightarrow 0$, then it follows that $\mathbb{E}\Gamma_{k}\rightarrow0$ and $\mathbb{E}\Upsilon_{k}\rightarrow0$.
\end{itemize}
\end{definition}
 In Proposition \ref{vr_gra}, we show both SAGA and SARAH are variance reduced stochastic gradients.


\section{Convergence analysis}\label{convergence_analysis}






This section establishes the convergence properties of the iTableSMD algorithm. We prove its sublinear convergence rate for the subsequential sequence and further show that iTableSMD requires at most $\mathcal{O}(\varepsilon^{-2})$ iterations in expectation to attain an $\varepsilon$-stationary point. Additionally, we confirm the global convergence of the generated sequence.

\subsection{Subsequential convergence analysis}

Next, we show the descent amount of $\Phi\left(A_1^{k+1}, \cdots, A_N^{k+1}\right)$ under expectation  in the following lemma.

\begin{lemma}\label{lemma_Phi_kk1}
Suppose Assumption\ref{assume_01} is satisfied  and  $\tilde{\nabla}_{A_{n}}f$ with $n=1,2\dots,N$,  is variance-reduced by Definition \ref{vr_definition}. Let $\{A_{n}^{k}\}_{k>0}$ with $n\in\{1,\dots,N\}$ be the sequence generated by Algorithm \ref{iTableSMD}.
Then the following inequality holds for any $k>0$, 
\begin{align*}
&\mathbb{E}_{k}[\Phi\left(A_1^{k+1}, \cdots, A_N^{k+1}\right)] +\frac{1}{2\bar\gamma\tau}\mathbb{E}_{k}[\Gamma_{k+1}]+\left(\frac{1}{\eta_{k}}-\alpha-\bar\gamma-\frac{\gamma_{k}}{\eta_{k}}\right)\mathbb{E}_{k}[D^N_{\psi}(A^{k},A^{k+1})]\\
\le&\Phi\left(A_1^k, \cdots, A_N^k\right)+\frac{1}{2\bar\gamma\tau}\Gamma_{k}+\left(\frac{\delta-\epsilon}{\eta_{k}}+\frac{\bar{\gamma}}{2}+\frac{M_{2}^{2}(\alpha_{k}-\beta_{k})^{2}}{\eta_{k}\gamma_{k}}\right)D^N_{\psi}(A^{k-1},A^{k})+\frac{\bar\gamma}{2}D^N_{\psi}(A^{k-2},A^{k-1}).
\end{align*}
Here,  $\bar\gamma=\sqrt{2(V_{\Gamma}/\tau+V_{1})}$,  $\alpha$ is the weakly convex parameter in Assumption \ref{assume_01}~(i),  $\delta$ and $\epsilon$ are introduced in \eqref{extra_ineq}, and $V_{1}, V_{2},V_{\Gamma}\ge 0$, $\tau\in(0,1]$ are parameters in Definition~\ref{vr_definition}.
\end{lemma}
\begin{proof}

From the convexity of $h(\cdot)+\frac{\alpha}{2}\|\cdot\|^{2}$, we can obtain the following inequality
\begin{equation}\label{h_weakconvex}
 h_n(A_n^{k+1})+\frac{\alpha}{2}\left\|A_n^{k+1}\right\|^2 +\left\langle\xi_{k+1}+\alpha A_n^{k+1}, A_n^k-A_n^{k+1}\right\rangle \leq h_n(A_n^k)+\frac{\alpha}{2}\left\|A_n^k\right\|^2,
\end{equation}
where $\xi_{k+1}\in \partial h_n(A_{n}^{k+1})$. From the optimality condition of \eqref{xk_update}, it shows that
\begin{equation*}
\xi_{k+1}+\tilde{\nabla}_{A_n}f(A_1^k \cdots \underline{A}_n^k \cdots A_N^k)+\frac{1}{\eta_k}(\nabla \psi(A_n^{k+1})-\nabla \psi(\tilde{A}_n^k))=0,
\end{equation*}
which combined with \eqref{h_weakconvex} yields that
$$
\begin{aligned}
h_n(A_n^{k+1}) & \leq h_n(A_n^k)+\frac{\alpha}{2}\left\|A_n^{k+1}-A_n^k\right\|^2+\langle\tilde{\nabla}_{A_n}f(A_1^k \cdots \underline{A}_n^k \cdots A_N^k), A_n^k-{A}_n^{k+1}\rangle \\
& -\frac{1}{\eta_k} D_\psi(A_n^k, A_n^{k+1})-\frac{1}{\eta_k} D_\psi(A_n^{k+1}, \tilde{A}_n^k) +\frac{1}{\eta_k} D_\psi(A_n^k, \tilde{A}_n^k).
\end{aligned}
$$
Furthermore, since $f$ is an $(\bar{L},\underline{L})$-relative smooth function with respect to $\psi$, we have
$$
\begin{aligned}
    f(A_1^{k+1} \cdots A_n^{k+1} \cdots A_N^{k+1}) &\leq f(A_1^k \cdots \underline{A}_n^k \cdots A_N^k)\\
    &+ \langle\nabla_{A_n} f(A_1^k \cdots \underline{A}_n^k \cdots A_N^k), A_n^{k+1}-\underline{A}_n^k\rangle +\bar{L} D_\psi(A_n^{k+1}, \underline{A}_n^k),
\end{aligned}
$$
and
$$
\begin{aligned}
    &f(A_1^k \cdots \underline{A}_n^k \cdots A_N^k)+\langle\nabla_{A_n} f(A_1^k \cdots \underline{A}_n^k \cdots A_N^k), A_n^{k}-\underline{A}_n^k\rangle 
    \leq f(A_1^k, \cdots A_n^k\cdots A_N^k) 
     +\underline{L} D_\psi(A_n^{k}, \underline{A}_n^k).
\end{aligned}
$$
Combining two inequalities, we can get
\begin{eqnarray}
\begin{aligned}
f(A_1^{k+1} \cdots A_n^{k+1} \cdots A_N^{k+1}) &\leq   f(A_1^k \cdots A_n^k \cdots A_N^k)+\langle\nabla_{A_n} f(A_1^k \cdots \underline{A}_n^k \cdots A_N^k), A_n^{k+1}-A_n^k\rangle\\
&+\bar{L} D_\psi(A_n^{k+1}, \underline{A}_n^k)  +\underline{L} D_\psi(A_n^k, \underline{A}_n^k).
\end{aligned} 
\end{eqnarray}

By summing the two inequalities together, we obtain
\begin{eqnarray}
\begin{aligned}
\Phi\left(A_1^{k+1}, \cdots, A_N^{k+1}\right) &\leq \Phi\left(A_1^k, \cdots, A_N^k\right)+\frac{\alpha}{2}\|A_n^{k+1}-A_n^k\|^2 
 \\ &+\langle\nabla_{A_n} f(A_1^k \cdots \underline{A}_n^k \cdots A_N^k)-\tilde{\nabla}_{A_n}f(A_1^k \cdots \underline{A}_n^k \cdots A_N^k), A_n^{k+1}-A_n^k\rangle\\
 &+\underline{L} D_\psi(A_n^k, \underline{A}_n^k)+\frac{1}{\eta_k}D_\psi(A_n^k, \tilde{A}_n^k)-\frac{1}{\eta_k} D_\psi(A_n^k, A_n^{k+1})\\
 &+\bar{L} D_\psi(A_n^{k+1}, \underline{A}_n^k)-\frac{1}{\eta_k}D_\psi(A_n^{k+1}, \tilde{A}_n^k)\\
 &\leq \Phi\left(A_1^k, \cdots, A_N^k\right)+\frac{\alpha+\bar\gamma_{k}}{2}\|A_n^{k+1}-A_n^k\|^2 \\
& + \frac{1}{2\bar\gamma_{k}}\|\nabla_{A_n} f(A_1^k \cdots \underline{A}_n^k \cdots A_N^k)-\tilde{\nabla}_{A_n}f(A_1^k \cdots \underline{A}_n^k \cdots A_N^k)\|_{*}^{2} \\
 &+\underline{L} D_\psi(A_n^k, \underline{A}_n^k)+\frac{1}{\eta_k}D_\psi(A_n^k, \tilde{A}_n^k)-\frac{1}{\eta_k} D_\psi(A_n^k, A_n^{k+1})\\
 &+\bar{L} D_\psi(A_n^{k+1}, \underline{A}_n^k)-\frac{1}{\eta_k}D_\psi(A_n^{k+1}, \tilde{A}_n^k),
\end{aligned}\label{ineq_01}
\end{eqnarray}
where the last inequality follows from $\langle a,b\rangle \le \frac{\gamma}{2}\|a\|^{2}+\frac{1}{2\gamma}\|b\|_{*}^{2}$ for any $\gamma_k>0$ and $\eta_{k}\le \bar{L}^{-1}$.

By \eqref{equ:3point}, we know
 \begin{align*}
D_{\psi}(A_n^{k+1}, \underline A_n^{k})=&D_{\psi}(A_n^{k+1}, A_n^{k})+D_{\psi}(A_n^{k}, \underline A_n^{k})+\langle \nabla\psi(A_n^{k})-\nabla\psi(\underline A_n^{k}), A_n^{k+1}-A_n^{k}\rangle,\\
D_{\psi}(A_n^{k+1}, \tilde A_n^{k})=&D_{\psi}(A_n^{k+1}, A_n^{k})+D_{\psi}(A_n^{k}, \tilde A_n^{k})+\langle \nabla\psi(A_n^{k})-\nabla\psi(\tilde A_n^{k}), A_n^{k+1}-A_n^{k}\rangle.
 \end{align*}
For the last two terms on the right of the  inequality \eqref{ineq_01}, it shows that
\begin{eqnarray}\begin{aligned}
&\bar{L} D_\psi(A_n^{k+1}, \underline{A}_n^k)-\frac{1}{\eta_k}D_\psi(A_n^{k+1}, \tilde{A}_n^k)\\
\le&\frac{1}{\eta_{k}}\left[D_\psi(A_n^{k+1}, \underline{A}_n^k)-D_\psi(A_n^{k+1}, \tilde{A}_n^k)\right]\\
=&\frac{1}{\eta_{k}}D_{\psi}(A_n^{k}, \underline A_n^{k})-\frac{1}{\eta_{k}}D_{\psi}(A_n^{k}, \tilde A_n^{k})+\frac{1}{\eta_{k}}\langle \nabla\psi(\tilde A_n^{k})-\nabla\psi(\underline A_n^{k}), A_n^{k+1}-A_n^{k}\rangle\\
\le&\frac{1}{\eta_{k}}D_{\psi}(A_n^{k}, \underline A_n^{k})-\frac{1}{\eta_{k}}D_{\psi}(A_n^{k}, \tilde A_n^{k})+\frac{1}{2\eta_{k}\gamma_{k}}\|\nabla\psi(\tilde A_n^{k})-\nabla\psi(\underline A_n^{k})\|^{2}+\frac{\gamma_{k}}{2\eta_{k}}\|A_n^{k+1}-A_n^{k}\|^{2}\\
\le&\frac{1}{\eta_{k}}D_{\psi}(A_n^{k}, \underline A_n^{k})-\frac{1}{\eta_{k}}D_{\psi}(A_n^{k}, \tilde A_n^{k})+\frac{M_{2}^{2}}{2\eta_{k}\gamma_{k}}\|\tilde A_n^{k}-\underline A_n^{k}\|^{2}+\frac{\gamma_{k}}{2\eta_{k}}\|A_n^{k+1}-A_n^{k}\|^{2}\\
\le&\frac{1}{\eta_{k}}D_{\psi}(A_n^{k}, \underline A_n^{k})-\frac{1}{\eta_{k}}D_{\psi}(A_n^{k}, \tilde A_n^{k})+\frac{M_{2}^{2}(\alpha_{k}-\beta_{k})^{2}}{2\eta_{k}\gamma_{k}}\|A_n^{k}-A_n^{k-1}\|^{2}+\frac{\gamma_{k}}{2\eta_{k}}\|A_n^{k+1}-A_n^{k}\|^{2}.
 \end{aligned}\label{ineq_0002}
 \end{eqnarray}

Suppose $n=\xi^{k}$ at the $k$-th iteration. We apply the conditional expectation operator $\mathbb{E}_{k}$ to the above inequality \eqref{ineq_01} and bounding the MSE term by \eqref{MSE_l22} in Definition \ref{vr_definition}, then we have
\begin{eqnarray}
\begin{aligned}
&\mathbb{E}_{k}[\Phi\left(A_1^{k+1}, \cdots, A_N^{k+1}\right)] \\
\leq &\Phi\left(A_1^k, \cdots, A_N^k\right)+\left(\frac{\alpha+\bar\gamma_{k}}{2}+\frac{\gamma_k}{2\eta_k}\right) \mathbb{E}_{k}[\|A_{\xi^{k}}^{k+1}-A_{\xi^{k}}^k\|^2]\\
& + \frac{1}{2\bar\gamma_{k}}\mathbb{E}_{k}[\|\nabla_{A_\xi^{k}} f(A_1^k \cdots \underline{A}_{\xi^{k}}^k \cdots A_N^k)-\tilde{\nabla}_{A_\xi^{k}}f(A_1^k \cdots \underline{A}_{\xi^{k}}^k \cdots A_N^k)\|_{*}^{2}]\\
& +\left(\frac{1}{\eta_{k}}+\underline{L}\right)\mathbb{E}_{k}[D_\psi({A}_{\xi^{k}}^k, \underline{A}_{\xi^{k}}^{k})]+\frac{M_{2}^{2}(\alpha_{k}-\beta_{k})^{2}}{2\eta_{k}\gamma_{k}}\mathbb{E}_{k}[\|A_{\xi^{k}}^{k}-A_{\xi^{k}}^{k-1}\|^{2}]-\frac{1}{\eta_k}\mathbb{E}_{k}[D_\psi({A}_{\xi^{k}}^k, A_{\xi^{k}}^{k+1})]\\
\leq & \Phi\left(A_1^k, \cdots, A_N^k\right)+\left(\frac{\alpha+\bar\gamma_{k}}{2}+\frac{\gamma_k}{2\eta_k}\right) \mathbb{E}_{k}[\|A^{k+1}-A^k\|^2]\\
&+\frac{1}{2\bar\gamma_{k}}(\Gamma_{k}+V_{1}(\left\|A^{k}-A^{k-1}\right\|^{2}+\left\|A^{k-1}-A^{k-2}\right\|^{2}))+\frac{M_{2}^{2}(\alpha_{k}-\beta_{k})^{2}}{2\eta_{k}\gamma_{k}}\|A^{k}-A^{k-1}\|^{2}\\
&+\left(\frac{1}{\eta_{k}}+\underline{L}\right)\left[D_\psi\left({A}_{1}^k, \underline{A}_{1}^{k}\right)+D_\psi\left({A}_{2}^k, \underline{A}_{2}^{k}\right)+\cdots+D_\psi\left({A}_{N}^k, \underline{A}_{N}^{k}\right)\right]\\
&-\frac{1}{\eta_k }\mathbb{E}_{k}[D_\psi\left({A}_{1}^k, A_{1}^{k+1}\right)+D_\psi\left({A}_{2}^k, A_{2}^{k+1}\right)+\cdots+D_\psi\left({A}_{N}^k, A_{N}^{k+1}\right)]\\
\leq & \Phi\left(A_1^k, \cdots, A_N^k\right)+\left(\frac{\alpha+\bar\gamma_{k}}{2}+\frac{\gamma_k}{2\eta_k}\right)\mathbb{E}_{k}[\|A^{k+1}-A^k\|^2]\\
&+\frac{1}{2\bar\gamma_{k}\tau}(\Gamma_{k}-\mathbb{E}_{k}[\Gamma_{k+1}])+\frac{V_{\Gamma}}{2\bar\gamma_{k}\tau }(\left\|A^{k}-A^{k-1}\right\|^{2}+\left\|A^{k-1}-A^{k-2}\right\|^{2})\\
&+\frac{V_{1}}{2\bar\gamma_{k}}\left(\left\|A^{k}-A^{k-1}\right\|^{2}+\left\|A^{k-1}-A^{k-2}\right\|^{2}\right)+\frac{M_{2}^{2}(\alpha_{k}-\beta_{k})^{2}}{2\eta_{k}\gamma_{k}}\|A^{k}-A^{k-1}\|^{2}\\
&+\frac{\delta-\epsilon}{\eta_{k}}D^N_\psi\left({A}^{k-1}, A^{k}\right)-\frac{1}{\eta_k }\mathbb{E}_{k}[D^N_\psi\left({A}^k, A^{k+1}\right)],
\end{aligned}\label{ineq_002}
\end{eqnarray}
where the last inequality follows from \eqref{Gamma_k1_k} in Definition \ref{vr_definition}. From \eqref{extra_ineq} and $\eta_{k}\le\min\{\eta_{k-1},\bar{L}^{-1}\}$, it presents that
\begin{eqnarray*}
\begin{aligned}
\left(\frac{1}{\eta_{k}}+\underline{L}\right)D_{\psi}({A}_{n}^k, \underline{A}_{n}^{k})\le \frac{\underline{L}\eta_{k}+1}{\eta_{k}}\frac{\delta-\epsilon}{1+\underline{L}\eta_{k-1}}D_{\psi}({A}_{n}^{k-1},{A}_{n}^{k})\le\frac{\delta-\epsilon}{\eta_{k}}D_{\psi}({A}_{n}^{k-1},{A}_{n}^{k}), \label{ineq_003}
\end{aligned}
\end{eqnarray*}
and we also use notation $D^N_{\psi}(A,B):=D_{\psi}(A_1,B_1)+\cdots+D_{\psi}(A_N,B_N)$ for simplicity.
Then we can get
\[
\begin{aligned}
&\mathbb{E}_{k}[\Phi\left(A_1^{k+1}, \cdots, A_N^{k+1}\right)]\\
\le& \Phi\left(A_1^k, \cdots, A_N^k\right) - \left(\frac{1}{\eta_{k}}-\alpha-\bar\gamma_{k}-\frac{\gamma_{k}}{\eta_{k}} \right)\mathbb{E}_{k}[D^N_{\psi}(A^{k},A^{k+1})]\\
&+\frac{1}{2\bar\gamma_{k}\tau}(\Gamma_{k}-\mathbb{E}_{k}[\Gamma_{k+1}])
+\left(\frac{\delta-\epsilon}{\eta_{k}}+\frac{V_{\Gamma}}{\bar\gamma_{k}\tau }+\frac{V_{1}}{\bar\gamma_{k}}+\frac{M_{2}^{2}(\alpha_{k}-\beta_{k})^{2}}{\eta_{k}\gamma_{k}}\right)D^N_{\psi}(A^{k-1},A^{k})\\
&+\left(\frac{V_{\Gamma}}{\gamma_{k}\tau}+\frac{V_{1}}{\gamma_{k}}\right)D^N_{\psi}(A^{k-2},A^{k-1}).
\end{aligned}
\]
Therefore, the results can be obtained by rearranging the above terms with $\bar\gamma_{k}=\sqrt{2(V_{\Gamma}/\tau+V_{1})}$. 
This completes the proof.
\end{proof}

Next, we introduce a new Lyapunov function and show it is
monotonically decreasing in expectation. For simplicity, we denote $$\Phi^k=\Phi\left(A_1^{k}, \cdots, A_N^{k}\right).$$
\begin{lemma}\label{lyapunov_descent}
Suppose the same conditions with Lemma \ref{lemma_Phi_kk1} hold, and  the stepsize satisfies 
\begin{eqnarray}
\eta_{k}\le \min\left\{\eta_{k-1}, \frac{1}{\bar{L}}, \frac{1-\delta-2|\alpha_{k}-\beta_{k}|M_{2}}{\alpha+2\bar{\gamma}}\right\}, \quad \forall\, k>0. \label{stepsize_set}
\end{eqnarray}
Let $\{A_{n}^{k}\}_{k>0}$ with $n\in\{1,\dots,N\}$ be a sequence generated by iTableSMD (Algorithm \ref{iTableSMD}) 
and define the following Lyapunov sequence 
\begin{eqnarray}
\begin{aligned}
\Psi_{k+1}:= &\eta_{k}\left(\Phi^{k+1} -\mathcal{V}(\Phi)\right)  + \left(1-\eta_{k}\alpha-\eta_{k}\bar\gamma-\gamma_{k}-\frac{\epsilon}{3}\right)D^N_{\psi}(A^{k},A^{k+1})\\
&+\eta_{k}\left(\frac{\bar\gamma}{2}+\frac{\epsilon}{3\eta_{k}}\right)D^N_{\psi}(A^{k-1},A^{k})+\frac{\eta_{k}}{2\tau \bar\gamma}\Gamma_{k+1}, \label{lyapunov_function}
\end{aligned}
\end{eqnarray}
where  $\gamma_{k}=|\alpha_{k}-\beta_{k}|M_{2}$.
Then, for all $k\in\mathbb{N}$, we have
\begin{eqnarray}
\begin{aligned}
\mathbb{E}_{k}[\Psi_{k+1}]\le\Psi_{k} -\frac{\epsilon}{3}(\mathbb{E}_{k}[D^N_{\psi}(A^{k},A^{k+1})]+D^N_{\psi}(A^{k-1},A^{k})+D^N_{\psi}(A^{k-2},A^{k-1})).
\end{aligned}\label{decent_inequality_01}
\end{eqnarray} 
\end{lemma}
\begin{proof}
From Lemma \ref{lemma_Phi_kk1}, it shows that
\begin{eqnarray}
\begin{aligned}
&\eta_{k}(\Phi^{k})-\mathcal{V}(\Phi))\\
\ge&\eta_{k}(\mathbb{E}_{k}[\Phi^{k+1}]-\mathcal{V}(\Phi)) +  (1-\eta_{k}\alpha-\eta_{k}\bar\gamma-\gamma_{k})\mathbb{E}_{k}[D^N_{\psi}(A^{k},A^{k+1})]\\
&+\frac{\eta_{k}}{2\bar\gamma\tau}(\mathbb{E}_{k}[\Gamma_{k+1}]-\Gamma_{k})
-\left(\delta-\epsilon+
\frac{\bar \gamma \eta_{k}}{2}+\frac{M_{2}^{2}(\alpha_{k}-\beta_{k})^{2}}{\gamma_{k}}\right)D^N_{\psi}(A^{k-1},A^{k})-\frac{\bar\gamma\eta_{k}}{2 }D^N_{\psi}(A^{k-2},A^{k-1}) \label{inequality_01}
\end{aligned}
\end{eqnarray}
Combining \eqref{lyapunov_function} with $\eta_{k}\le\eta_{k-1}$, we have
\begin{align*}
&\Psi_{k}-\mathbb{E}_{k}[\Psi_{k+1}]\\
=&\eta_{k-1}(\Phi^{k}-\mathcal{V}(\Phi))  +  \left(1-\eta_{k-1}\alpha-\eta_{k-1}\bar\gamma-\gamma_{k-1}-\frac{\epsilon}{3}\right)D^N_{\psi}(A^{k-1}, A^{k})-\frac{\eta_{k}}{2\tau\bar\gamma}\mathbb{E}_{k}[\Gamma_{k+1}]\\
&+\frac{\eta_{k-1}}{2\tau\bar\gamma}\Gamma_{k}+\eta_{k-1}\left(\frac{\bar\gamma}{2}+\frac{\epsilon}{3\eta_{k-1}}\right)D^N_{\psi}(A^{k-2},A^{k-1})-\eta_{k}(\mathbb{E}_{k}[\Phi^{k+1}]-\mathcal{V}(\Phi))\\
&-\eta_{k}\left(\frac{\bar\gamma}{2}+\frac{\epsilon}{3\eta_{k}}\right)D^N_{\psi}(A^{k-1},A^{k}) -  (1-\eta_{k}\alpha-\eta_{k}\bar\gamma-\gamma_{k}-\frac{\epsilon}{3})\mathbb{E}_{k}[D^N_{\psi}(A^{k},A^{k+1})]\\
\ge&\eta_{k}(\Phi^{k})-\mathcal{V}(\Phi)) + \left(1-\eta_{k-1}\alpha-\eta_{k-1}\bar\gamma-\gamma_{k-1}-\frac{\epsilon}{3}\right)D^N_{\psi}(A^{k-1},A^{k})-\frac{\eta_{k}}{2\tau\bar\gamma}\mathbb{E}_{k}[\Gamma_{k+1}]\\
&+\frac{\eta_{k}}{2\tau\bar\gamma}\Gamma_{k}+\eta_{k-1}\left(\frac{\bar\gamma}{2}+\frac{\epsilon}{3\eta_{k-1}}\right)D^N_{\psi}(A^{k-2},A^{k-1})
-\eta_{k}(\mathbb{E}_{k}[\Phi^{k+1}]-\mathcal{V}(\Phi))\\
&-\eta_{k}\left(\frac{\bar\gamma}{2}+\frac{\epsilon}{3\eta_{k}}\right)D^N_{\psi}(A^{k-1},A^{k}) -  \left(1-\eta_{k}\alpha-\eta_{k}\bar\gamma-\gamma_{k}-\frac{\epsilon}{3}\right)\mathbb{E}_{k}[D^N_{\psi}(A^{k},A^{k+1})]\\
\ge&\left(1-\delta-\eta_{k-1}\alpha-(\eta_{k-1}+\eta_{k})\bar\gamma-\gamma_{k-1}-\frac{M_{2}^{2}(\alpha_{k}-\beta_{k})^{2}}{\gamma_{k}}\right)D^N_{\psi}(A^{k-1},A^{k})\\
&+\frac{\epsilon}{3}(\mathbb{E}_{k}[D^N_{\psi}(A^{k},A^{k+1})]+D^N_{\psi}(A^{k-1},A^{k})+D^N_{\psi}(A^{k-2},A^{k-1})).
\end{align*}
Let $\gamma_{k}=|\alpha_{k}-\beta_{k}|M_{2}$, and assume $\gamma_{k}\ge \gamma_{k-1}$\footnote{In numerical experiments in \cite{PockS16, WangH23c}, there is $\alpha_{k}=c_{1}\frac{k-1}{k+2}$ and $\beta_{k}=c_{2}\frac{k-1}{k+2}$. Hence, we have this inequality holds.}, then we have

\begin{align*}
&\Psi_{k}-\mathbb{E}_{k}[\Psi_{k+1}]\\
\ge&\left(1-\delta-\eta_{k-1}\alpha-2\eta_{k-1}\bar{\gamma}-\gamma_{k-1}-\gamma_{k}\right)D^N_{\psi}(A^{k-1},A^{k})\\
&+\frac{\epsilon}{3}(\mathbb{E}_{k}[D^N_{\psi}(A^{k},A^{k+1})]+D^N_{\psi}(A^{k-1},A^{k})+D^N_{\psi}(A^{k-2},A^{k-1}))\\
\ge&\left(1-\delta-\eta_{k-1}\alpha-2\eta_{k-1}\bar{\gamma}-2\gamma_{k}\right)D^N_{\psi}(A^{k-1},A^{k})\\
&+\frac{\epsilon}{3}(\mathbb{E}_{k}[D^N_{\psi}(A^{k},A^{k+1})]+D^N_{\psi}(A^{k-1},A^{k})+D^N_{\psi}(A^{k-2},A^{k-1}))\\
\ge&\frac{\epsilon}{3}(\mathbb{E}_{k}[D^N_{\psi}(A^{k},A^{k+1})]+D^N_{\psi}(A^{k-1},A^{k})+D^N_{\psi}(A^{k-2},A^{k-1})),
\end{align*}
where the second and the last inequality follow from \eqref{inequality_01} and \eqref{stepsize_set}, respectively. This completes the proof.

\end{proof}

\begin{theorem}\label{subsequence_convergence}
Let $\{A_{n}^{k}\}_{k>0}$ with $n\in\{1,\dots,N\}$ be a sequence generated by iTableSMD algorithm. Then, the following statements hold.
\begin{itemize}
\item[(i)] The sequence $\{\mathbb{E}[\Psi_{k}]\}_{k\in\mathbb{N}}$ is nonincreasing.
\item[(ii)] $\sum\limits_{k=1}^{+\infty}\mathbb{E}[D^N_{\psi}(A^{k-1},A^{k})]<+\infty$,
and    the sequence $\{\mathbb{E}[ D^N_{\psi}(A^{k-1},A^{k}) ]\}$  
converges to zero. 
\item[(iii)] $\min\limits_{1\le k\le K}\mathbb{E}[D^N_{\psi}(A^{k-1},A^{k})]\le \frac{3\Psi_{1}}{\epsilon K}$.
\end{itemize}
\end{theorem}
\begin{proof}
\begin{itemize}
\item[(i)] This statement  follows   directly from Lemma \ref{lyapunov_descent} and $\epsilon>0$.
\item[(i)] By  summing \eqref{decent_inequality_01} from $k=0$ to a positive integer $K$, we have
\[
\sum_{k=1}^{K}\mathbb{E}[D^N_{\psi}(A^{k-1},A^{k})]\le \frac{3}{\epsilon}\mathbb{E}[\Psi_{1}-\Psi_{K+1}]\le \frac{3}{\epsilon}\Psi_{1},
\]
where the last inequality follows from $\Psi_{k}\ge 0$ for any $k>0$ due to \eqref{stepsize_set}. Taking the limit as $K\rightarrow+\infty$, we have $\sum_{k=1}^{+\infty}\mathbb{E}[D^N_{\psi}(A^{k-1},A^{k})]<+\infty$. Then we may deduce that  the sequence $\{\mathbb{E}[ D^N_{\psi}(A^{k-1},A^{k})]\}$ converges to zero. 
\item[(iii)] We have
\[
K\min_{1\le k\le K}\mathbb{E}[D^N_{\psi}(A^{k-1},A^{k})]\le\sum_{k=1}^{K}\mathbb{E}[D^N_{\psi}(A^{k-1},A^{k})]\le\frac{3}{\epsilon}\Psi_{1},
\]
which yields the desired result. 
\end{itemize}
This completes the proof.
\end{proof}

\subsection{Global convergence analysis}
In this subsection,  we present the analysis of iTableSMD algorithm with the expected squared distance of the subgradient and global convergence. In addition, We impose another stronger assumption on function $f$. 
\begin{assumption}\label{assume_03}
     
		The partial gradient $\nabla_{A_{i}} f$ is Lipschitz continuous with modulus $M_{1}$ on bounded sets of $\Pi_{n=1}^{N}\mathbb{R}^{I_{n}\times R}$. Namely, for any two points $A$ and $\hat{A}$, where  $A:=(A_{1},\dots,A_{i},\dots,A_{N})$, $\hat{A}:=(A_{1},\dots,A_{i-1},\hat{A}_{i},A_{i+1},\dots,A_{N})$ $\in\Pi_{n=1}^{N}\mathbb{R}^{I_{n}\times R}$, it shows that
		\[
		\|\nabla_{A_{i}}f(A)-\nabla_{A_{i}}f(\hat{A})\|\le M_1\|A_{i}-\hat{A}_{i}\|, \quad i=1,2,\dots,N.
		\]
\end{assumption}

 Under Definition \ref{vr_definition}  and the definition of  SAGA \cite{DefazioBL14} and SARAH \cite{NguyenLST17}, we have the following proposition.
\begin{proposition}\label{vr_gra}
Under Assumption \ref{assume_03}, we have the following two statements hold.
\begin{itemize}
\item[(i)] The SAGA gradient estimator  \cite{DefazioBL14} is defined as
\begin{eqnarray}
	\tilde{\nabla}^{SAGA}_{n}f(\underline{A}^{k}):=\frac{1}{{I_n}\left|\mathcal{F}_{n}^{k}\right|}(\sum_{j\in \mathcal{F}_{n}^{k}}\nabla_{A_{n}}f_{j}(\underline A^{k})-\nabla_{A_{n}}f_{j}((\phi^{k})^{j}))+\frac{1}{J_{n}}\sum_{i=1}^{J_{n}}\nabla_{A_{n}}f_{i}((\phi^{k})^{i}),\label{vr_gradient}
\end{eqnarray}
where $\underline A^{k}:=(A_{1}^{k},\dots, A_{n-1}^{k},\underline{A}_{n}^{k},A_{n+1}^{k},\dots, A_{N}^{k})$,  and the variable $(\phi^{k})^{i}$ follow the update rules   $(\phi^{k})^{i}=\underline{A}^{k-1}$ if $i\in \mathcal{F}_{n}^{k}$ and $(\phi^{k})^{i}=(\phi^{k-1})^{i}$ otherwise. A set of sampled mode-$n$ fibers is indexed by $\mathcal{F}_{n}^{k}\subset\{1,\dots,J_{n}\}$ with $|\mathcal{F}_{n}^{k}|=B$. Then it is variance reduced with
\[
\Gamma_{k+1}:=\frac{1}{B J_{n}}\sum_{i=1}^{J_{n}}\|\nabla_{A_{n}}f_{i}(\underline A^{k})-\nabla_{A_{n}}f_{i}((\phi^{k})^{i})\|_{*}^{2},
\]
\[
\Upsilon_{k+1}:=\frac{1}{\sqrt{B J_{n}}}\sum_{i=1}^{J_{n}}\|\nabla_{A_{n}}f_{i}(\underline A^{k})-\nabla_{A_{n}}f_{i}((\phi^{k})^{i}))\|_{*}.
\]
The constants $\tau=\frac{B}{2J_n}$, $V_{\Gamma}=2J_{n}+\frac{4J_{n}^2}{B}M_{1}^{2}$, $V_{1}=M_{1}^2, V_{2}=M_{1}$.
\item[(ii)] The SARAH gradient estimator \cite{NguyenLST17} which is defined as
\begin{eqnarray*}
&&\tilde{\nabla}^{SARAH}_{n}f(\underline{A}^{k})\\
&=&
\left\{\begin{array}{ll}
\nabla_{A_{n}} f(\underline{A}^{k}),\,\,&  \mbox{w.p.}\,\,\frac{1}{p},\\
\frac{1}{B}(\underset{j\in \mathcal{F}_{n}^{k}}{\sum}\nabla_{A_{n}} f_{j}(\underline{A}^{k})-\nabla_{A_{n}} f_{j}(\underline{A}^{k-1}))+\tilde{\nabla}^{SARAH}_{n}f(\underline{A}^{k-1}),&\mbox{otherwise.}
\end{array}\right.
\end{eqnarray*}
Here ``w.p. $\frac{1}{p}$'' means with probability $\frac{1}{p} \in (0, 1]$. 
Then it is variance reduced with
\begin{eqnarray*}
\Gamma_{k+1}=\|\tilde{\nabla}^{SARAH}_{n}f(\underline{A}^{k})-\nabla_{A_{n}} f(\underline{A}^{k})\|_{*}^{2},\quad \Upsilon_{k+1}=\|\tilde{\nabla}^{SARAH}_{n}f(\underline{A}^{k})-\nabla_{A_{n}} f(\underline{A}^{k})\|_{*},
\end{eqnarray*}
and constants $\tau=\frac{1}{p}$, $V_{1}=V_{\Gamma}=2M_{1}^{2}$, $V_{2}=2M_{1}$.
\end{itemize}
\end{proposition}

\begin{proof}

From the definition of SAGA stochastic gradient estimator $\tilde{\nabla}^{SAGA}f(\underline{A}^{k})$ and the Lipschitz continuity of $\nabla_{A_{i}}f(\cdot)$, it shows that
\begin{align*}
&\mathbb{E}_{k}\|\tilde{\nabla}^{SAGA}f(\underline{A}^{k})-\nabla f(\underline{A}^{k})\|_{*}^{2}\\
=&\mathbb{E}_{k}\|\frac{1}{I_n B}(\sum_{j\in \mathcal{F}_{n}^{k}}\nabla_{A_{n}}f_{j}(\underline A^{k})-\nabla_{A_{n}}f_{j}((\phi^{k})^{j}))+\frac{1}{J_{n}}\sum_{i=1}^{J_{n}}\nabla_{A_{n}}f_{i}((\phi^{k})^{i})-\nabla f(\underline{A}^{k})\|_{*}^{2}\\
\le&\frac{1}{B^{2}I^2_n}\mathbb{E}_{k}\sum_{j\in \mathcal{F}_{n}^{k}}\|\nabla_{A_{n}}f_{j}(\underline A^{k})-\nabla_{A_{n}}f_{j}((\phi^{k})^{j})\|_{*}^{2}\\
=&\frac{1}{BI^2_nJ_{n}}\sum_{i=1}^{J_{n}}\|\nabla_{A_{n}}f_{i}(\underline A^{k})-\nabla_{A_{n}}f_{i}((\phi^{k})^{i})\|_{*}^{2}\\
\le&\frac{1}{BJ_{n}}\sum_{i=1}^{J_{n}}\|\nabla_{A_{n}}f_{i}(\underline A^{k})-\nabla_{A_{n}}f_{i}((\phi^{k})^{i})\|_{*}^{2}\, ,
\end{align*}
where the last inequality follows from the fact that $\mathbb{E}_{k}\|y_{1}+\cdots+y_{t}\|_{*}^{2}=\mathbb{E}_{k}\|y_{1}\|_{*}^{2}+\cdots+\mathbb{E}_{k}\|y_{t}\|_{*}^{2}$ for any independent random variables $y_{i} (i=1,\dots,t)$ with $\mathbb{E}_{k}[y_{i}]=0$ for all $i$. Combined with Jensen's inequality, we can get 

\begin{align*}
&\mathbb{E}_{k}\|\tilde{\nabla}^{SAGA}_{n} f(\underline{A}^{k})-\nabla f(\underline{A}^{k})\|_{*}\\
\le&\sqrt{\mathbb{E}_{k}\|\tilde{\nabla}^{SAGA}_{n} f(\underline{A}^{k})-\nabla f(\underline{A}^{k})\|_{*}^{2}}\\
\le&\frac{1}{\sqrt{BJ_{n}}}\sqrt{\sum_{i=1}^{J_n}\|\nabla_{A_{n}}f_{i}(\underline A^{k})-\nabla_{A_{n}}f_{i}((\phi^{k})^{i})\|_{*}^{2}}\\
\le&\frac{1}{\sqrt{BJ_{n}}}\sum_{i=1}^{J_n}\|\nabla_{A_{n}}f_{i}(\underline A^{k})-\nabla_{A_{n}}f_{i}((\phi^{k})^{i})\|_{*}.
\end{align*}
We bound the MSE of the stochastic gradient estimator $\tilde{\nabla}^{SAGA}f(\cdot)$ as follows,
\begin{align*}
&\frac{1}{BJ_{n}}\sum_{i=1}^{J_{n}}\mathbb{E}_{k}\|\nabla_{A_{n}}f_{i}(\underline A^{k})-\nabla_{A_{n}}f_{i}((\phi^{k})^{i})\|_{*}^{2}\\
\le&\frac{1+\delta}{BJ_{n}}\mathbb{E}_{k}\sum_{i=1}^{J_{n}}\|\nabla_{A_{n}}f_{i}(\underline A^{k-1})-\nabla_{A_{n}}f_{i}((\phi^{k})^{i})\|_{*}^{2}+\frac{1+\delta^{-1}}{BJ_{n}}\mathbb{E}_{k}\sum_{i=1}^{J_{n}}\|\nabla_{A_{n}}f_{i}(\underline A^{k})-\nabla_{A_{n}}f_{i}(\underline A^{k-1})\|_{*}^{2}\\
\le&\frac{1+\delta}{BJ_{n}}(1-\frac{B}{J_{n}})\sum_{i=1}^{J_{n}}\|\nabla_{A_{n}}f_{i}(\underline A^{k-1})-\nabla_{A_{n}}f_{i}((\phi^{k-1})^{i})\|_{*}^{2}+\frac{1+\delta^{-1}}{B}M_{1}^{2} \mathbb{E}_{k}\|\underline A_n^{k}-\underline A_n^{k-1}\|^{2}\\
\le&\frac{1+\delta}{BJ_{n}}(1-\frac{B}{J_{n}})\sum_{i=1}^{J_{n}}\|\nabla_{A_{n}}f_{i}(\underline A^{k-1})-\nabla_{A_{n}}f_{i}((\phi^{k-1})^{i})\|_{*}^{2}+\frac{1+\delta^{-1}}{B}M_{1}^{2}\mathbb{E}_{k}[(1+\alpha_{k}^{2})\|A_n^{k}-A_n^{k-1}\|^{2}\\
&+\alpha_{k-1}^{2}\|A_n^{k-1}-A_n^{k-2}\|^{2}]\\
\le&\frac{1+\delta}{BJ_{n}}(1-\frac{B}{J_{n}})\sum_{i=1}^{J_n}\|\nabla_{A_{n}}f_{i}(\underline A^{k-1})-\nabla_{A_{n}}f_{i}((\phi^{k-1})^{i})\|_{*}^{2}+\frac{2+2\delta^{-1}}{B N}M_{1}^{2}[\|A^{k}-A^{k-1}\|^{2}+\|A^{k-1}-A^{k-2}\|^{2}],
\end{align*}
where the first inequality follows from $\|x-z\|_{*}^{2}\le (1+\delta)\|x-y\|_{*}^{2}+(1+\delta^{-1})\|y-z\|_{*}^{2}$.

Let $\Gamma_{k+1}:=\frac{1}{BJ_{n}}\sum_{i=1}^{J_{n}}\|\nabla_{A_{n}}f_{i}(\underline A^{k})-\nabla_{A_{n}}f_{i}((\phi^{k})^{i})\|_{*}^{2}$ and $\delta=\frac{B}{2J_{n}}$, it shows that 
\[
\begin{aligned}
\mathbb{E}_{k}\Gamma_{k+1}\le& (1+\frac{B}{2J_{n}})(1-\frac{B}{J_{n}})\Gamma_{k}+(2J_{n}+\frac{4J_{n}^2}{B})\frac{M_{1}^{2}}{N}[\|A^{k}-A^{k-1}\|^{2}+\|A^{k-1}-A^{k-2}\|^{2}]\\
\le&(1-\frac{B}{2J_{n}})\Gamma_{k}+(2J_{n}+\frac{4J_{n}^2}{B})\frac{M_{1}^{2}}{N}[\|A^{k}-A^{k-1}\|^{2}+\|A^{k-1}-A^{k-2}\|^{2}]. 
\end{aligned}
\] 
This proves the geometric decay of $\Gamma_{k}$ in expectation. Similar to Appendix B in \cite{DriggsTLDS2020}, we also have that the third condition holds in Definition \ref{vr_definition}. 
For the SARAH stochastic gradient estimator,  we can get the results directly similar to Lemma 5 in \cite{WangH23}. The proof of  Proposition \ref{vr_gra}\,(2) is completed.  
This completes the proof.
\end{proof}

\begin{corollary}\label{lemma_Phi_kk1_Rd}
If $\psi:=\frac{1}{2}\|\cdot\|^{2}$, the inequality from Lemma \ref{lyapunov_descent} becomes
\[
\mathbb{E}_{k}[\Psi_{k+1}]\le\Psi_{k} -\frac{\epsilon}{6}\left(\mathbb{E}_{k}[\|A^{k+1}-A^{k}\|^{2}]+\|A^{k}-A^{k-1}\|^{2}+\|A^{k-1}-A^{k-2}\|^{2}\right).
\]
\end{corollary}
Now we can prove the following result, which means that the subgradient of $\Phi(A^{k})$ is bounded.

\begin{lemma} \label{subgradient_bound}
Suppose that Assumptions \ref{assume_01}-\ref{assume_03} hold and the stepsize $\eta_{k}$ satisfies $0<\eta\le \eta_{k}$ and \eqref{stepsize_set}.The sequence $\{A_{1}^k,\dots,A_{N}^k\}$ generated by iTableSMD is bounded for all $k$. Define

\[
P_{n}^{k+1}:=\nabla_{A_n}f(A^{k+1})-\tilde{\nabla}_{A_n}f(\underline{A}^k)+\frac{1}{\eta_k}(\nabla \psi(\phi_n^k)-\nabla \psi(A_n^{k+1})), 
\]
where $P_{n}^{k+1} \in \partial_n \Phi\left(A^{k+1}\right)$ and $P^{k+1}=\left(P_{1}^{k+1}, P_{2}^{k+1}, \dots, P_{N}^{k+1}\right)$, implying that $P^{k+1} \in \partial \Phi\left(A^{k+1}\right)$.
Then, we can obtain
\[
\mathbb{E}_{k}\|P^{k+1}\| \le w(\mathbb{E}_{k}\|A^{k+1}-A^k\|+\|A^{k}-A^{k-1}\|+\|A^{k-1}-A^{k-2}\|) +\Upsilon_{k},
\]
where $w=\max\left\{M_{1}+\frac{M_{2}}{\eta }, V_{2}+\beta^{k}M_{1}+\frac{\alpha^{k}M_{2}}{\eta }, V_{2}\right\}$.
\end{lemma}

\begin{proof}
From the implicit deﬁnition of the proximal operator \eqref{xk_update} in the iTableSMD algorithm, we have  
\[
0\in \partial h_n(A_{n}^{k+1}) + \tilde{\nabla}_{A_n}f(\underline{A}^k)+\frac{1}{\eta_k}(\nabla \psi(A_n^{k+1})-\nabla \psi(\tilde{A}_n^k)), 
\]
where $\underline A^{k}:=(A_{1}^{k},\dots, A_{n-1}^{k},\underline{A}_{n}^{k},A_{n+1}^{k},\dots, A_{N}^{k})$. Combining it with  $\partial_n \Phi\left(A^{k+1}\right)\equiv\nabla_{A_n}f\left(A^{k+1}\right)+\partial h_n(A_{n}^{k+1})$, we have $P_{n}^{k+1}\in \partial_n \Phi(A^{k+1})$. Furthermore, in Problem \eqref{regular_GCPD}, with $h(A^{k+1})=\sum_{n=1}^{N}h_{n}(A_{n})$, we have $P^{k+1}=(P_{1}^{k+1}, P_{2}^{k+1}, \dots, P_{N}^{k+1})$, and it follows that $P^{k+1} \in \partial \Phi(A^{k+1})$, where $\partial \Phi(A^{k+1})\equiv\nabla f\left(A^{k+1}\right)+\partial h(A^{k+1})$.

All that remains is to bound the norm of $P^{k+1}$.  
Suppose $n=\xi^{k}$ at the $k$-th iteration. It shows that
\begin{align*}
&\mathbb{E}_{k}\|P^{k+1}\|\\
=&\mathbb{E}_{k}\|P_{\xi^{k}}^{k+1}\|\\
\le&\mathbb{E}_{k}\|\nabla_{A_{\xi^{k}}}f(A^{k+1})-\tilde{\nabla}_{A_{\xi^{k}}}f(\underline{A}^k)+\frac{1}{\eta_k}(\nabla \psi(\tilde{A}_{\xi^{k}}^k)-\nabla \psi(A_{\xi^{k}}^{k+1}))\|\\
\le&\mathbb{E}_{k}\|\nabla_{A_{\xi^{k}}}f(A^{k+1})-\tilde{\nabla}_{A_{\xi^{k}}}f(\underline{A}^k)\|+\frac{1}{\eta_{k}}\mathbb{E}_{k}\|\nabla \psi(\tilde{A}_{\xi^{k}}^k)-\nabla \psi(A_{\xi^{k}}^{k+1})\|\\
\le&\mathbb{E}_{k}\|\nabla_{A_{\xi^{k}}}f(A^{k+1})-\nabla_{A_{\xi^{k}}}f(\underline{A}^k)\|+\mathbb{E}_{k}\| \nabla_{A_{\xi^{k}}}f(\underline{A}^k)-\tilde{\nabla}_{A_{\xi^{k}}}f(\underline{A}^k)\|+\frac{1}{\eta_{k}}\mathbb{E}_{k}\|\nabla \psi(\tilde{A}_{\xi^{k}}^k)-\nabla \psi(A_{\xi^{k}}^{k+1})\|\\
\le&M_{1}\mathbb{E}_{k}\|A_{\xi^{k}}^{k+1}-\underline{A}_{\xi^{k}}^k\|+\Upsilon_{k}+V_{2}\|A^{k}-A^{k-1}\|+V_{2}\|A^{k-1}-A^{k-2}\|+\frac{M_{2}}{\eta_{k}}\mathbb{E}_{k}\|A_{\xi^{k}}^{k+1}-\tilde{A}_{\xi^{k}}^k\|\\
\le&M_{1}\mathbb{E}_{k}\|A^{k+1}-\underline{A}^k\|+\Upsilon_{k}+V_{2}\|A^{k}-A^{k-1}\|+V_{2}\|A^{k-1}-A^{k-2}\|+\frac{M_{2}}{\eta_{k}}\mathbb{E}_{k}\|A^{k+1}-\tilde{A}^k\|\\
\le& \left(M_{1}+\frac{M_{2}}{\eta_{k}}\right)\mathbb{E}_{k}\|A^{k+1}-A^k\|+\left(V_{2}+\beta^{k}M_{1}+\frac{\alpha^{k}M_{2}}{\eta^{k}}\right)\|A^{k}-A^{k-1}\| +V_{2}\|A^{k-1}-A^{k-2}\|+\Upsilon_{k}\\
\le& \left(M_{1}+\frac{M_{2}}{\eta }\right)\mathbb{E}_{k}\|A^{k+1}-A^k\|+\left(V_{2}+\beta^{k}M_{1}+\frac{\alpha^{k}M_{2}}{\eta }\right)\|A^{k}-A^{k-1}\| +V_{2}\|A^{k-1}-A^{k-2}\|+\Upsilon_{k}\\
\le& w(\mathbb{E}_{k}\|A^{k+1}-A^k\|+\|A^{k}-A^{k-1}\|+\|A^{k-1}-A^{k-2}\|) +\Upsilon_{k},
\end{align*}
where $w=\max\left\{M_{1}+\frac{M_{2}}{\eta }, V_{2}+\beta^{k}M_{1}+\frac{\alpha^{k}M_{2}}{\eta }, V_{2}\right\}$. This completes the proof.    
\end{proof}

\begin{lemma}\label{lemma_dist2}
Under the same conditions in Lemma \ref{subgradient_bound}, there exists a constant $\bar{w}>0$ such that
\[
\mathbb{E}[\mathrm{dist}(0,\partial \Phi(A^{k+1}))^{2}]\le \bar{w}\left(\mathbb{E}_{k}[\|A^{k+1}-A^{k}\|^{2}]+\|A^{k}-A^{k-1}\|^{2}+\|A^{k-1}-A^{k-2}\|^{2}\right) +3\mathbb{E}\Gamma_{k}.
\] 
\end{lemma}
\begin{proof}
From Lemma \ref{subgradient_bound}, it shows that
\begin{align*}
&\mathbb{E}_{k}\|P^{k+1}\|^{2}\\
\le&3\mathbb{E}_{k}\|\nabla_{A_{\xi^{k}}}f(A^{k+1})-\nabla_{A_{\xi^{k}}}f(\underline{A}^k)\|^{2}+3\mathbb{E}_{k}\|\nabla_{A_{\xi^{k}}}f(\underline{A}^k)-\tilde{\nabla}_{A_{\xi^{k}}}f(\underline{A}^k)\|^{2}+\frac{3}{\eta_{k}}\mathbb{E}_{k}\|\nabla \psi(\tilde{A}_{\xi^{k}}^k)-\nabla \psi(A_{\xi^{k}}^{k+1})\|^{2}\\
\le&3M_{1}^{2}\mathbb{E}_{k}\|A^{k+1}-\underline{A}^k\|^2+3\Gamma_{k}+3V_{1}\|A^{k}-A^{k-1}\|^{2}+3V_{1}\|A^{k-1}-A^{k-2}\|^{2}+\frac{3M_{2}^{2}}{\eta_{k}}\mathbb{E}_{k}\|A^{k+1}-\tilde{A}^k\|^{2}\\
\le& \left(6M_{1}^{2}+\frac{6M_{2}^{2}}{\eta_{k}}\right)\mathbb{E}_{k}\|A^{k+1}-{A}^k\|^{2}+\left(3V_{1}+6\beta_{k}^{2}M_{1}^{2}+\frac{6\alpha_{k}^{2}M_{2}^{2}}{\eta_{k}}\right)\|A^{k}-A^{k-1}\|^{2}\\
& +3V_{1}\|A^{k-1}-A^{k-2}\|^{2}+3\Gamma_{k}\\
\le&\left(6M_{1}^{2}+\frac{6M_{2}^{2}}{\eta }\right)\mathbb{E}_{k}\|A^{k+1}-{A}^k\|^{2}+\left(3V_{1}+6\beta_{k}^{2}M_{1}^{2}+\frac{6\alpha_{k}^{2}M_{2}^{2}}{\eta }\right)\|A^{k}-A^{k-1}\|^{2}\\
& +3V_{1}\|A^{k-1}-A^{k-2}\|^{2}+3\Gamma_{k}\\
\le& \bar{w}\left(\mathbb{E}_{k}[\|A^{k+1}-A^{k}\|^{2}]+\|A^{k}-A^{k-1}\|^{2}+\|A^{k-1}-A^{k-2}\|^{2}\right) +3\mathbb{E}\Gamma_{k},
\end{align*}
where $\bar{w}:=\max\left\{6M_{1}^{2}+\frac{6M_{2}^{2}}{\eta}, 3V_{1}+6\beta_{k}^{2}M_{1}^{2}+\frac{6\alpha_{k}^{2}M_{2}^{2}}{\eta }, 3V_{1}\right\}$. Through $\mathrm{dist}\left(0,\partial \Phi\left(A^{k+1}\right)\right)^{2}\le\|P^{k+1}\|^{2}$ and taking full expectation on both sides, it shows that
\[
\mathbb{E}[\mathrm{dist}(0,\partial \Phi(A^{k+1}))^{2}]\le \bar{w}\left(\mathbb{E}_{k}[\|A^{k+1}-A^{k}\|^{2}]+\|A^{k}-A^{k-1}\|^{2}+\|A^{k-1}-A^{k-2}\|^{2}\right) +3\mathbb{E}\Gamma_{k}.
\]
This completes the proof.
\end{proof}

Using Lemma \ref{lemma_dist2}, we can show the convergence rate of the expected squared distance of the subgradient to $0$.
\begin{theorem}\label{subgradient_rate}
Assume that Assumptions \ref{assume_01}-\ref{assume_03} hold, and the stepsize satisfies $0<\eta\le\eta_{k}$ and \eqref{stepsize_set}.  Let $\{A^{k}\}_{k\in\mathbb{N}}$ generated by iTableSMD be bounded for all $k$. Then there exists $0<\sigma<\epsilon/6$ such that
\[
\mathbb{E}[\mathrm{dist}(0,\partial \Phi(A^{\hat{k}}))^{2}]\le \frac{\bar{w}}{(\epsilon/6-\sigma)K}(\mathbb{E}\Psi_{1}+\frac{\epsilon/2-3\sigma}{\tau\bar{w}}\mathbb{E}\Gamma_{1})=\mathcal{O}(1/K),
\]
where $\hat{k}$ is drawn from $\{2, \dots, K+1\}$. In other words, it takes at most $\mathcal{O}(\epsilon^{-2})$ iterations  in expectation  to obtain an $\epsilon$-stationary point (see Definition \ref{stationary-point}) of $\Phi$. 
\end{theorem}
\begin{proof}
From Corollary \ref{lemma_Phi_kk1_Rd} and Lemma \ref{lemma_dist2}, it shows that
	\begin{align*}
		&\mathbb{E}[\Psi_{k}-\Psi_{k+1}]\\
		\ge&\frac{\epsilon}{6}\mathbb{E}[\|A^{k+1}-A^{k}\|^{2}+\|A^{k}-A^{k-1}\|^{2}+\|A^{k-1}-A^{k-2}\|^{2}]\\
		\ge&\sigma\mathbb{E}[\|A^{k+1}-A^{k}\|^{2}+\|A^{k}-A^{k-1}\|^{2}+\|A^{k-1}-A^{k-2}\|^{2}]
		+\frac{\epsilon/6-\sigma}{\bar{w}}\mathbb{E}[\mathrm{dist}(0,\partial \Phi(A^{k+1}))^{2}]\\ &-\frac{\epsilon/2-3\sigma}{\bar{w}}\mathbb{E}\Gamma_{k}\\
		\ge&\sigma\mathbb{E}[\|A^{k+1}-A^{k}\|^{2}+\|A^{k}-A^{k-1}\|^{2}+\|A^{k-1}-A^{k-2}\|^{2}]
		+\frac{\epsilon/6-\sigma}{\bar{w}}\mathbb{E}[\mathrm{dist}(0,\partial \Phi(A^{k+1}))^{2}]\\
		&+\frac{\epsilon/2-3\sigma}{\tau\bar{w}}\mathbb{E}[\Gamma_{k+1}-\Gamma_{k}] -\frac{(\epsilon/2-3\sigma)V_{\Gamma}}{\tau\bar{w}}\mathbb{E}[\|A^{k}-A^{k-1}\|^{2}+\|A^{k-1}-A^{k-2}\|^{2}]\\
		\ge&\sigma\mathbb{E}[\|A^{k+1}-A^{k}\|^{2}+\|A^{k}-A^{k-1}\|^{2}+\|A^{k-1}-A^{k-2}\|^{2}]
		+\frac{\epsilon/6-\sigma}{\bar{w}}\mathbb{E}[\mathrm{dist}(0,\partial \Phi(A^{k+1}))^{2}]\\
		&+\frac{\epsilon/2-3\sigma}{\tau\bar{w}}\mathbb{E}[\Gamma_{k+1}-\Gamma_{k}] -\frac{(\epsilon/2-3\sigma)V_{\Gamma}}{\tau\bar{w}}\mathbb{E}[\|A^{k+1}-A^{k}\|^{2}+\|A^{k}-A^{k-1}\|^{2}+\|A^{k-1}-A^{k-2}\|^{2}],
	\end{align*}
	where the third inequality follows from \eqref{Gamma_k1_k} in Definition \ref{vr_definition}. If we let $\sigma=\frac{(\epsilon/2-3\sigma)V_{\Gamma}}{\tau\bar{w}}$, i.e., $\sigma=\frac{\epsilon V_{\Gamma}}{2(3V_{\Gamma}+\tau\bar{w})}$, it shows that
	\[
	\mathbb{E}[\Psi_{k}-\Psi_{k+1}]\ge \frac{\epsilon/6-\sigma}{\bar{w}}\mathbb{E}[\mathrm{dist}(0,\partial \Phi(A^{k+1}))^{2}]+\frac{\epsilon/2-3\sigma}{\tau\bar{w}}\mathbb{E}[\Gamma_{k+1}-\Gamma_{k}].
	\]
	Summing up $k=1$ to $K$, we have
	\[
	\mathbb{E}[\Psi_{1}-\Psi_{K+1}]\ge \frac{\epsilon/6-\sigma}{\bar{w}}\sum_{k=1}^{K}\mathbb{E}[\mathrm{dist}(0,\partial \Phi(A^{k+1}))^{2}]+\frac{\epsilon/2-3\sigma}{\tau\bar{w}}\mathbb{E}[\Gamma_{K+1}-\Gamma_{1}],
	\]
	which means  there exists a  $\hat{k} \in\{2,\dots,K+1\}$ such that
	\begin{align*}
		\mathbb{E}[\mathrm{dist}(0,\partial \Phi(A^{\hat{k}}))^{2}]\le&\frac{1}{K}\sum_{k=1}^{K}\mathbb{E}[\mathrm{dist}(0,\partial \Phi(A^{k+1}))^{2}]\\
		\le&\frac{\bar{w}}{(\epsilon/6-\sigma)K}(\mathbb{E}[\Psi_{1}-\Psi_{K+1}]+\frac{\epsilon/2-3\sigma}{\tau\bar{w}}\mathbb{E}[\Gamma_{1}-\Gamma_{K+1}])\\
		\le&\frac{\bar{w}}{(\epsilon/6-\sigma)K}(\mathbb{E}\Psi_{1}+\frac{\epsilon/2-3\sigma}{\tau\bar{w}}\mathbb{E}\Gamma_{1}).
	\end{align*}
	This completes the proof.
\end{proof}
We define the  set of cluster points of $\{A^{k}\}_{k\in\mathbb{N}}$ as
\begin{eqnarray}\label{limit_set}
	\begin{aligned}
		\Omega(A^{0}):=&\{A^{*}:\exists \text{ an increasing sequence of integers }\{k_{l}\}_{l\in\mathbb{N}} \text{ such that }A^{k_{l}}\rightarrow A^{*} \text{ as } l\rightarrow+\infty \}.
	\end{aligned}
\end{eqnarray}

\begin{lemma}\label{statements_lemma}
Suppose that Assumptions \ref{assume_01} to \ref{assume_03} hold,   the step $\eta_{k}$ satisfies $0<\eta\le \eta_{k}$ and \eqref{stepsize_set}.  Then the following statements hold. 
\begin{itemize}
\item[(1)] $\sum_{k=0}^{\infty}\|A^{k+1}-A^{k}\|^{2}<+\infty$ a.s., and $\lim_{k\rightarrow+\infty}\|A^{k+1}-A^{k}\|\rightarrow0$ a.s.
\item[(2)] $\mathbb{E}[\Phi(A^k)]\rightarrow\Phi^{*}$, where $\Phi^{*}\in[\mathcal{V}(\Phi),+\infty)$  with $\mathcal{V}(\Phi):=\inf_{A} \Phi(A)$, and $\mathbb{E}\Phi(A^{*})=\Phi_{*}$ for all $A^{*}\in\Omega(A_{0})$.
\item[(3)] $\mathbb{E}[\mathrm{dist}(0,\partial \Phi(A^{k}))]\rightarrow0$. Moreover, the set 
 $\Omega(A^{0})$ is nonempty, and $\mathbb{E}[\mathrm{dist}(0,\partial \Phi(A^{*}))]=0$  for all $A^{*}\in\Omega(A_{0})$.
\item[(4)] $\mathrm{dist}(A^{k},\Omega(A_{0}))\rightarrow0$ a.s., and $\Omega(A_{0})$ is a.s. compact and connected.
\end{itemize}
\end{lemma}	
\begin{proof}
The proof of the above statements is similar to that of Lemma 9 in \cite{WangH23}, so we omit the details here for simplicity.
\end{proof}

The following lemma is from \cite{DriggsTLDS2020}, which is analogous to the Uniformized K{\L} property of \cite{BolteST14} and allows us to apply the K{\L} inequality.
\begin{lemma}\label{E_KL_inequality}
Assuming $\{A^{k}\}_{k\in\mathbb{N}}$ is a bounded sequence of iterates for all $k$ generated by the iTableSMD algorithm using a variance-reduced gradient estimator (see Definition \ref{vr_definition}). Let $\Phi$ be a semialgebraic function satisfying the K{\L} property \cite{BolteST14} with exponent $\theta$. Then there exists an index $\bar{k}$ and a desingularizing function $\phi(r) = ar^{1-\theta}$ with $a>0$, $\theta\in[0,1)$  so that the following bound holds almost surely (a.s.),
	\begin{eqnarray}
		\phi'(\mathbb{E}[\Phi(A^k)-\Phi_{k}^{*}])\mathbb{E}\mbox{dist}(0,\partial \Phi(A^{k}))\ge 1,\,\,\forall k>\bar{k},
	\end{eqnarray}
	where $\Phi_{k}^{*}$ is a nondecreasing sequence  converging to $\mathbb{E}\Phi(A^{*})$ for some $A^{*}\in\Omega(A_{0})$.
\end{lemma}

Now we give the global convergence result of the iTableSMD algorithm in the following theorem which can be proved by the above lemmas and we omit the proof details here. See \cite{DriggsTLDS2020, WangH23} for details. 

\begin{theorem}\label{global_convergence}
Suppose that Assumptions \ref{assume_01}-\ref{assume_03} hold, the step $\eta_{k}$ satisfies $0<\eta\le \eta_{k}$ and \eqref{stepsize_set}. Let $\{A^{k}\}_{k\in\mathbb{N}}$ be the sequence generated by the iTableSMD algorithm which is assumed to be bounded. If the optimization function $\Phi$ is a semialgebraic function that satisﬁes the K{\L} property with exponent $\theta \in [0, 1)$ (see Lemma \ref{E_KL_inequality}), then either the point $A^{k}$ is a critical point after a ﬁnite number of iterations or the sequence $\{A^{k}\}_{k\in\mathbb{N}}$ almost surely satisﬁes the ﬁnite length property in expectation, namely,
\[
\sum_{k=0}^{+\infty}\mathbb{E}\|A^{k+1}-A^{k}\|<+\infty.
\]
\end{theorem}

\newpage
\section{Numerical experiments}\label{numercial_experiments}

In this section, we evaluate the proposed iTableSMD (Algorithm \ref{iTableSMD}) using synthetic datasets as well as multiple real-world datasets. We aim to demonstrate its superior efficiency through comparisons with state-of-the-art algorithms as our main baseline. The first one is an entry-sampling based stochastic non-Euclidean CP decomposition  optimization algorithm, namely, GCP-OPT, proposed in \cite{Kolda2019, GCP2020}. The GCP-OPT method is implemented in Tensor Toolbox and ``Adam'’ is selected as the optimization solver. The sampling rule of GCP-OPT is the default ``uniform'’ setting for dense tensors unless specified for particular examples. The Second one is a tensor
fiber-sampling based flexible stochastic mirror descent framework denoted by SmartCPD \cite{Pu2022}.

We perform experiments involving low-rank GCP decomposition with nonnegative constraints  $A_{n}\ge0$  for $n=1,\dots,N$. Our focus extends to three distinct synthetic data distributions: Gamma, Poisson, and Bernoulli. Additionally, we incorporate several real datasets into our analysis, including the Enron emails dataset \cite{shetty2004enron}, six months of Uber pickup data \cite{frosttdataset} in New York City, both of which are characterized by integer counts following the Poisson distribution. 
Each element in the Enron emails dataset represents the sender-receiver-word, and the values are counts of words. 
Each element in the  Uber pickup data represents the date-latitude-longitude of pickup, and the values are counts of pickups.
We also test the tags from the Flickr dataset \cite{gorlitz2008}, where non-zero values are binary, indicating user tagging of images on a given day. In synthetic experiments, we generate third-order tensors with different sizes and ranks and we do not require each dimension of the tensor to remain the same. For the fiber-sampling based algorithms, in each iteration, iTableSMD and SmartCPD sample $2R$ fibers, while for entry-sampling algorithm, GCP-OPT samples $2\frac{\sum_{i=1}^N I_{n}}{N}R$ entries. The generating function of Bregman distance is chosen as $\psi(a)= a\log a$ in the update of iTableSMD and SmartCPD.  The numerical experiment performance is measured by the cost function value (denotes by “NRE”) and  the mean squared error (MSE). The MSE of the latent matrices is used as a performance metric, which is defined as 
$$
\mathrm{MSE}=\min _{\pi(r) \in [R]} \frac{1}{R} \sum_{r=1}^R\left\|\frac{\boldsymbol{A}_{n}(:, \pi(r))}{\left\|\boldsymbol{A}_{n}(:, \pi(r))\right\|_2}-\frac{\bar{\boldsymbol{A}}_{n}(:, r)}{\left\|\bar{\boldsymbol{A}}_{n}(:, r)\right\|_2}\right\|^2,
$$
where $\bar{\boldsymbol{A}}_{n}$ denotes the estimate of original matrix $\boldsymbol{A}_{(n)}$ and $\{\pi(1), \ldots, \pi(R)\}$ represents a permutation of the set $[R]=\{1, \ldots, R\}$, which is used to fix the intrinsic column permutation in CP decomposition. 



\subsection{Synthetic data experiments} 
\subsubsection{Gamma distribution} 
In this subsection, 
we compute the GCP decomposition on two artificial three-way tensors of size $150\times100\times150$ and $300\times400\times300$ with different ranks using the gamma loss function: $x/m+\log (m)$. In practice, we use the constraint $m \ge 0$ and replace $m$ with $m + \epsilon$ (e.g., $\epsilon=10^{-9}$) in the loss function to prevent function values or gradients from becoming $\pm \infty$. 
Namely, $f(m \,; x) = x/(m+\epsilon)+\log (m+\epsilon)$.
With nonnegative constraints on the factor matrices, the latent factors $A_1$, $A_2$, and $A_3$ are drawn from i.i.d.uniform distribution between 0 and $A_{max}$, where the $A_{max}=0.5$ is a positive constant. The observed nonnegative data tensor $\mathcal{X}$ is generated following the gamma distribution, i.e. $\underline{\mathcal{X}}_{i}\sim Gamma (\underline{\mathcal{M}}_{i})$.  Namely, we focus on 
\begin{eqnarray*}\label{GCP_gamma}
 \begin{aligned}
\min _{A_1, A_2,  A_3} &\quad  \frac{1}{I^N} \sum_{i \in \mathcal{I}} \underline{\mathcal{X}}_{i}/ (\underline{\mathcal{M}}_{i}+\epsilon)+\log (\underline{\mathcal{M}}_{i}+\epsilon)+\sum_{n=1}^3 h_n\left(A_n\right) \\
 \text { s.t. } & \quad \underline{\mathcal{M}}_{i}=\sum_{r=1}^R \prod_{n=1}^3 {A}_n\left(i_n, r\right), \forall\, i \in \mathcal{I}, \\
\end{aligned}   
\end{eqnarray*}

We set the inertial parameters as $\alpha^{k}=\frac{3(k-1)}{5(k+2)}$, $\beta^{k}=\frac{4(k-1)}{5(k+2)}$ for simplicity\footnote{Theoretically, the inequality \eqref{extra_ineq}   is required. However, it is time-consuming to check this inequality in the numerical experiments. Therefore, we directly set $\beta^{k}=\frac{4(k-1)}{5(k+2)}$. Our numerical experiments show that iTableSMD always converges with this $\beta^k$.} and set the stepsize as $\eta^{k}=0.1$ to verify the difference between  SmartCPD with SGD and SAGA, GCP-OPT, and iTableSMD with SGD and SAGA. Our numerical results for two synthetic data are presented in Figure \ref{syn_gamma_exp}. 

The first synthetic experiment, visualized in the top row of the figure, captures the algorithmic performance across a tensor of dimensions $150\times100\times150$, evaluated at varying ranks. 
Notably,  the iTableSMD, especially when coupled with the SAGA, achieves a rapid improvement in MSE. For instance, in Figure \ref{syn_gamma_exp}(a), iTableSMD-SAGA reduces the MSE to below $10^{-6}$ within an average time of fewer than 3 seconds, outpacing SmartCPD, which requires a minimum of 6 seconds, and GCP-OPT, which exceeds 10 seconds to attain comparable MSE reductions. 


The second row in Figure \ref{syn_gamma_exp} evaluates the efficacy of various algorithms on a tensor with increased dimensions $300\times400\times300$. It is clear that as the tensor size increases, the iTableSMD method gains a lower MSE within the same time compared to SmartCPD and GCP-OPT. The noticeable improvement over the SmartCPD method indicates that the inertial acceleration framework of iTableSMD has a significant impact on its performance, helping to achieve faster convergence. 


 
\begin{figure}[!htb]
	\setlength\tabcolsep{2pt}
	\centering
	\begin{tabular}{ccc}
		\includegraphics[width=0.32\textwidth]{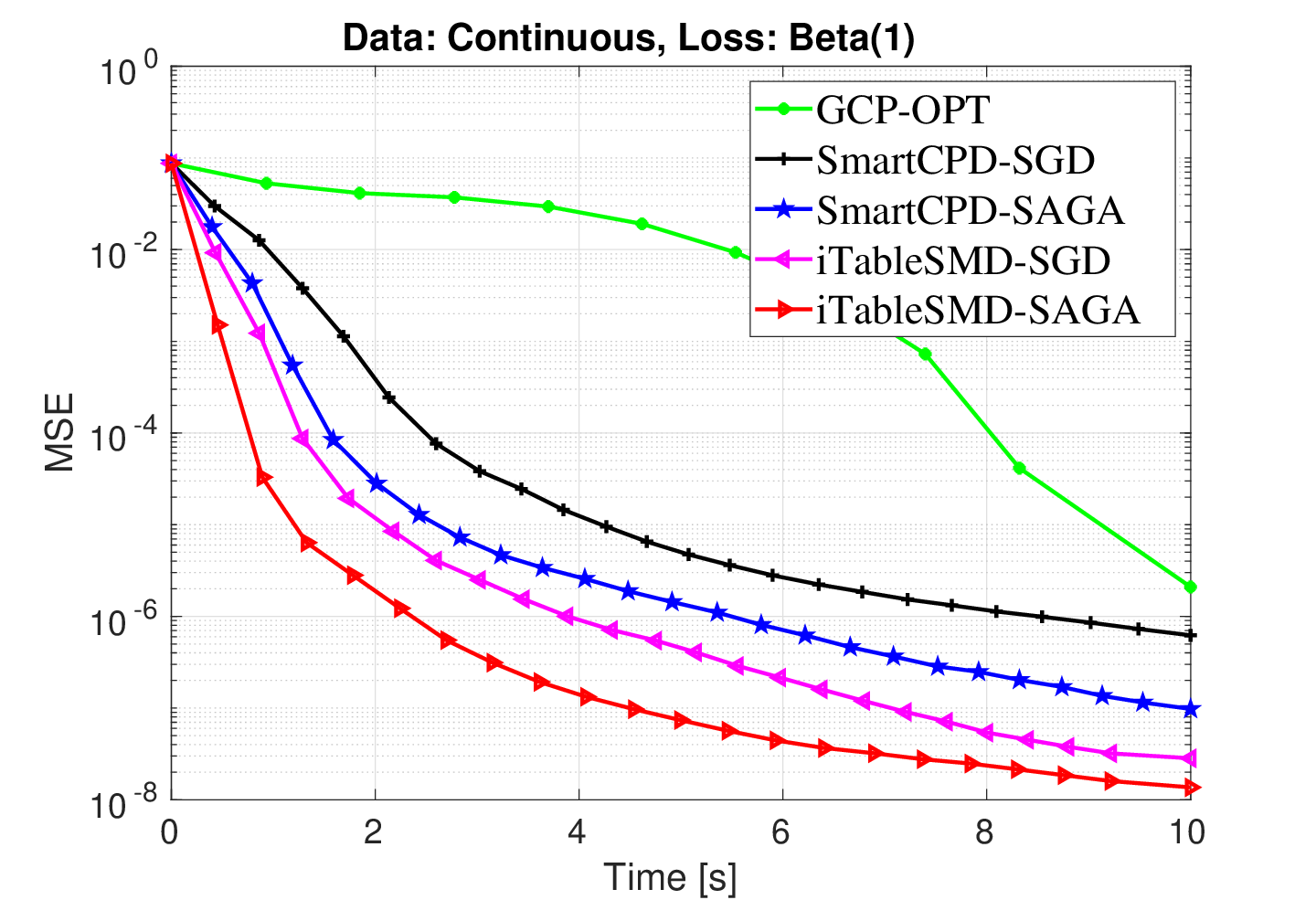}\label{a}&
		\includegraphics[width=0.32\textwidth]{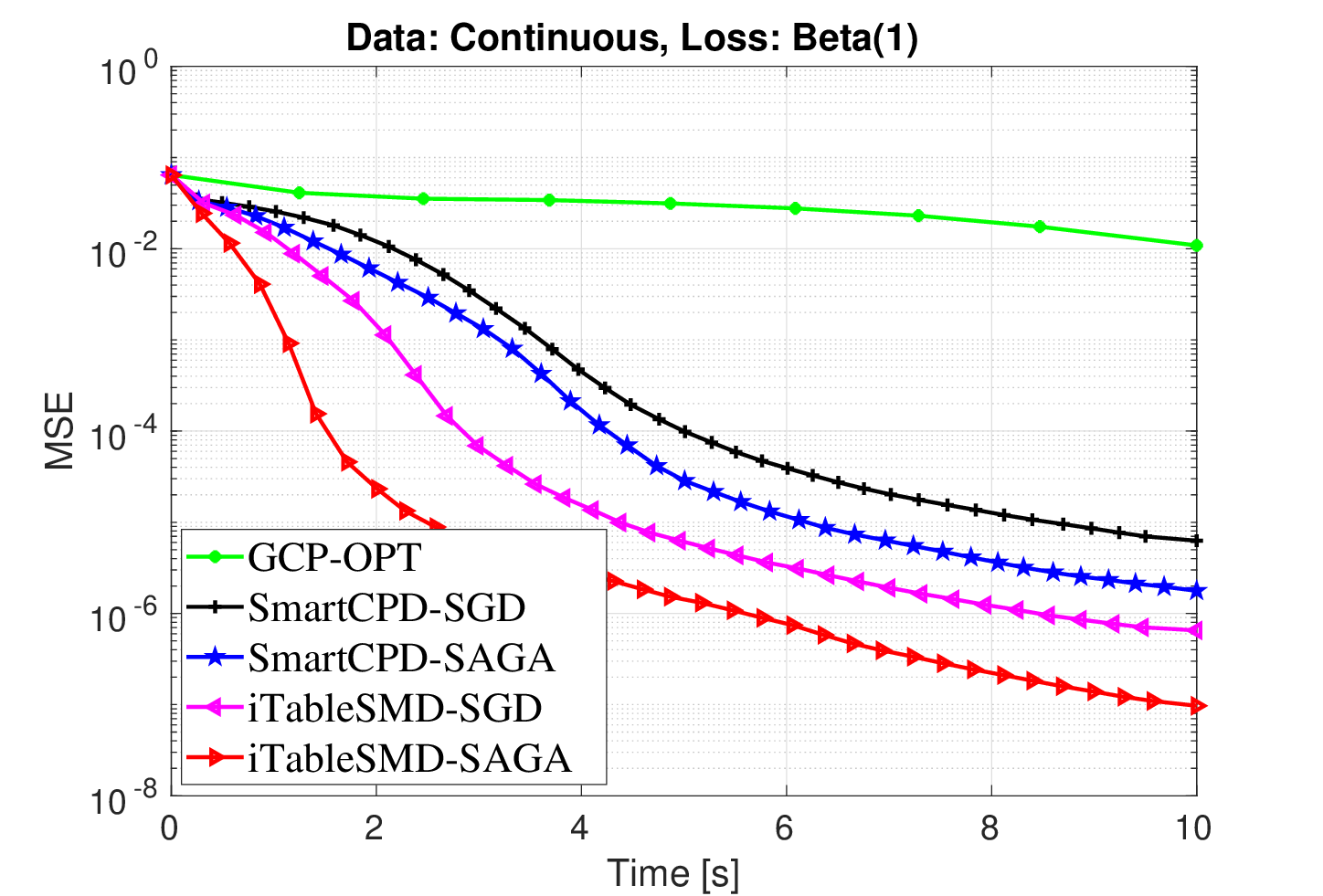}&
        \includegraphics[width=0.32\textwidth]{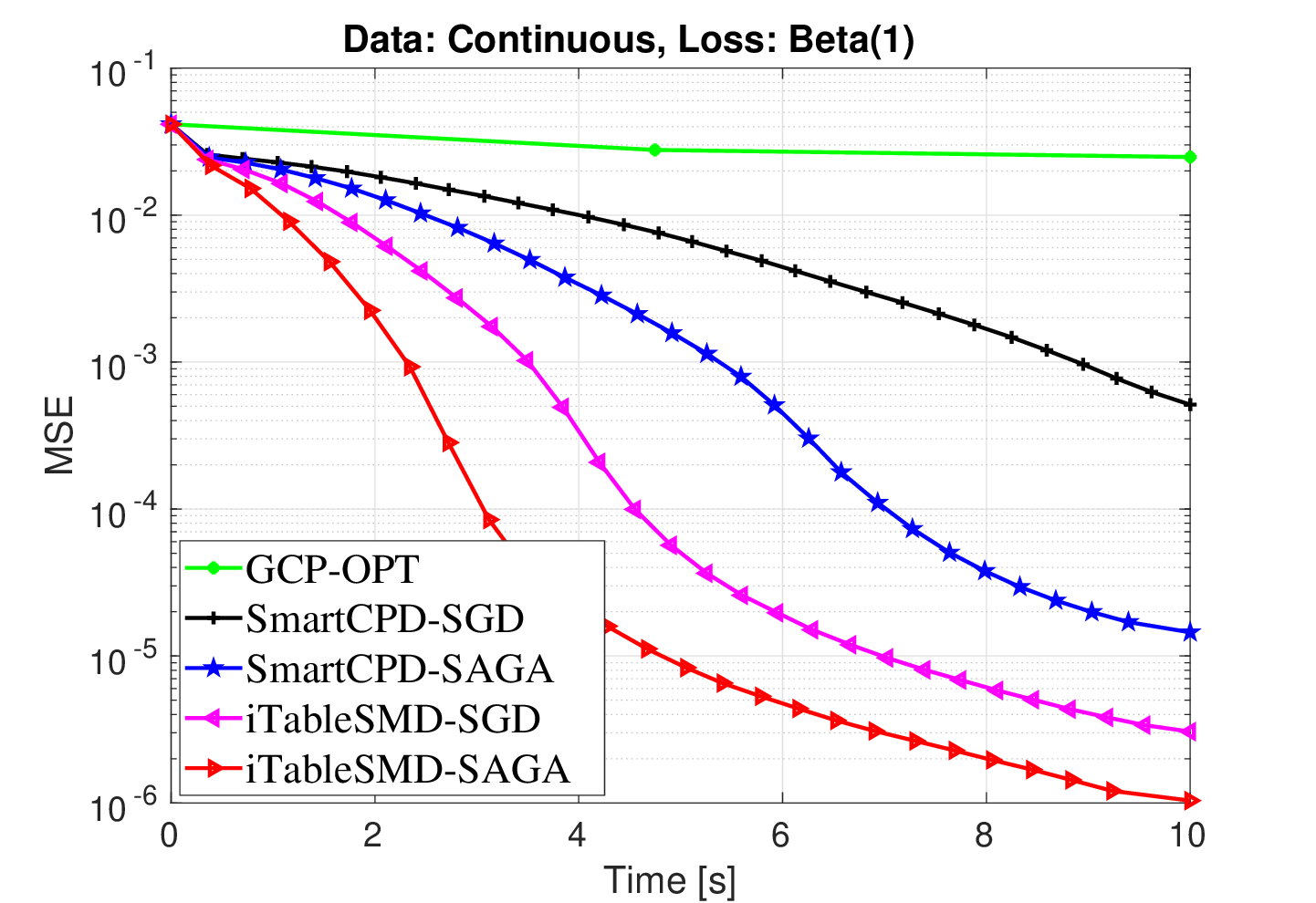}\\
		(a) $150\times100\times150$, $R=15$&(b) $150\times100\times150$,  $R=20$&(c) $150\times100\times150$, $R=30$\\
  \includegraphics[width=0.32\textwidth]{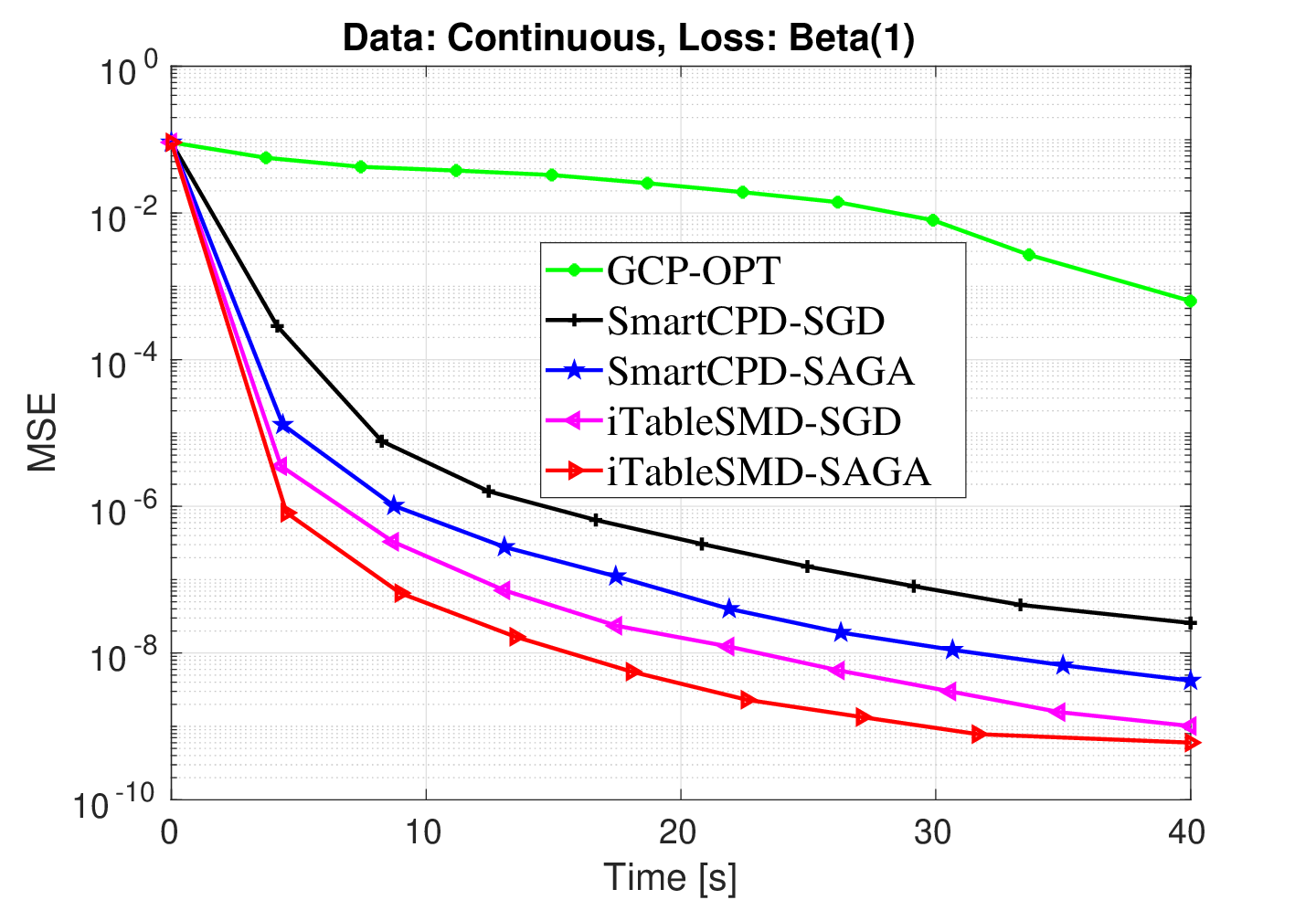}&
		\includegraphics[width=0.32\textwidth]{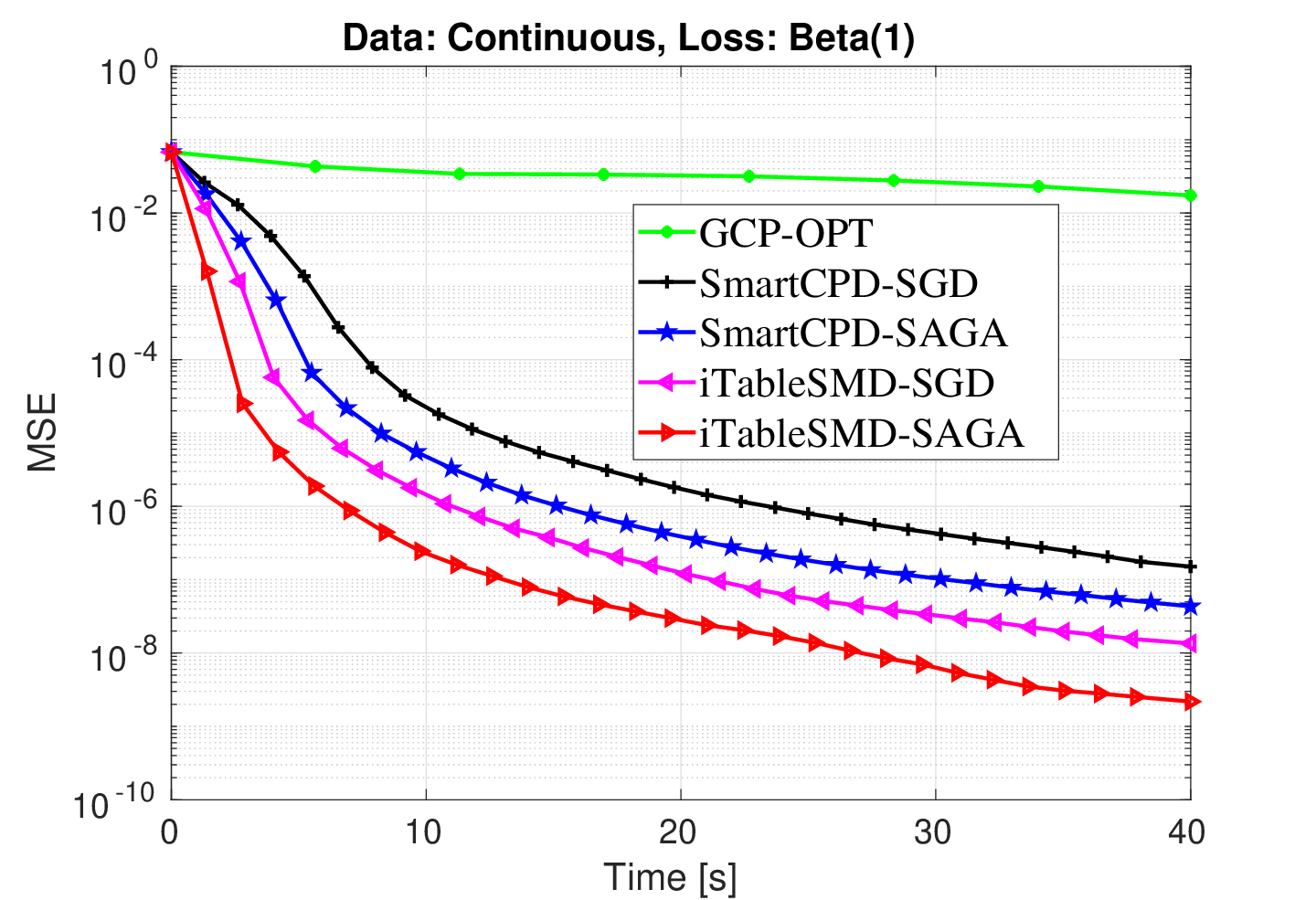}&
        \includegraphics[width=0.32\textwidth]{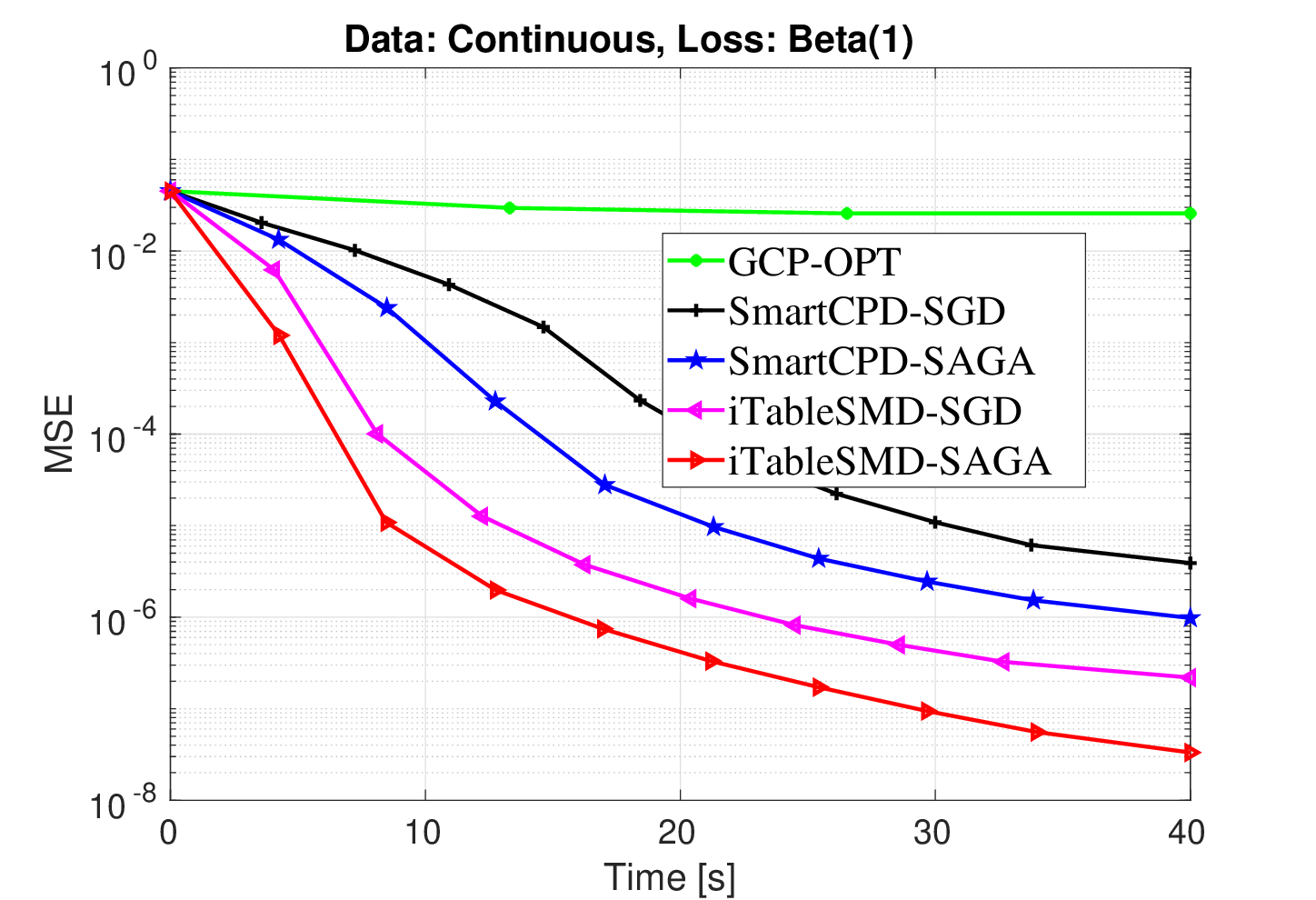}\\
		(d) $300\times400\times300$, $R=15$&(e) $300\times400\times300$,  $R=20$&(f) $300\times400\times300$, $R=30$
	\end{tabular}
	\caption{Numerical experiments for Gamma distribution on synthetic datasets. 
	}
	\label{syn_gamma_exp}
\end{figure}

\subsubsection{Poisson distribution}

We next evaluate the performance on two synthetic count data tensors with the size of $150\times51\times152$ and $300\times300\times300$ . For simplicity in our experiments, we set inertial parameters $\alpha^{k}=\frac{3(k-1)}{5(k+2)}$ and $\beta^{k}=\frac{4(k-1)}{5(k+2)}$ using a formula based on the iteration $k$, and chose a stepsize as $\eta^{k}=0.2$. The loss function used is a modification of the standard Poisson log-likelihood and is defined as $f(m \,; x) = m-xlog(m+\epsilon)$.  We initialized the factor matrices $A_1$, $A_2$, and $A_3$ with values uniformly distributed between 0 and a set maximum $A_{max}$, here chosen as 0.5. The observed count data tensor $\mathcal{X}$ is generated following the Poisson distribution, i.e.
$\underline{\mathcal{X}}_{i}\sim Poisson (\underline{\mathcal{M}}_{i})$.  
\begin{eqnarray*}
 \begin{aligned}
\min _{A_1, A_2,  A_3} &\quad  \frac{1}{I^N} \sum_{i \in \mathcal{I}}  \underline{\mathcal{M}}_{i}+\underline{\mathcal{X}}_{i}\log (\underline{\mathcal{M}}_{i}+\epsilon)+\sum_{n=1}^3 h_n\left(A_n\right) \\
 \text { s.t. } & \quad \underline{\mathcal{M}}_{i}=\sum_{r=1}^R \prod_{n=1}^3 {A}_n\left(i_n, r\right), \forall\, i \in \mathcal{I}, \\
\end{aligned}   
\end{eqnarray*}

In Figure \ref{syn_poisson_exp}, we see the results of numerical experiments on synthetic datasets modeled with Poisson distribution for tensors of two sizes, with varying tensor ranks. 
For the smaller tensor ($150\times51\times152$), as the rank increases from $R=10$ to $R=20$,  the iTableSMD-SAGA maintains a lower MSE compared to others, indicating a more efficient performance. In particular, for R = 20 in Figure \ref{syn_poisson_exp}(c), iTableSMD-SAGA reduces the MSE significantly faster than the other methods in 10 seconds. 
For the larger tensor ($300\times300\times300$), in Figures \ref{syn_poisson_exp}(d)-(f), as the rank grows, the MSE tends to decrease at a slower rate for all methods. However, the iTableSMD-SAGA still shows a consistent advantage, reaching lower MSEs quicker than the competing algorithms, which becomes more notable as the rank moves to 30 and beyond. From figures \ref{syn_poisson_exp}(b)-(e),  compared with figure \ref{syn_gamma_exp}, we can see that SmartCPD can not always perform better than GCP-OPT, since the type of data distribution may affect its effectiveness. However, iTableSMDs continue to show superior performance in terms of iteration speed, regardless of the change in data distribution.

\begin{figure}[!htb]
	\setlength\tabcolsep{2pt}
	\centering
	\begin{tabular}{ccc}
		\includegraphics[width=0.32\textwidth]{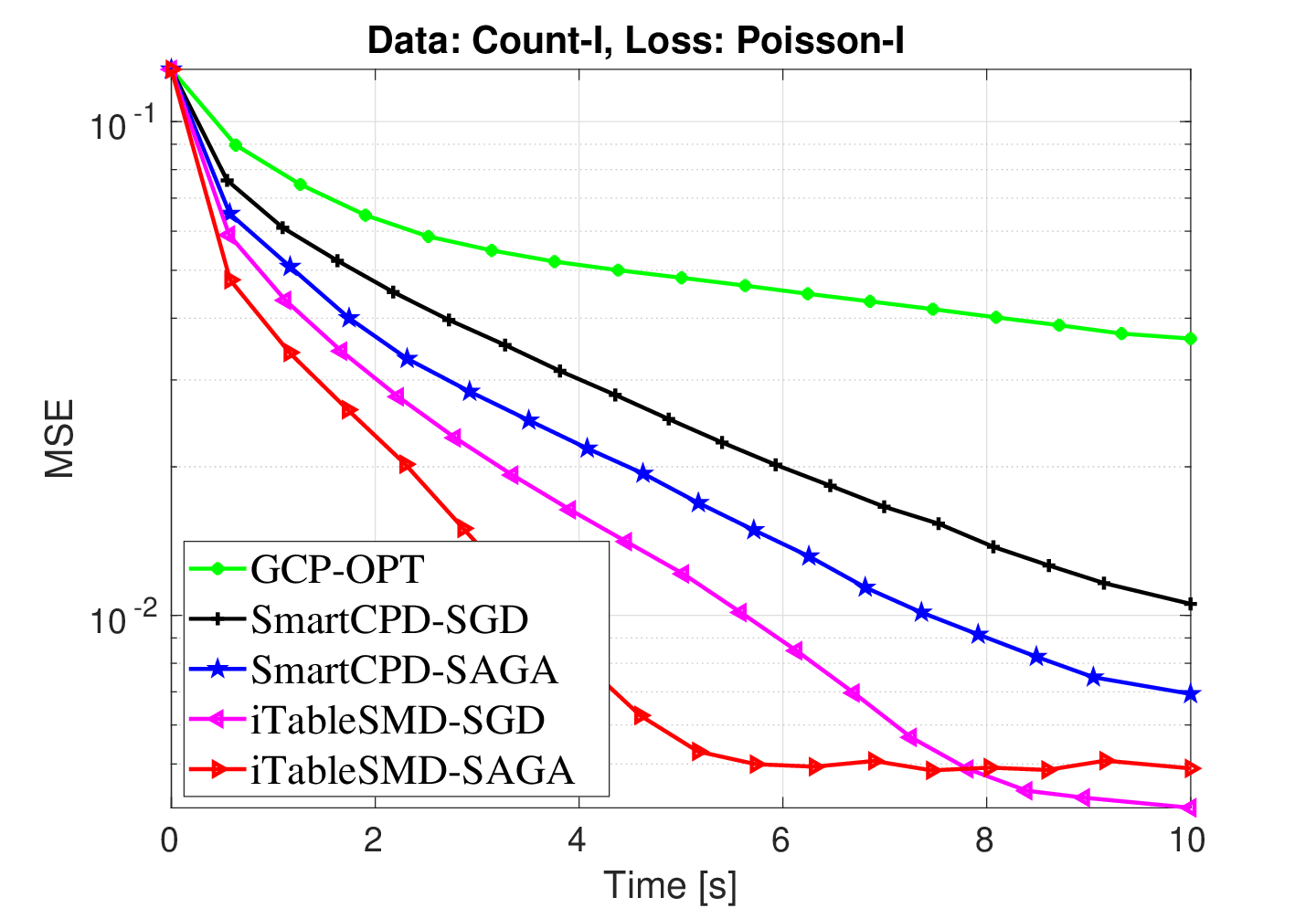}&
		\includegraphics[width=0.32\textwidth]{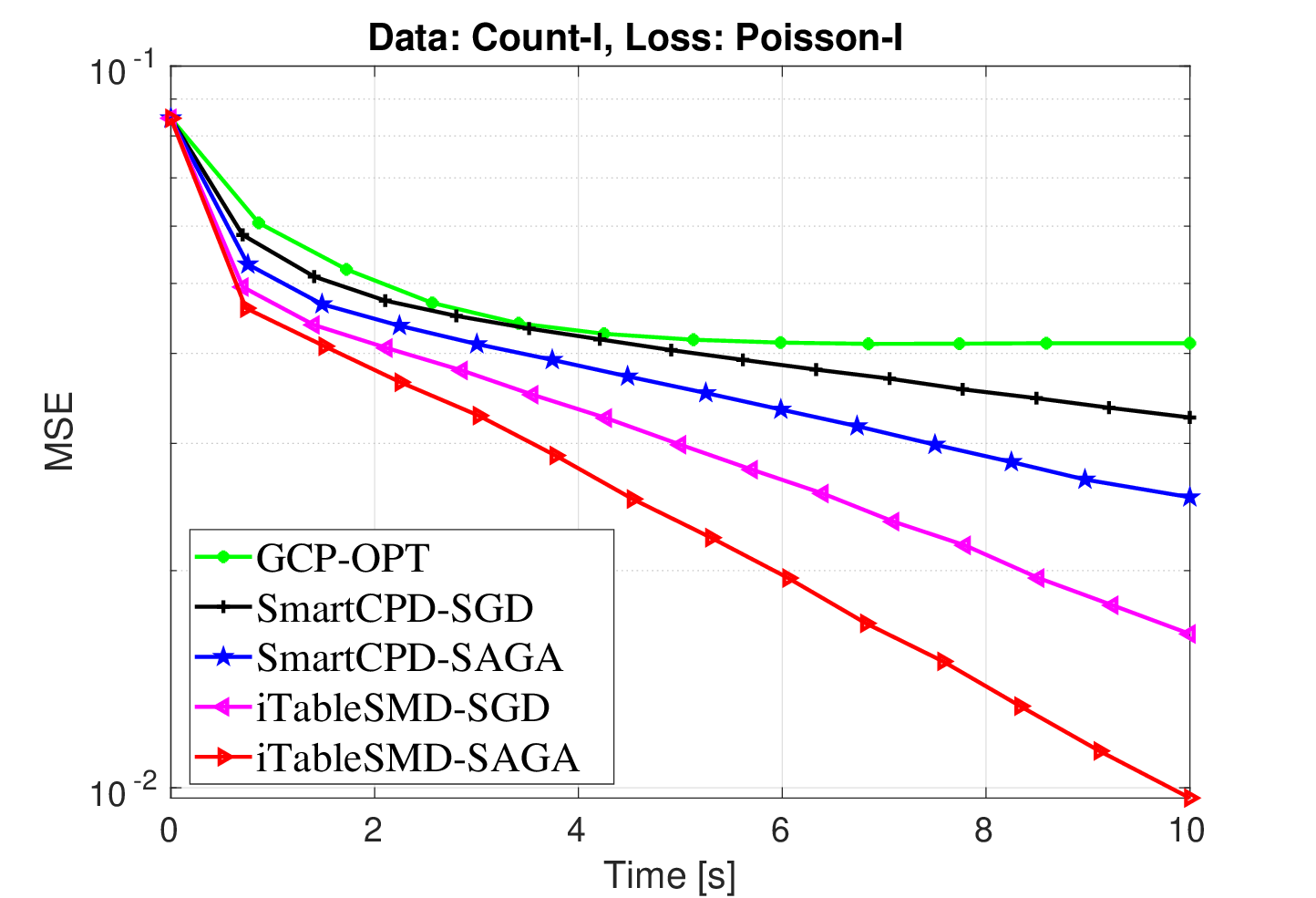}&
        \includegraphics[width=0.32\textwidth]{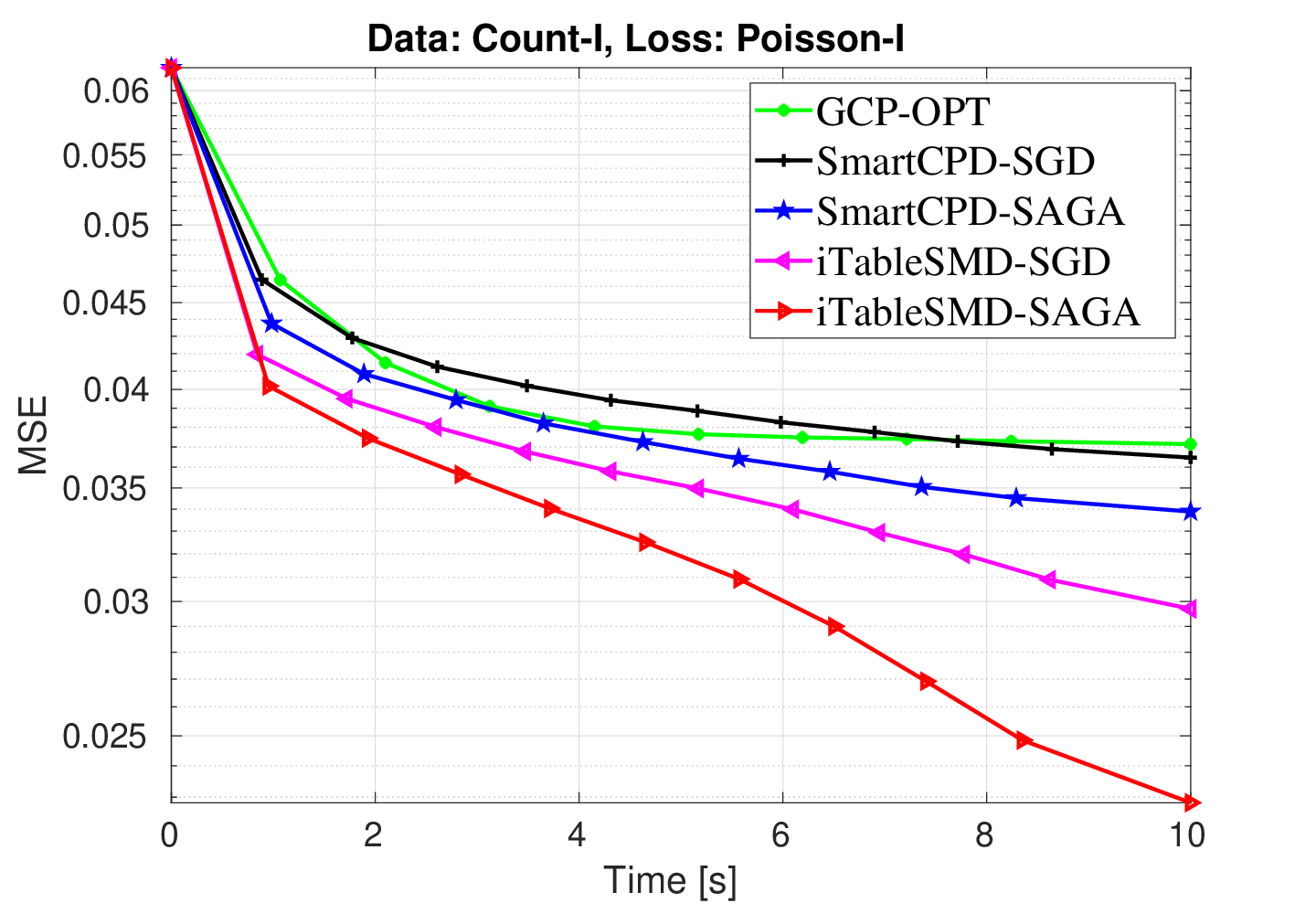}\\
		(a) $150\times51\times152$, $R=10$&(b) $150\times51\times152$,  $R=15$&(c) $150\times51\times152$, $R=20$\\
        \includegraphics[width=0.32\textwidth]{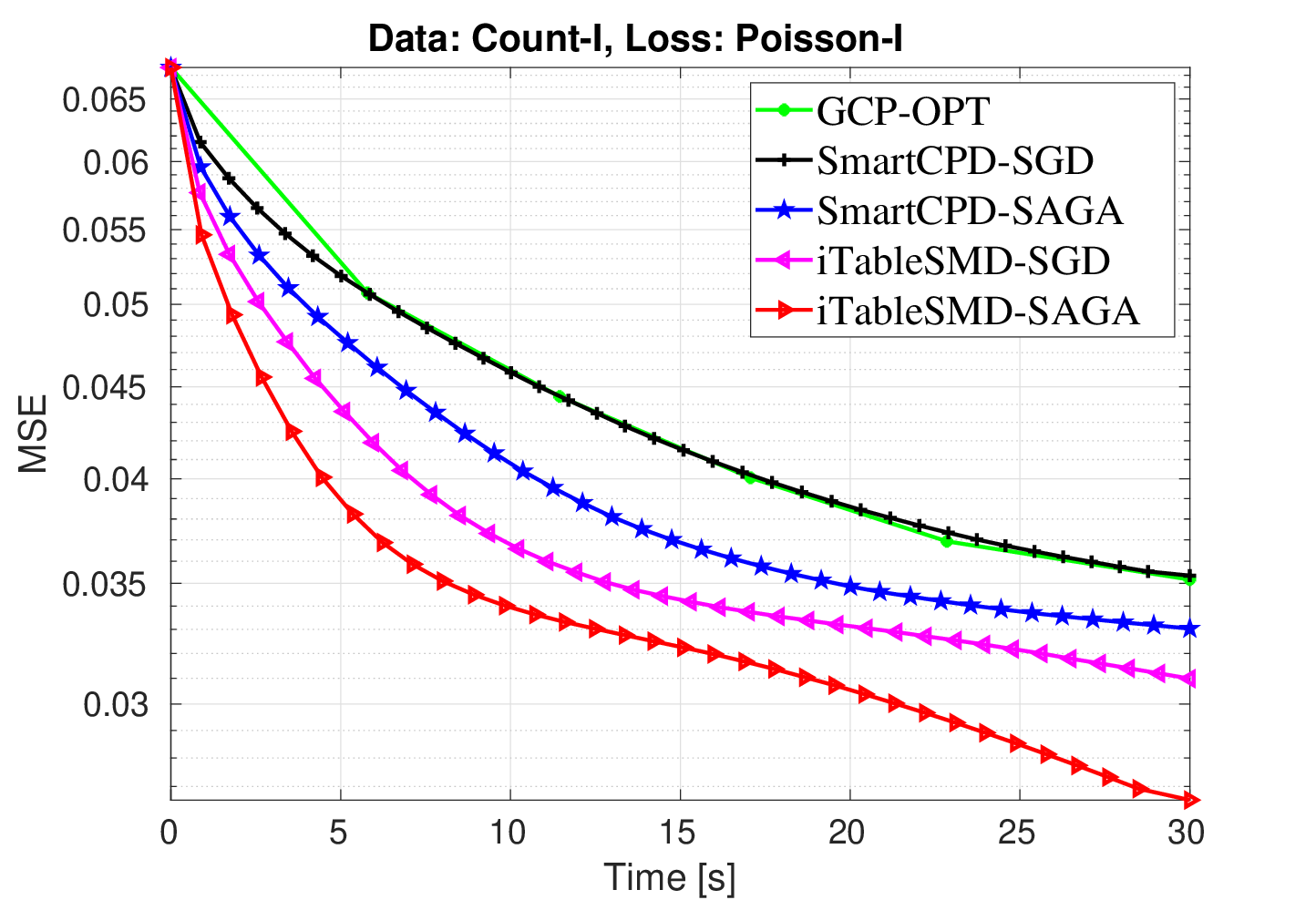}&
		\includegraphics[width=0.32\textwidth]{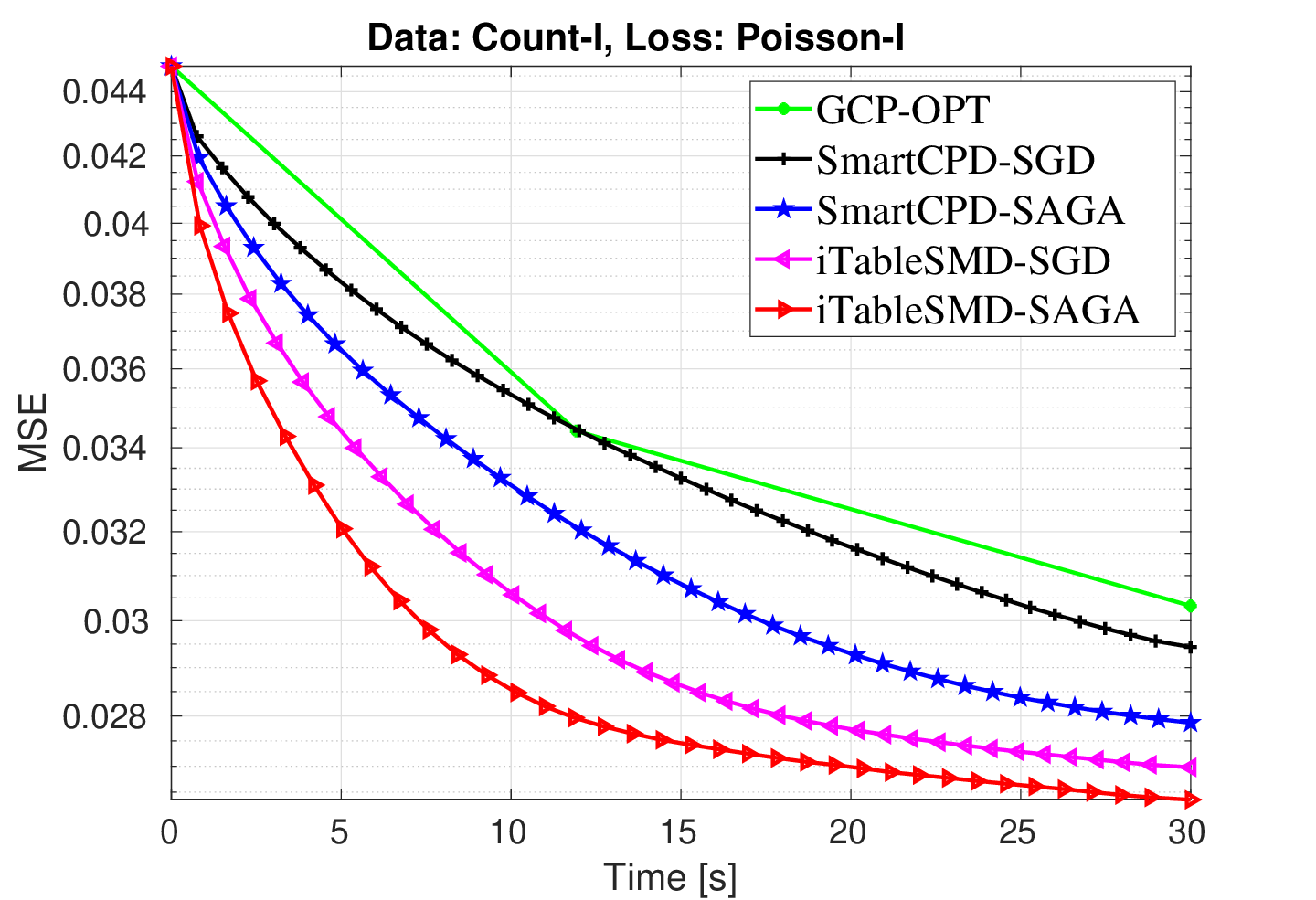}&
        \includegraphics[width=0.32\textwidth]{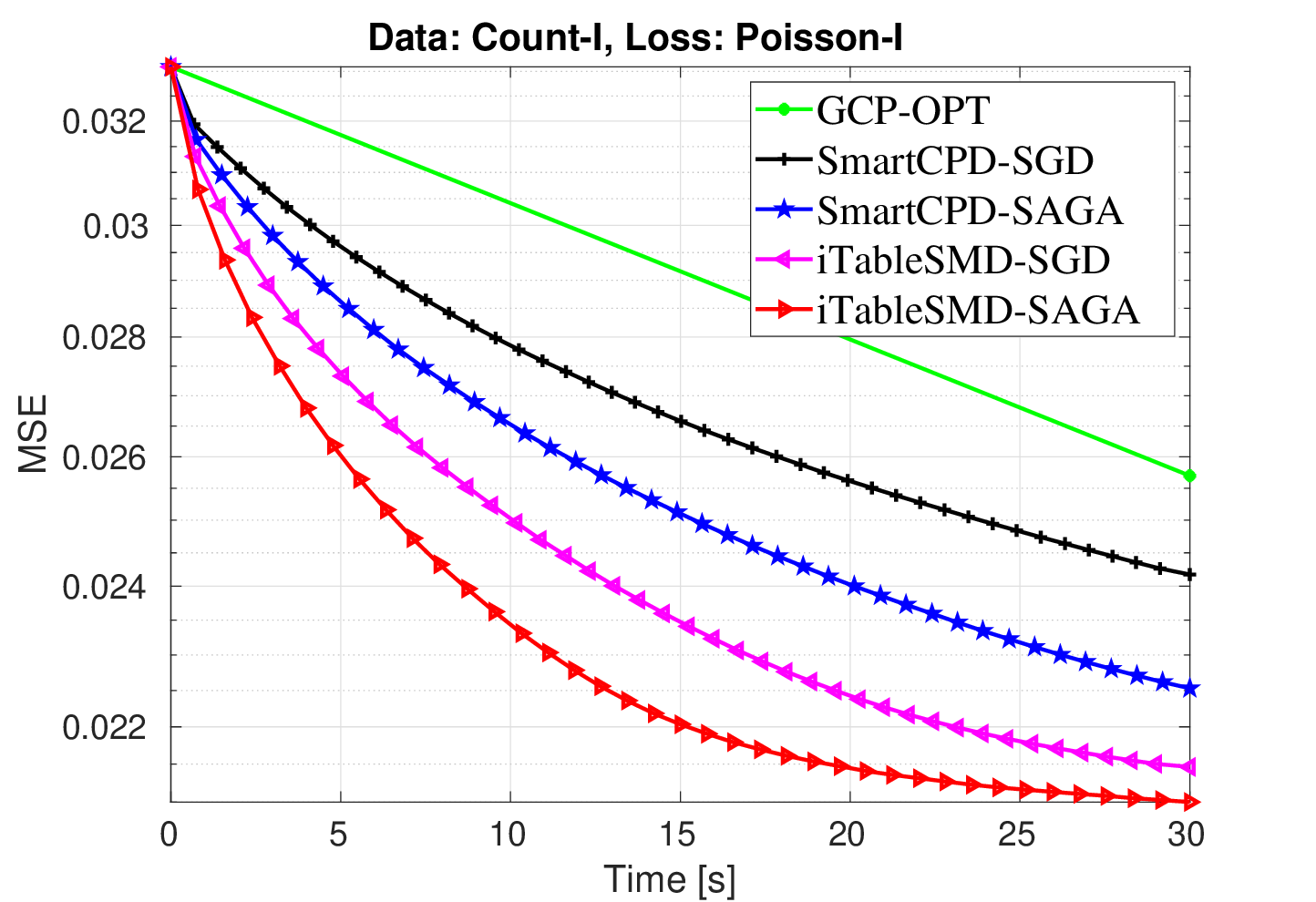}\\
		(d) $300\times300\times300$, $R=20$&(e) $300\times300\times300$,  $R=30$&(f) $300\times300\times300$, $R=40$
	\end{tabular}
	\caption{Numerical experiments for Poisson distribution on synthetic datasets.
	}
	\label{syn_poisson_exp}
\end{figure}

\subsubsection{Bernoulli distribution}
We see the results of numerical experiments on binary tensors of sizes $100\times80\times100$ and $50\times100\times200$. We set inertial parameters $\alpha^{k}=\frac{3(k-1)}{5(k+2)}$ and $\beta^{k}=\frac{4(k-1)}{5(k+2)}$ using a formula based on the iteration $k$, and chose a stepsize as $\eta^{k}=0.2$. The loss function is $f(m \,; x) = log(m+1)-x\,log(m+\epsilon)$ and each entry of the
binary tensor is generated from the Bernoulli distribution, i.e., $\underline{\mathcal{X}}_{i}=1$ with probability $\underline{\mathcal{M}}_{i}/(1+\underline{\mathcal{M}}_{i})$. Then, we focus on 
\begin{eqnarray*}
 \begin{aligned}
\min _{A_1, A_2,  A_3} &\quad  \frac{1}{I^N} \sum_{i \in \mathcal{I}}  \log (\underline{\mathcal{M}}_{i}+1)-\underline{\mathcal{X}}_{i}\log (\underline{\mathcal{M}}_{i}+\epsilon)+\sum_{n=1}^3 h_n\left(A_n\right) \\
 \text { s.t. } & \quad \underline{\mathcal{M}}_{i}=\sum_{r=1}^R \prod_{n=1}^3 {A}_n\left(i_n, r\right), \forall\, i \in \mathcal{I}, \\
\end{aligned}   
\end{eqnarray*}

In Figure \ref{syn_bernoulli_exp}(a) with tensor size $100\times80\times100$ and rank $R=5$, it is observed that all algorithms quickly reduce the MSE within the first few seconds. As the rank increases to $R=7$ in Figures \ref{syn_bernoulli_exp}(b) and (e), and $R=10$ in subfigures \ref{syn_bernoulli_exp}(c) and (f) respectively, there is a noticeable shift in the speed at which MSE decreases, with higher ranks leading to a slight slow down in convergence. In these four subfigures, we can further confirm that SmartCPD does not consistently outperform GCP-OPT and may at times be less effective. Nonetheless, iTalbeSMD-SAGA consistently shows robust performance, achieving low MSEs faster compared to the other method. When examining larger tensor sizes, as in subfigures \ref{syn_bernoulli_exp}(d)-(f), a similar pattern is evident, with all algorithms performing slower as the size and rank increase. Yet, the relative efficiency of iTalbeSMD-SAGA remains apparent, suggesting its advantage in dealing with larger and more complex data sets. 

Overall, iTalbeSMD-SGD and iTalbeSMD-SAGA are shown to reliably achieve lower MSEs more quickly across various synthetic tensor sizes and ranks, underlining its efficiency in different scenarios.

\begin{figure}[!htb]
	\setlength\tabcolsep{2pt}
	\centering
	\begin{tabular}{ccc}
		\includegraphics[width=0.32\textwidth]{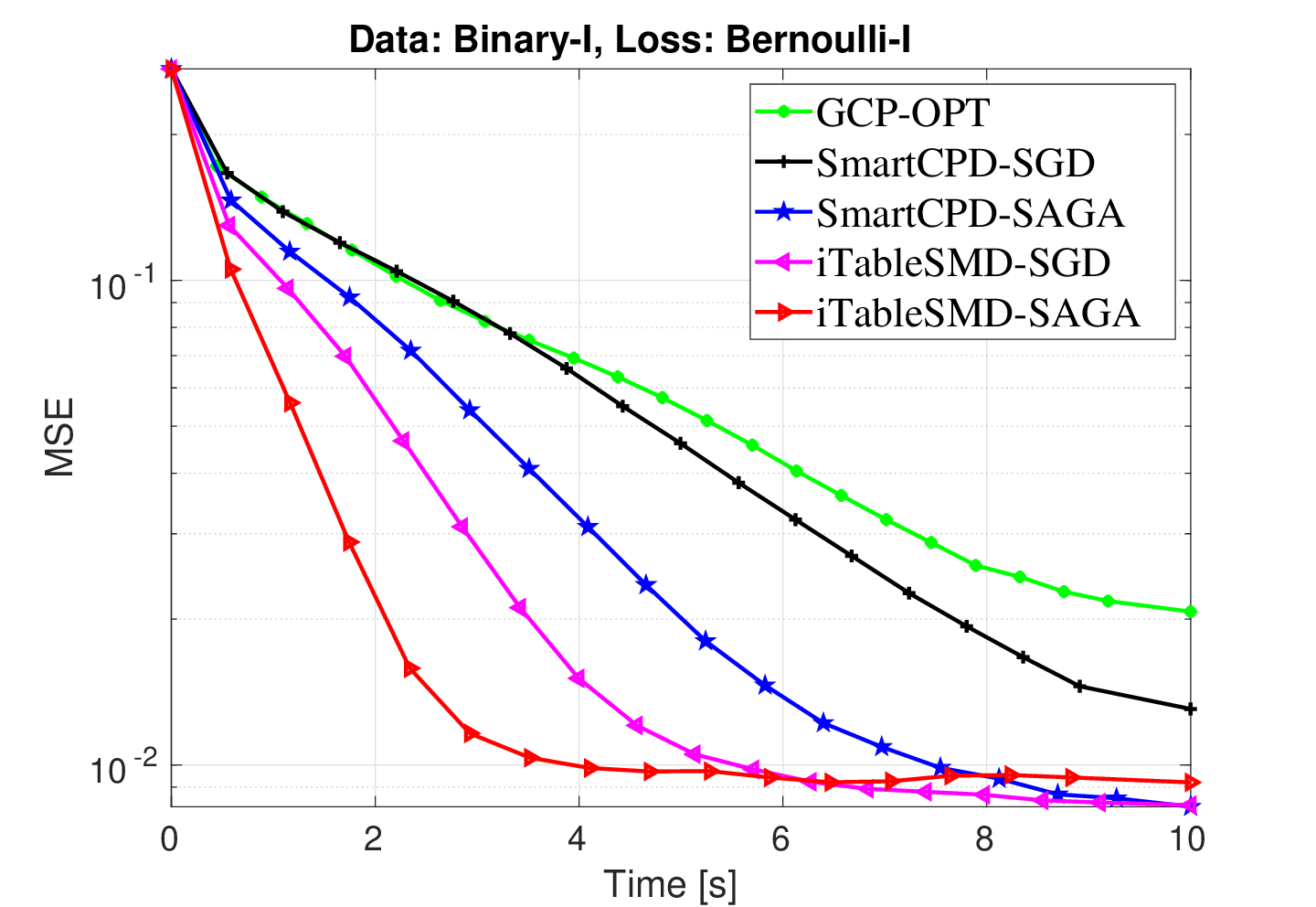}&
		\includegraphics[width=0.32\textwidth]{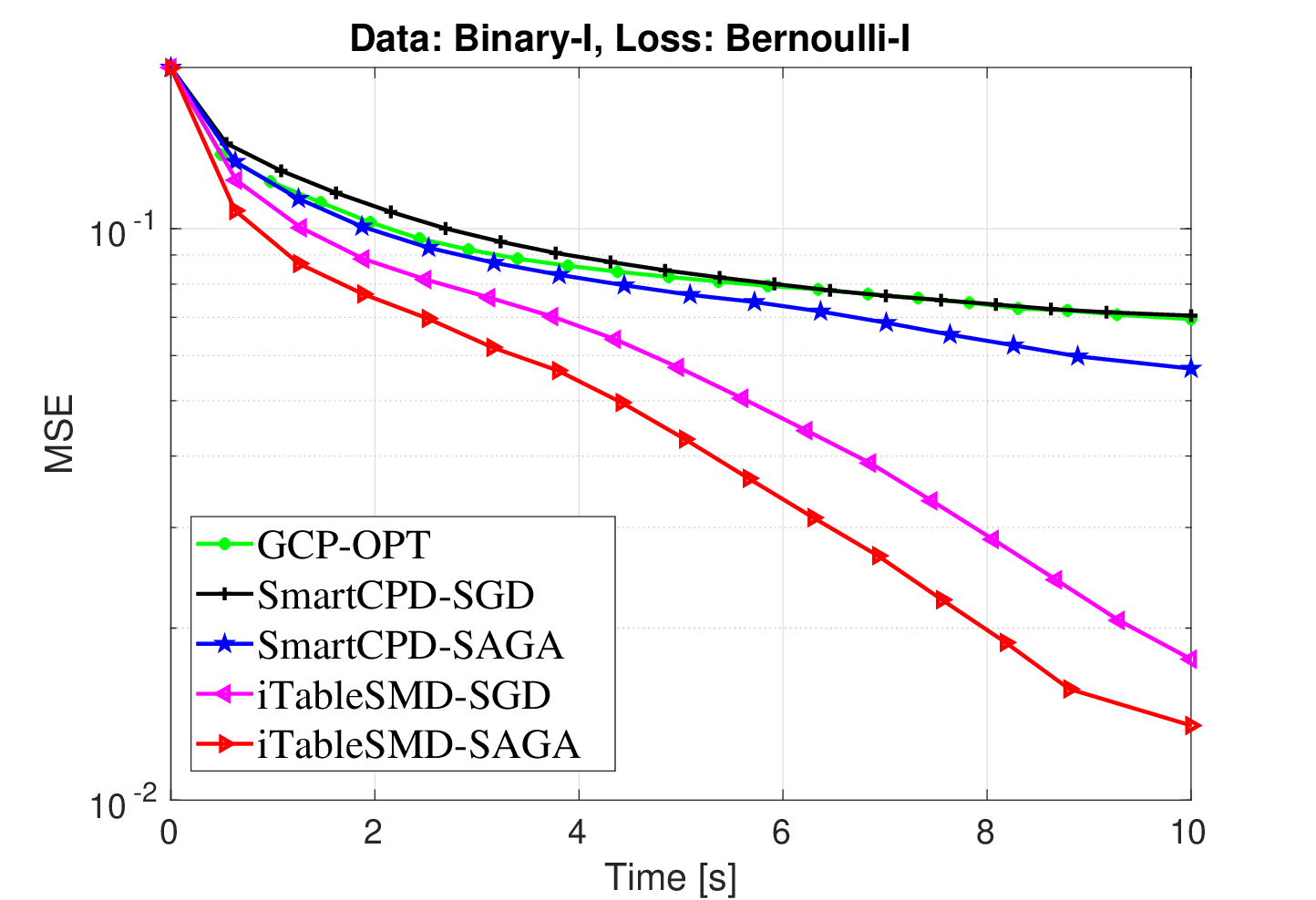}&
        \includegraphics[width=0.32\textwidth]{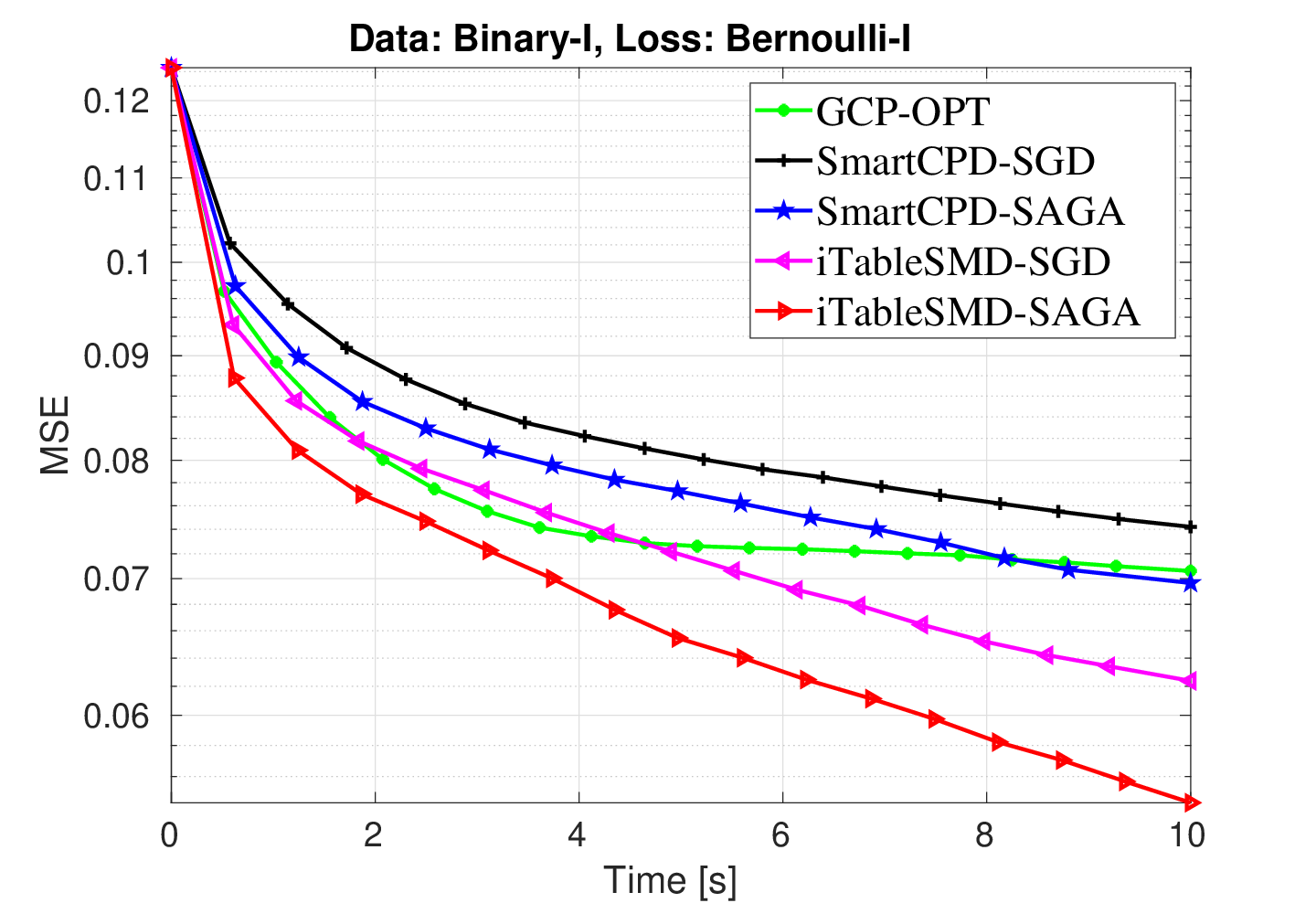}\\
		(a) $100\times80\times100$, $R=5$&(b) $100\times80\times100$, $R=7$&(c) $100\times80\times100$, $R=10$\\
        \includegraphics[width=0.32\textwidth]{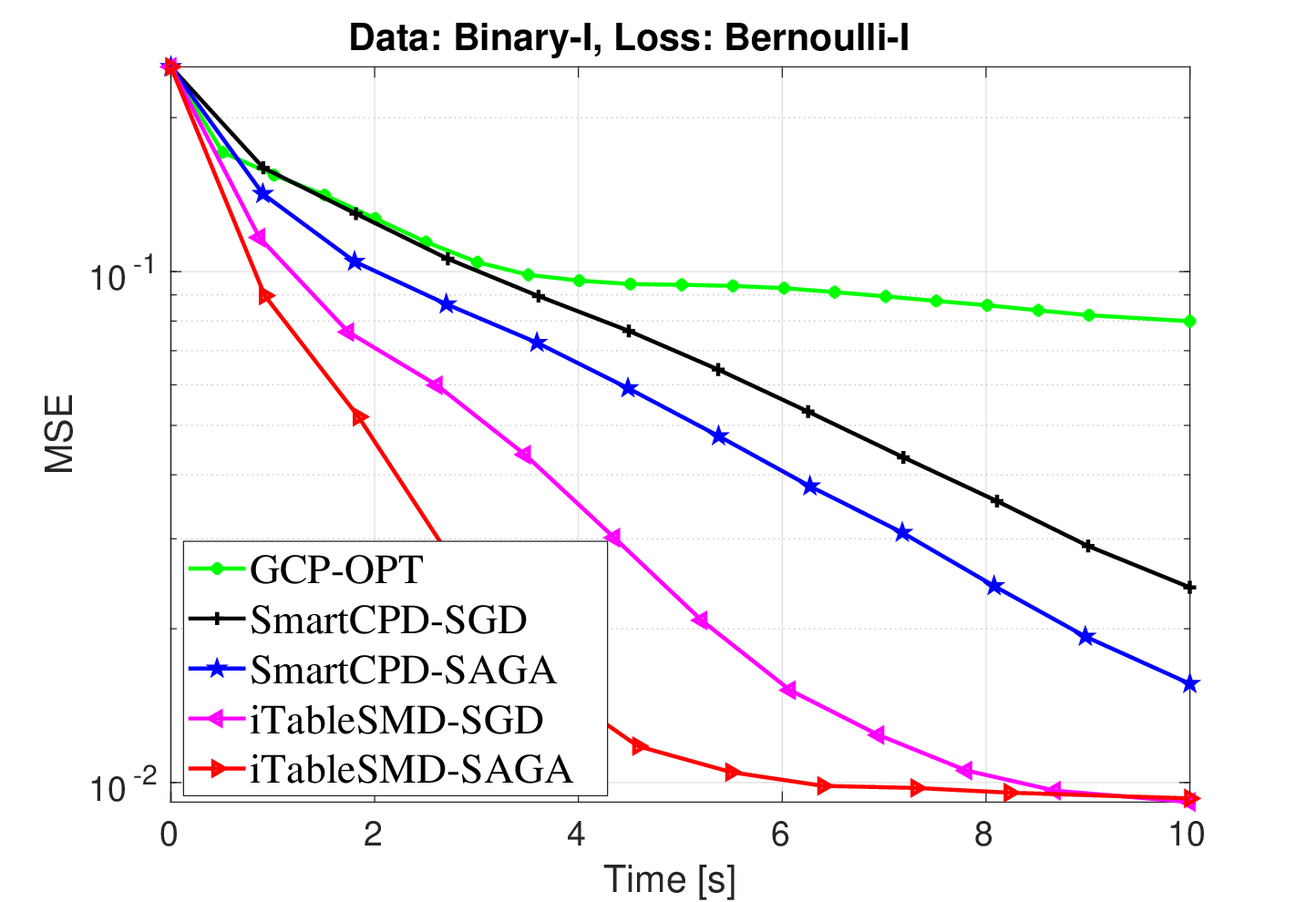}&
		\includegraphics[width=0.32\textwidth]{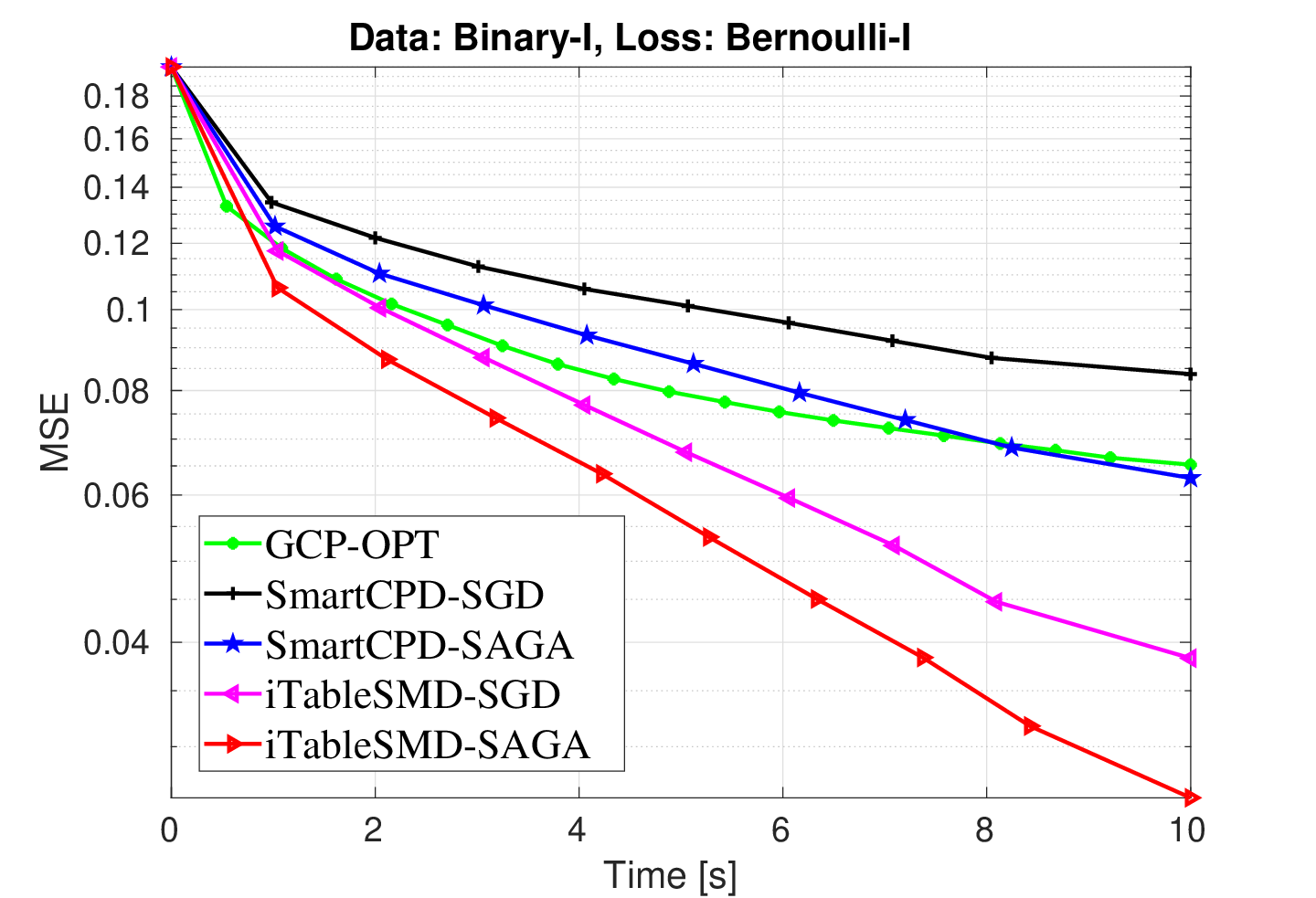}&
        \includegraphics[width=0.32\textwidth]{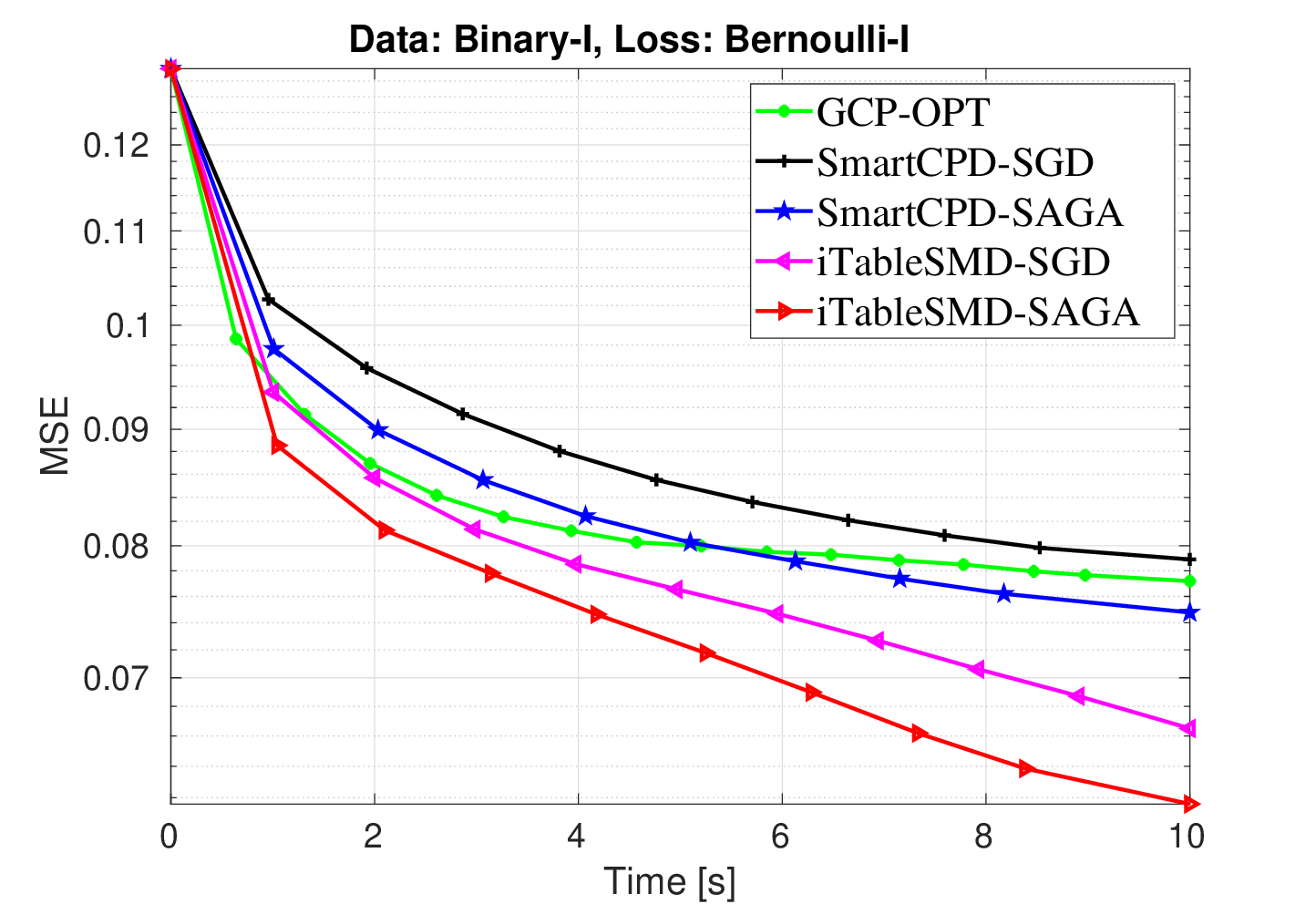}\\
		(d) $50\times100\times200$, $R=5$&(e) $50\times100\times200$, $R=7$&(f) $50\times100\times200$, $R=10$
	\end{tabular}
	\caption{Numerical experiments for Bernoulli distribution on synthetic datasets.
	}
	\label{syn_bernoulli_exp}
\end{figure}

\subsection{Real data experiments}

\subsubsection{Enron emails dataset}
    
    We apply the algorithms to the Enron emails dataset. This dataset comprises a large collection of email messages exchanged by the employees of the Enron Corporation, which was collected and prepared by the CALO Project (A Cognitive Assistant that Learns and Organizes). We use the extractive version in \cite{shetty2004enron} and select a subset involving 142 senders, 147 receivers, and 148 unique words, excluding any additional dimensions. The data is in the form of a third-order tensor ($sender \times receiver \times word$) with integer entries representing the number of words. The size of the tensor is $142\times147\times148$. It has 6581 ($\approx 0.2\%$) nonzero entries. We choose the loss function corresponding to the Poisson distribution, i.e., $f(x, m) = m-x \,log (m+\epsilon) $ and non-negativity constraints are considered for the latent matrices. In every iteration, $2R$ fibers are sampled by iTableSMD and SmartCPD and $2R \times \frac{142+147+148}{3}$ entries are sampled for GCP-OPT. We set inertial parameters $\alpha^{k}=\frac{3(k-1)}{5(k+2)}$. All algorithms under test are stopped when the relative change in the loss function is less than $10^{-10}$.

    Figure \ref{real_poisson_exp_01} presents a series of numerical experiments conducted on enron dataset with a Poisson distribution, examining the performance of various optimization algorithms with different tensor ranks under $R=3$, $R=5$, and $R=7$, respectively. Each algorithm is run for 5 trials and in each trial, the factor matrices are initialized by randomly sampling its entries from uniform distribution between 0 and 1. It is observed that iTableSMD stands out for its quick cost reduction, reaching a low cost within approximately 6 seconds, noticeably faster than the baseline methods, which require more time and yet do not achieve as low of a cost.  This quick performance is most notable at the higher rank of $R=7$, where iTableSMD quickly lower the cost apparently, surpassing other algorithms that struggle to converge within the same time.
  
\begin{figure}[!htb]
	\setlength\tabcolsep{2pt}
	\centering
	\begin{tabular}{ccc}
		\includegraphics[width=0.32\textwidth]{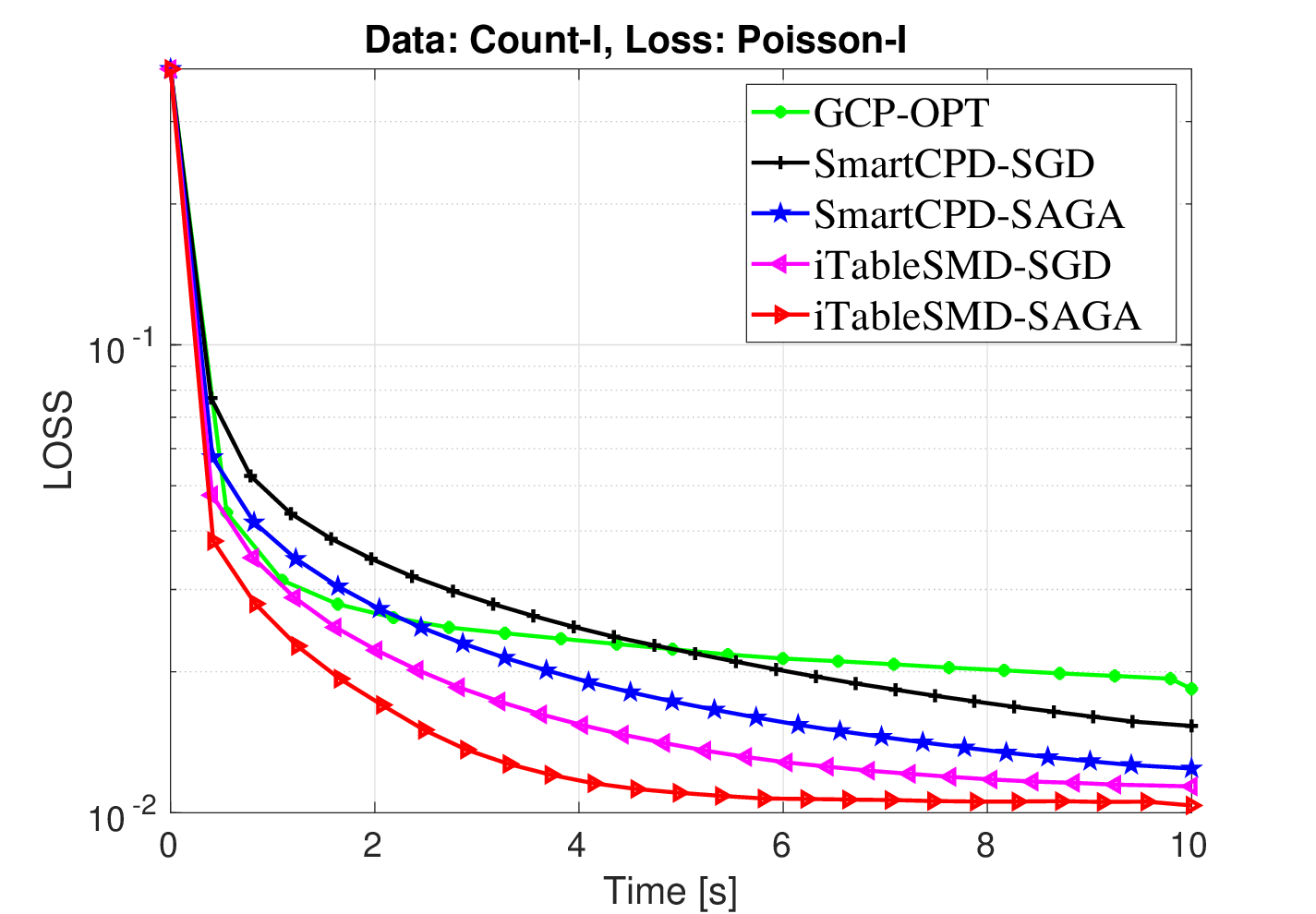}&
		\includegraphics[width=0.32\textwidth]{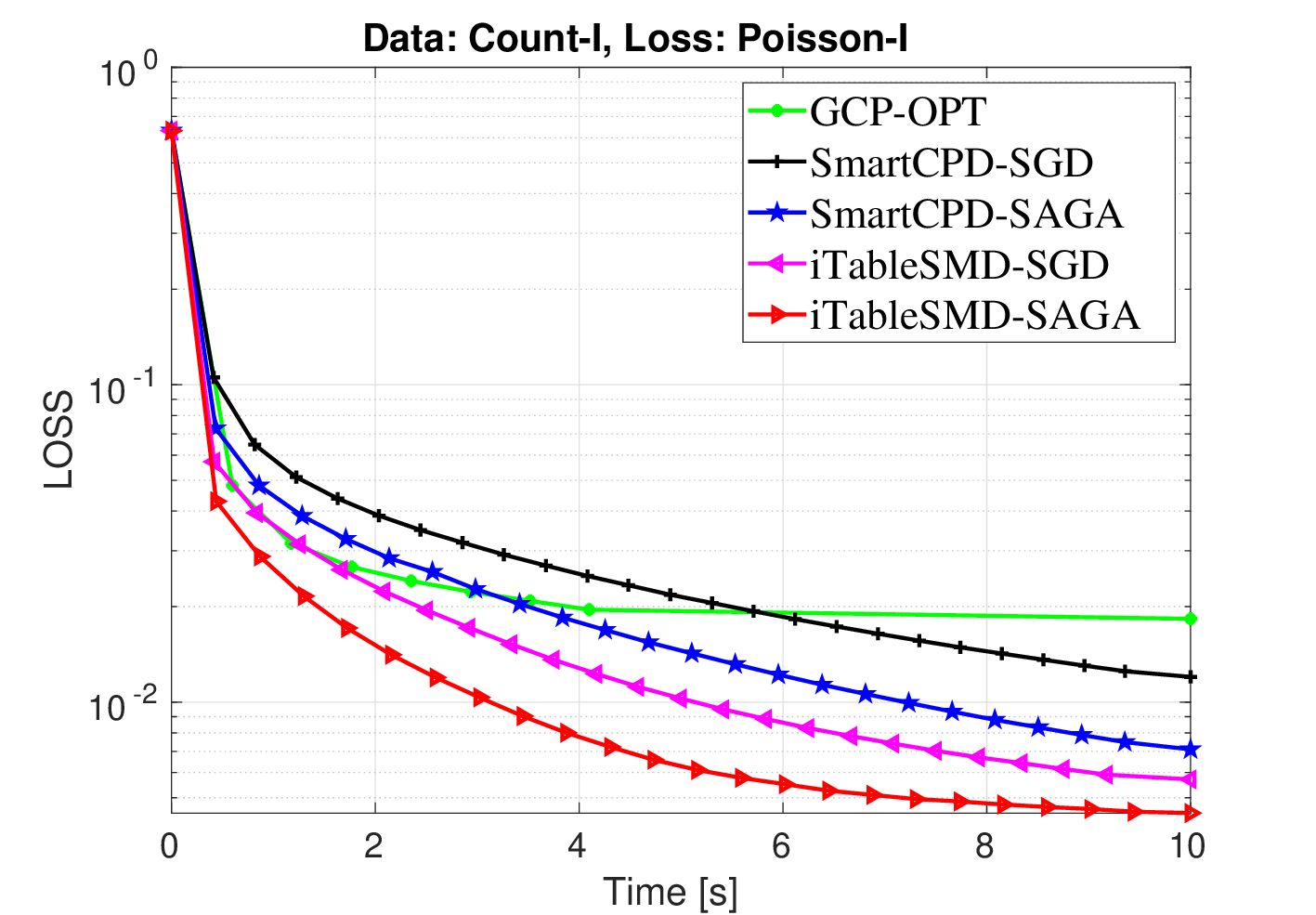}&
        \includegraphics[width=0.32\textwidth]{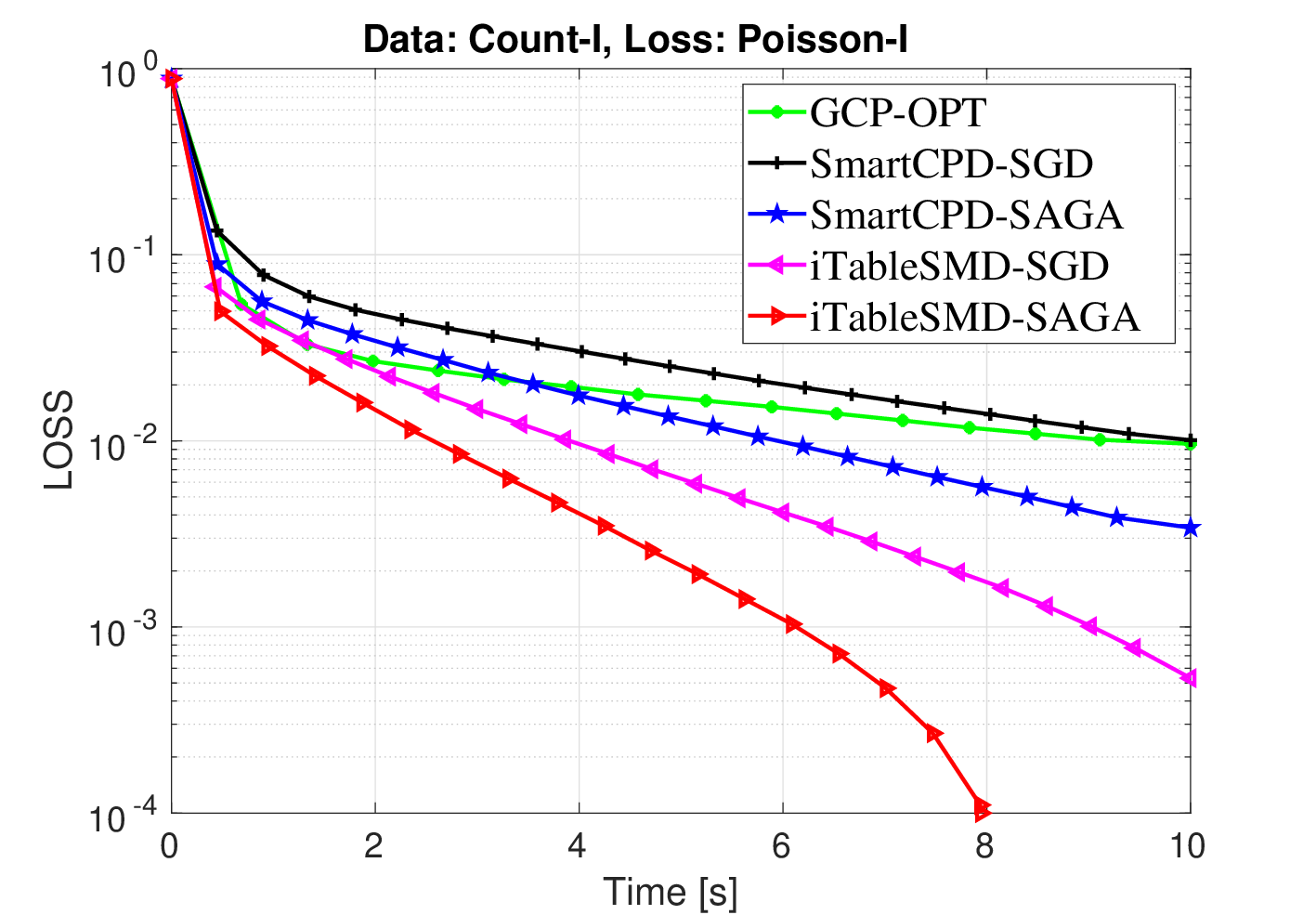}\\
		(a) $142\times147\times148$, $R=3$&(b) $142\times147\times148$, $R=5$&(c) $142\times147\times148$, $R=7$
	\end{tabular}
	\caption{Numerical experiments for Poisson distribution on Enron emails dataset.
	}
	\label{real_poisson_exp_01}
\end{figure}


 \subsubsection{The Flickr dataset}

 We evaluate the algorithms on the Flickr dataset, as referenced by Gorlitz et al. \cite{gorlitz2008}. The dataset consists of tags representing whether a user has labeled an image on a particular day, with non-zero values marked as binary indicators. We form a third-order binary tensor of size $520\times520\times520$. The chosen loss function is tailored for the Bernoulli distribution, i.e., $f(x, m) = log(m+1)-x\,log(m+\epsilon)$. Other settings and parameters are as before. 
Figure \ref{real_bernoulli_exp} shows the cost value change against time in seconds for different values of $R$. Similar to the previous datasets, the proposed iTableSMD shows considerable runtime advantages over GCP-OPT. The results reinforce the capability of iTableSMD to handle complex, real-world datasets effectively.
\begin{figure}[!htb]
	\setlength\tabcolsep{2pt}
	\centering
	\begin{tabular}{ccc}
		\includegraphics[width=0.32\textwidth]{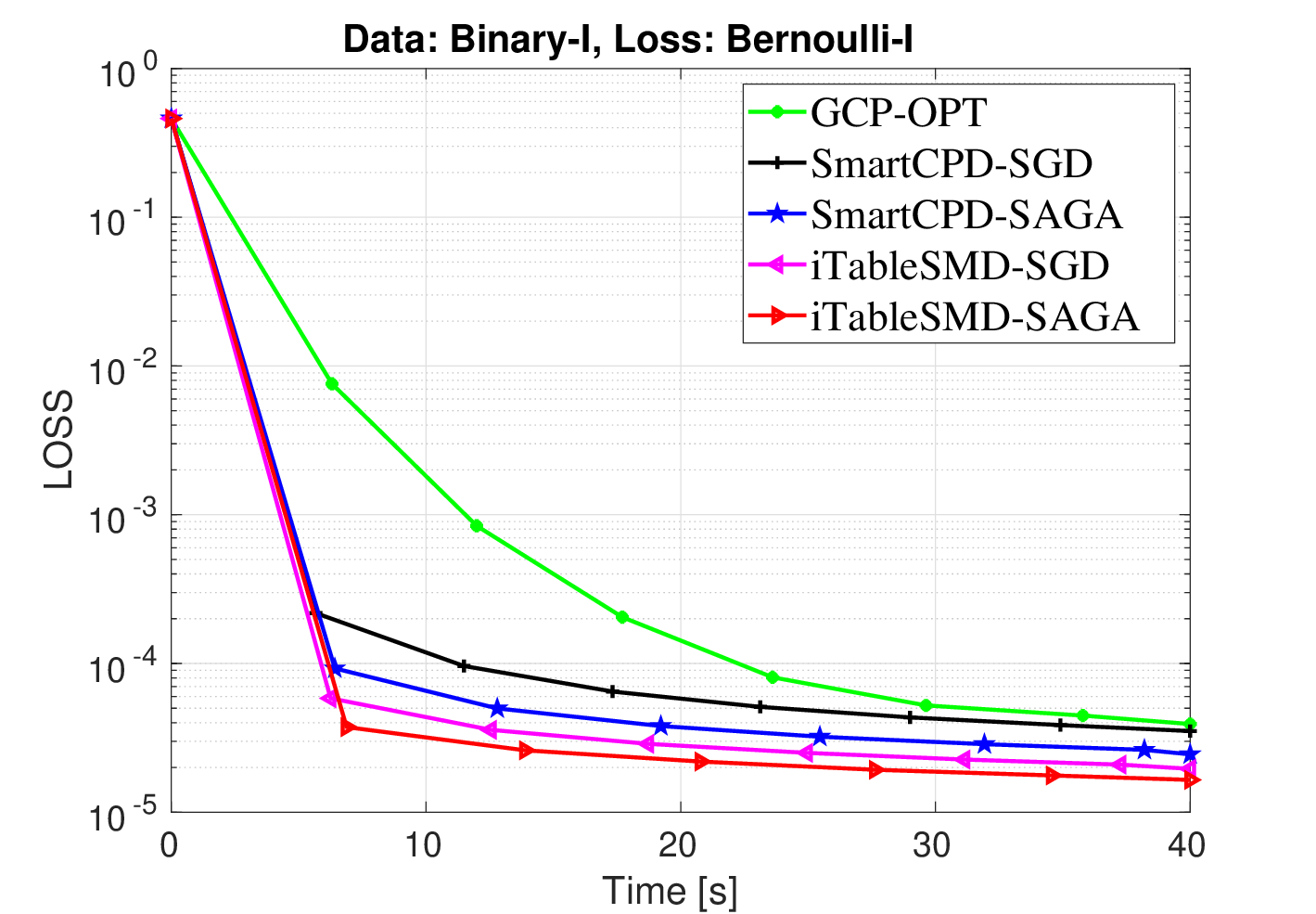}&
		\includegraphics[width=0.32\textwidth]{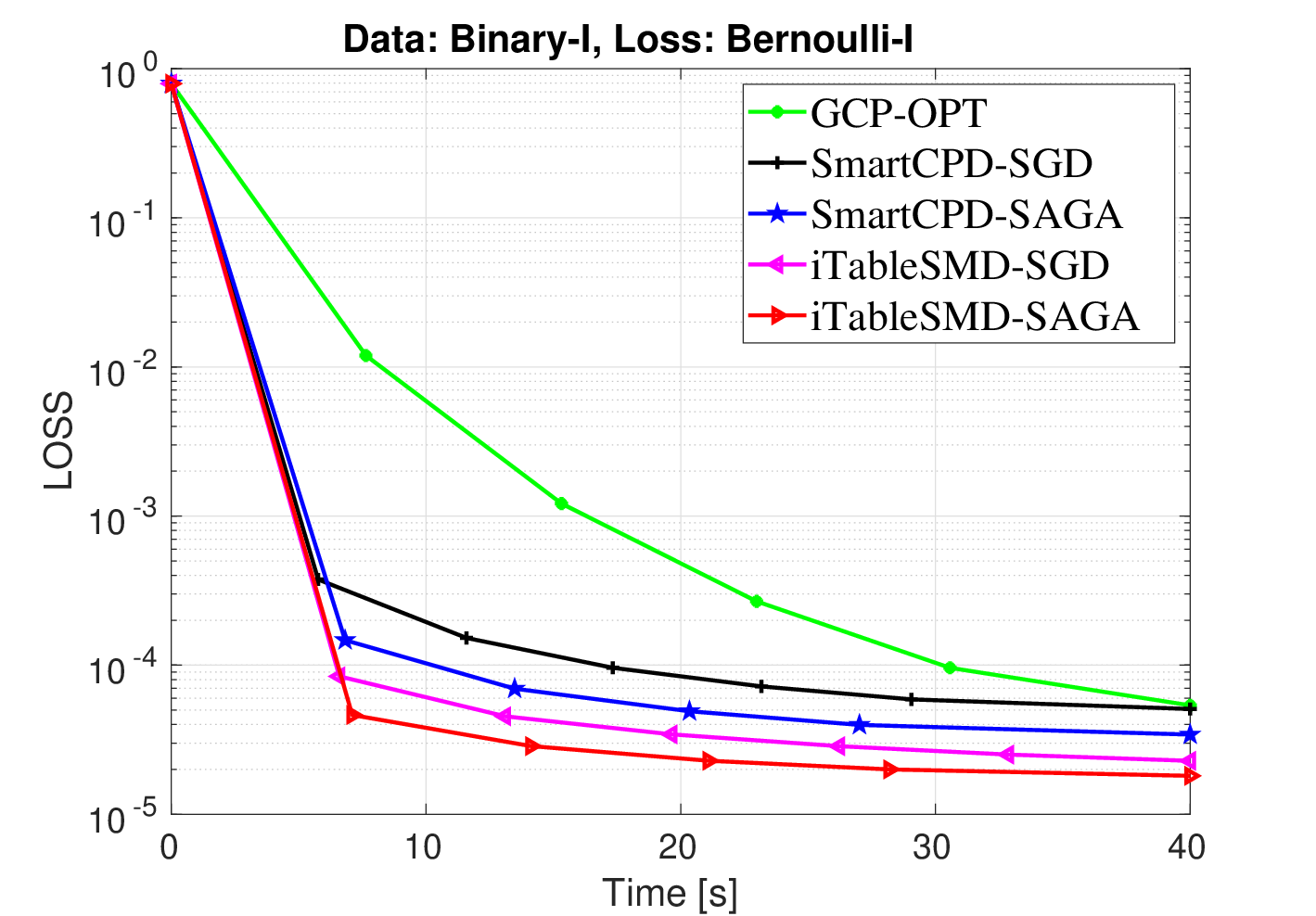}&
        \includegraphics[width=0.32\textwidth]{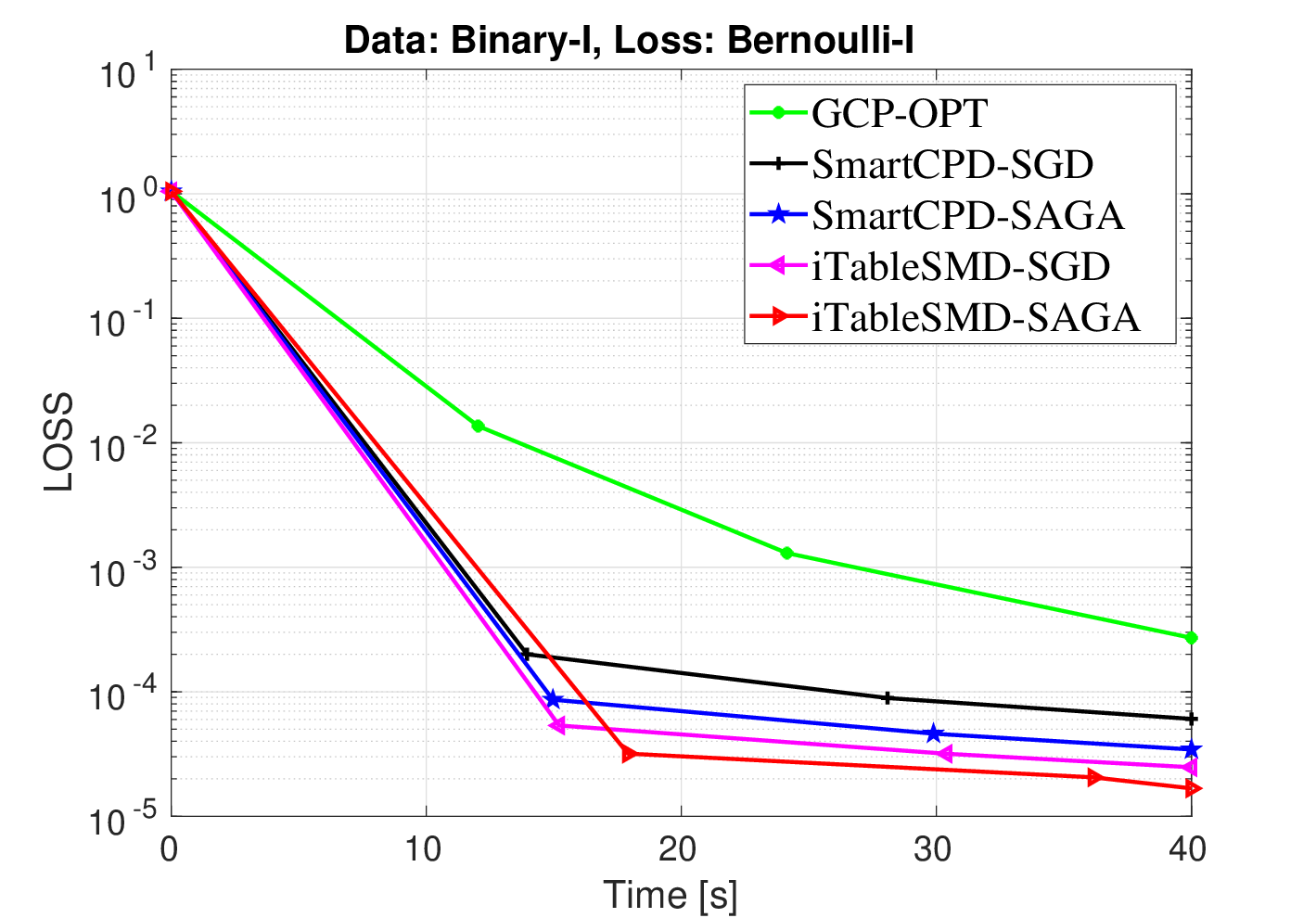}\\
		(a) $520\times520\times520$, $R=5$&(b) $520\times520\times520$, $R=10$&(c) $520\times520\times520$, $R=15$
	\end{tabular}
	\caption{Numerical experiments for Bernoulli distribution on the Flickr dataset.
	}
	\label{real_bernoulli_exp}
\end{figure}

\section{Conclusion}\label{conclusion}

In this paper, we proposed an inertial accelerated block randomized stochastic mirror descent algorithm (iTableSMD) for nonconvex multi-block objective functions beyond global Lipschitz gradient continuity. This algorithm is particularly tailored for large-scale Generalized Tensor CP (GCP) decomposition under non-Euclidean losses. By integrating a broader version of multi-block variance ruduction, we establish the sublinear convergence rate for the subsequential sequence produced
by the iTableSMD algorithm and prove it requires at most $\mathcal{O}(\varepsilon^{-2})$ iterations in expectation to attain an $\varepsilon$-stationary point. Additionally, we verify the global convergence of the sequence generated by iTableSMD. We tested the algorithm over various types of simulated and real data with several baselines, indicating significant computational efficiency improvements over existing state-of-the-art methods. These results highlight the advantages and effectiveness of incorporating an inertial accelerated stochastic approach in the algorithmic framework for GCP tensor decomposition.


\section*{Declarations}
{\bf Funding:} This research is supported by the R\&D project of Pazhou Lab (Huangpu) (Grant no. 2023K0603), the National Natural Science Foundation of China (NSFC) grant 12171021  and the Fundamental Research Funds for the Central Universities (Grant No. YWF-22-T-204).

\noindent{\bf Competing interests:} The authors have no competing interests to declare that are relevant to the content of this article.

\noindent{\bf Data  Availability Statement:} Data will be made available on reasonable request.

\bibliographystyle{abbrv}
\bibliography{Ref_iTable}

\end{document}